\documentclass[a4paper,12pt]{article}
\usepackage{amsmath,amsthm,amssymb,amscd,mathtools}
\usepackage{xfrac}
\usepackage[
            scr=boondoxo]
           {mathalpha}
\usepackage{hyperref,enumitem,cleveref}
\setlist[enumerate,1]{label=\textnormal{(\roman*)},ref=(\roman*)}
\usepackage{xspace}
\hypersetup{
    colorlinks=true,
    linkcolor=purple,
    citecolor=magenta,
    filecolor=cyan,
    urlcolor=cyan,
    pdftitle={A classification of group gradings on incidence algebras over commutative rings},
    pdfpagemode=FullScreen,
    }
\usepackage{tikz}
\usetikzlibrary{calc}
\usetikzlibrary{positioning}
\tikzset{
  treenode/.style = {align=center, inner sep=0pt, text centered,
    font=\sffamily},
 arn_n/.style = {treenode, circle, white, font=\sffamily\bfseries, draw=black,
    fill=white, text width=0.6ex},
  arn_p/.style = {treenode, circle, white, font=\sffamily\bfseries, draw=black,
    fill=white, text width=0.6ex},
}
\newif\ifelementary
\elementaryfalse
\newtheorem{theorem}{Theorem}[section]
\newtheorem{corollary}[theorem]{Corollary}
\newtheorem{lemma}[theorem]{Lemma}
\newtheorem{proposition}[theorem]{Proposition}

\theoremstyle{definition} 
\newtheorem{definition}[theorem]{Definition}
\newtheorem{example}[theorem]{Example}
\newtheorem{remark}[theorem]{Remark}

\theoremstyle{definition}
\providecommand{\subjclass}[2][2020]{\textbf{\textit{#1 Mathematics Subject Classification: }} #2}
\providecommand{\keywords}[1]{\textbf{\textit{Keywords: }} #1}

\newcommand{\highlight}[1]{\textit{#1}}

\DeclareMathOperator{\Hom}{Hom}
\DeclareMathOperator{\End}{End}

\DeclareMathOperator{\UT}{UT}
\DeclareMathOperator{\Aut}{Aut}

\DeclareMathOperator{\Exp}{Exp}

\DeclareMathOperator{\Span}{Span}

\newcommand{\mathmode}[1]{\ensuremath{#1}\xspace}

\newcommand{\aset}[1]{\mathmode{\mathcal{#1}}} 

\newcommand{\Sset}{\aset{S}}

\newcommand{\sset}{\mathmode{\star}-set\xspace} 
\newcommand{\ssets}{\mathmode{\star}-sets\xspace} 

\newcommand{\nbset}[1]{\mathmode{\mathbb{#1}}}
\newcommand{\ZZ}{\nbset{Z}}

\newcommand{\RR}{\nbset{R}}
\newcommand{\CC}{\nbset{C}}

\newcommand{\JR}[1]{\mathmode{J(#1)}} 
\newcommand{\NR}[1]{\mathmode{N(#1)}} 
\newcommand{\ZR}[1]{\mathmode{Z(#1)}} 

\DeclareMathOperator{\car}{char}
\newcommand{\omegas}[1]{\mathmode{\mathscr{#1}}}
\newcommand{\Eset}{\omegas{E}}

\newcommand{\Tset}{\omegas{T}}
\newcommand{\R}{\mathmode{\mathcal{R}}}   
\newcommand{\dR}{an indecomposable commutative ring with 1\xspace} 
\newcommand{\LR}{\mathmode{\mathcal{R}}}  
\newcommand{\dLR}{a commutative ring with $1$\xspace} %

\newcommand{\uR}{\mathmode{\R^{\times}}} 

\newcommand{\iM}{\mathmode{\mathfrak{m}}}

\newcommand{\F}{\mathmode{\mathbb{F}}}

\newcommand{\poset}[1]{\mathmode{\mathscr{#1}}}
\newcommand{\PO}{\poset{P}} 
\newcommand{\QO}{\poset{Q}} 
\newcommand{\RO}{\poset{R}} 
\newcommand{\CO}{\poset{C}} 
\newcommand{\dPO}{a finite partially ordered set\xspace} 
\newcommand{\dLPO}{a locally finite partially ordered set\xspace} 

\newcommand{\oP}{\mathmode{{|\PO|}}}

\newcommand{\alg}[1]{\mathmode{\mathscr{#1}}} 
\newcommand{\Aalg}{\alg{A}}
\newcommand{\Balg}{\alg{B}}

\newcommand{\Dalg}{\alg{D}} 

\newcommand{\modu}[1]{\mathmode{\mathcal{#1}}} 
\newcommand{\mM}{\modu{M}}
\newcommand{\mN}{\modu{N}}

\newcommand{\Inc}[2]{\mathmode{{I(#1,#2)}}} 
\newcommand{\IPR}{\mathmode{\Inc{\PO}{\R}}}  

\newcommand{\jIPR}{\mathmode{\JR{\Inc{\PO}{\R}}}}  
\newcommand{\zIPR}{\mathmode{\ZR{\Inc{\PO}{\R}}}}  
\newcommand{\nIPR}{\mathmode{\NR{\Inc{\PO}{\R}}}}  
\newcommand{\IQR}{\mathmode{\Inc{\QO}{\R}}}  
\newcommand{\IRR}{\mathmode{\Inc{\RO}{\R}}}  

\newcommand{\idemp}[1]{\mathmode{\mathscr{#1}}} 
\newcommand{\bidemp}[1]{\mathmode{\mathbf{#1}}} 
\newcommand{\ee}{\idemp{e}}
\newcommand{\ff}{\idemp{f}}
\newcommand{\cc}{\idemp{c}}
\newcommand{\dd}{\idemp{d}}
\newcommand{\hh}{\idemp{h}}
\newcommand{\bee}{\bidemp{e}}

\newcommand{\kk}{\idemp{k}}
\newcommand{\ww}{\mathmode{w}} 

\DeclareMathOperator{\supf}{Supp} 
\DeclareMathOperator{\minsup}{min\,Supp} 
\newcommand{\Hsup}[2]{\H_{#1}^{#2}}

\newcommand{\phisup}[2]{\phi_{#1}^{#2}}

\newcommand{\G}{\mathmode{G}}
\renewcommand{\H}{\mathmode{H}} 
\newcommand{\K}{\mathmode{K}}
\newcommand{\U}{\mathmode{U}}
\renewcommand{\L}{\mathmode{L}}
\renewcommand{\S}{\mathmode{S}}
\newcommand{\gL}{\mathmode{L}}
\newcommand{\gA}[2]{\mathmode{#1#2}} 
\newcommand{\RG}{\gA{\R}{\G}}
\newcommand{\FG}{\gA{\F}{\G}}
\newcommand{\RH}{\gA{\R}{\H}}
\newcommand{\RK}{\gA{\R}{\K}}
\newcommand{\RL}{\gA{\R}{\gL}}
\newcommand{\hG}{\mathmode{{\widehat{\G}}}} 
\newcommand{\hH}{\mathmode{{\widehat{\H}}}}
\newcommand{\hK}{\mathmode{{\widehat{\K}}}}
\newcommand{\hL}{\mathmode{{\widehat{\L}}}}
\newcommand{\hS}{\mathmode{{\widehat{\S}}}}
\newcommand\oG{\mathmode{{|\G|}}}
\newcommand\oH{\mathmode{{|\H|}}}
\newcommand\oK{\mathmode{{|\K|}}}
\newcommand{\Q}{\mathmode{Q}}            
\newcommand{\RQ}{\gA{\R}{\Q}}
\newcommand{\hQ}{\mathmode{{\widehat{\Q}}}}
\newcommand{\oQ}{\mathmode{{|\Q|}}}
\newcommand{\mA}{\modu{A}}               
\newcommand{\mP}{\modu{P}}
\newcommand{\gr}{\mathrm{gr}}
\newcommand{\Homgr}{\Hom^{\gr}}
\newcommand{\Ext}[2]{\mathmode{E(#1,#2)}} 
\newcommand{\Types}{\mathmode{\Tset_{\H,\K}(\G)}}

\title{A classification of group gradings on incidence algebras over commutative rings}
\author{Felipe de Mattos Chafik Hindi\footnote{Department of Mathematics, UFSCar, Brazil, \texttt{felipehindi@gmail.com}}\\
Waldeck Sch\"utzer\footnote{Department of Mathematics, UFSCar, Brazil, \texttt{waldeck@dm.ufscar.br}} }
\date{September 17, 2026}
\begin{document}
\maketitle

\begin{abstract}
Let \LR be \dLR, \PO \dLPO, and \G a group. We derive necessary and sufficient conditions for an \R-algebra isomorphism between the incidence algebra \IPR and the group algebra \RG. Then, for an indecomposable ring \R, a finite poset \PO and an arbitrary group \G, we classify the \G-gradings of \IPR up to graded isomorphism. The classification rests on a complete set of primitive orthogonal homogeneous idempotents. The corner algebras are split group algebras of finite abelian subgroups of \G, and the off-diagonal Peirce blocks are multiplicity-free sums of bimodules induced from characters of double coset stabilizers. Graded isomorphisms are shown to have a rigid form, and a grading is determined up to graded isomorphism by the poset of idempotents, the corner groups, the types of the atomic bimodules and the structure constants of their multiplication. The data which occur are characterized by polynomial conditions, and over an algebraically closed field of characteristic zero only finitely many graded isomorphism classes share given partial invariants. An example shows that the structure constants cannot be omitted. Some previous results are extended and enhanced, while providing alternative proofs for some known facts.
\end{abstract}


%
%
%
\section{Introduction}\label{sec:intro}
Incidence algebras of partially ordered sets and their generalizations
have attracted considerable attention ever since their
inception. This became evident shortly
after the series of seminal papers of Giancarlo Rota and his collaborators \cite{Rota}. Although the initial interest was mainly in the possible applications in combinatorics,
it has become evident that these objects are interesting in their own right.
These algebras can be seen as analogs to group algebras,
since both belong to the larger class known as
category algebras, in which the role of the group or the partially
ordered set is played by an arbitrary (locally finite, small) category.
The structural theory of both types of algebras is well developed, particularly in regards to the question of whether
any two algebras pertaining to one of these classes are isomorphic.
Although the main object of study here are the gradings on incidence algebras, it seemed quite natural to us to ask whether an incidence algebra could
be isomorphic to a group algebra. The question arises naturally
in our study and it seemed sufficiently interesting to give it some attention.
As a result, necessary and sufficient conditions for the existence of the
isomorphism were obtained. As an application,
we completely classified the group gradings on
incidence algebras up to graded isomorphism when \R is \dR and \PO is finite, no restriction being imposed on the group \G. The classification is \Cref{thm:classification}: a complete invariant is attached to each grading, two gradings are isomorphic precisely when their invariants agree, and \Cref{prop:incidence:condition,prop:incidence:explicit} decide which values of the invariant occur. The main results along the way are \Cref{thm:ipr:iso:comm:fin:grpalg,thm:xmin,thm:biregular:graph,thm:graded:peirce,cor:graded:radicals}, the structure theorem for the bimodules $\ee\IPR\ff$ (\Cref{teo:structure:bimod}) and the description of graded isomorphisms (\Cref{thm:graded:iso:structure}).

We acknowledge that this work draws from previous works, particularly \cite{Santulo-Souza-Yasumura} and \cite{Miller-Spiegel}, but also that our techniques and results differ significantly. In particular, we rely mostly on ring-theoretic concepts and facts, including the Galois theory on commutative rings, we impose no conditions on the group \G, and the only restriction on the unital commutative ring \R is indecomposability. We hope that the results contained here will greatly enhance our understanding of graded incidence algebras and their applications.

The paper is organized as follows. \Cref{sec:prelim} fixes notation and collects elementary facts about incidence algebras, their idempotents and split group algebras. \Cref{sec:groupalg} determines when an incidence algebra is isomorphic to a group algebra (\Cref{thm:ipr:iso:comm:fin:grpalg}). The resulting description of the isomorphism by characters is the model for what follows. \Cref{sec:xmin} contains the basic local result, \Cref{thm:xmin}: an extremal element $x$ of \PO yields a primitive homogeneous idempotent $\ee_x$ whose corner subalgebra is the split group algebra $\RH_x$ of a finite abelian subgroup $\H_x$ of \G. Borrowing terminology from \cite[Definition 1]{Miller-Spiegel}, in \Cref{sec:stars} these idempotents are assembled into a \sset primitive homogeneous pairwise orthogonal idempotents, which is unique up to graded inner automorphism and is partially ordered by the nonzero Peirce blocks. \Cref{sec:peirce} derives the graded triangular Peirce decomposition of \IPR, with consequences for good gradings, gradings by torsion-free groups, gradings on $\UT_n\R$ and the graded radicals. \Cref{sec:bimodules} describes the off-diagonal Peirce blocks in terms of double cosets and characters of stabilizers, and computes the tensor product of two such blocks (\Cref{prop:mackey}). \Cref{sec:isos} shows that graded isomorphisms have a rigid form (\Cref{thm:graded:iso:structure}). \Cref{sec:classification} identifies a complete invariant of a graded incidence algebra, namely the poset of a \sset, the corner groups, the types of the atomic bimodules and the structure constants of their multiplication, and characterizes the data which arise from incidence algebras. The two together classify the \G-graded incidence algebras of finite posets over \R up to graded isomorphism (\Cref{thm:classification,prop:incidence:condition}). An example shows that the structure constants cannot be dispensed with, and a rigidity theorem (\Cref{thm:rigidity}) shows that, over an algebraically closed field of characteristic zero, they take only finitely many values up to equivalence.

%
%
%

\section{Preliminaries}\label{sec:prelim}
We begin by recalling that a \highlight{locally finite partially ordered set} \PO is one whose
closed intervals $[x,y]=\{t\in \PO \mid x\leq t\leq y\}$ are finite. We say that \PO is \highlight{bounded} if it contains no infinite chain (a totally ordered subset) and \highlight{doubly infinite} if it admits an embedding of the totally ordered set $\mathbb{Z}$ (with the natural order). The partially ordered sets $\mathbb{Z}_{+}$, $\mathbb{Z}_{-}$,
and $\dot{\cup}_{n\geq 1}\CO_{n}$ (infinite disjoint union of increasingly
longer chains) are examples of unbounded partially ordered sets which
are not doubly infinite. A subset \QO of \PO is an \highlight{antichain}
if the elements in \QO are pairwise unrelated in \PO.

If \LR\ is any ring, the \highlight{incidence algebra} \IPR\ of \PO\ is the set
\[
\IPR = \{f:\PO\times \PO\rightarrow \LR\,|\,f(x,y)\neq 0\implies x\leq y\},
\]
together with the operations
\begin{align*}
(rf)(x,y) &= rf(x,y),   \\
(f+g)(x,y) &= f(x,y) + g(x,y), \\
(fg)(x,y) &= \sum_{z\in[x,y]}f(x,z)g(z,y), \\
\end{align*}
for all $f,g\in \IPR$, $r\in\LR$, and $x,y,z\in \PO$. When \R is a ring with 1, for $x,y\in\PO$, one may consider the
element
$\ee_{xy}$ defined as $\ee_{xy}(u,v)\neq 0\implies (u,v)=(x,y)$, $x\leq y$,  and $\ee_{xy}(x,y)=1$. For $x\nleq y$ we set $\ee_{xy}=0$. These elements are such  that $\ee_{xy}\ee_{yz} = \ee_{xz}$, if $x\leq y\leq z$, and $\ee_{xy}\ee_{uv}=0$ if $y\neq u$. The incidence algebra \IPR\ becomes unital, the unity element $1$ being defined as $1(x,y)\neq 0\implies x=y$ and $1(x,x)=1$ for all $x\in\PO$.
The following identity is easily verified in \IPR:
\[
\ee_{xy} f \ee_{zw} = f(y,z) \ee_{xw},
\]
for all $x,y,z,w\in\PO$. In particular, $\ee_{xx}f\ee_{yy}=f(x,y)\ee_{xy}$.
It can be shown that
every $f\in\IPR$ is uniquely expressible as a (possibly infinite)
sum of the form:
\[
   f = \sum_{(x,y)\in\PO\times\PO} f(x,y)\ee_{xy}.
\]
This is possible since \IPR is a complete Hausdorff topological algebra with the finite topology
inherited from $\R^{\PO\times \PO}$, where \R is usually considered as a topological ring with the discrete topology.

It follows from Szpilrajn's Lemma that
\PO can be embedded in a totally ordered set. In particular, when $\PO=\{x_1,\ldots,x_n\}$ is finite,
it is possible to index its elements in such a way that $x_i\leq x_j\implies i\leq j$, and it is common to write $\ee_{ij}$ for $\ee_{x_ix_j}$. In this case, the map $\ee_{ij}\mapsto E_{ij}$ defines an embedding
of $\IPR$ into the algebra $\UT_n(\R)$ of upper triangular matrices with entries in \R.
Even when \PO is infinite, we usually say that the elements in $\{f\in\IPR\,|\, f(x,y)\neq 0\implies x=y\}$ are \highlight{diagonal},
and the elements in
\[
  \zIPR = \{ f\in\IPR\,|\,f(x,y)\neq 0 \implies x\neq y \}
\]
(not to be confused with the center) are \highlight{strictly upper triangular}. Every element $f\in\IPR$ decomposes uniquely into $f=f_d+f_z$ where $f_d$ is diagonal and $f_z $ is strictly upper triangular.
It is easy to check that $\zIPR$ is
a two-sided ideal. This ideal plays an important role in the structure theory
of incidence algebras along with the \highlight{Jacobson Radical} \jIPR and
the \highlight{nilradical} \nIPR. It turns out that \jIPR is the intersection of all
maximal ideals (\cite[Proposition 4.2.8]{Spiegel-ODonnell}) and \nIPR is the intersection of all prime ideals in \IPR, thus
\zIPR and \nIPR are ideals in \jIPR. The following fact concerning these ideals will be very useful to us:
\begin{lemma}\label{lem:radicals}{\cite[Proposition 4.2.8]{Spiegel-ODonnell}} Let \R be a commutative ring with 1 and \PO  \dLPO. Then the map $\IPR/\zIPR\to\prod_{x\in\PO}\R$ given by $f+\zIPR\mapsto (f(x,x))_{x\in\PO}$ is an isomorphism of \R-algebras.
\end{lemma}
Our main reference for basic facts concerning incidence algebras of locally finite partially ordered sets over commutative rings is \cite{Spiegel-ODonnell}.

In this work, we consider group gradings on the algebra $\Aalg=\IPR$. Recall that an \R-algebra \Aalg is graded by a group \G if there exist \R-submodules $\Aalg^g$ (possibly zero) for each $g\in\G$, such that:
\begin{enumerate}
    \item $\Aalg = \bigoplus_{g\in\G}\Aalg^g;$
    \item $\Aalg^g\Aalg^h\subseteq \Aalg^{gh}$, for all $g,h\in\G$.
\end{enumerate}
In this case, an element $a\in\Aalg^g$ is said to be \highlight{homogeneous} of degree $g$, and we write $\deg a=g$. In the case where $\Aalg=\Aalg^1$, the grading is called \highlight{trivial}. The grading is \highlight{fine} if every nonzero $\Aalg^g$ is a free \R-module of rank one, and \highlight{strong} if $\Aalg^g\Aalg^h=\Aalg^{gh}$ for all $g,h\in\G$. A \highlight{graded} submodule, subalgebra or ideal is one that contains the homogeneous components of each of its elements. A subalgebra \Dalg of \Aalg is graded precisely when $\Dalg=\bigoplus_g(\Dalg\cap\Aalg^g)$, and then the identity of \Dalg, if any, is homogeneous of degree $1$. A \highlight{graded isomorphism} between graded algebras (or graded modules) is an isomorphism which carries homogeneous elements of degree $g$ to homogeneous elements of degree $g$, for every $g\in\G$. Two \G-gradings on the same algebra \Aalg are \highlight{isomorphic} if some \R-algebra automorphism of \Aalg is a graded isomorphism from the first grading to the second. A \highlight{graded inner automorphism} is the inner automorphism determined by an invertible element of $\Aalg^1$.

The grading on \IPR is called
\highlight{good} if the elements $\ee_{xx}$ are homogeneous, for all $x\in\PO$.
If \PO is finite of size $n$, we may assume that $\PO=\{1,2,\ldots,n\}$. In this case, we say that the grading is \highlight{elementary} if $\deg\ee_{ij}=g^{-1}_i g_j$ for some sequence $(g_1,\ldots,g_n)\in\G^n$. It is clear that an elementary grading is good, but not the other way around (cf. \cite[Example 12]{Jones}). When \R is indecomposable, a good grading makes every $\ee_{xy}$ homogeneous. Indeed, $\R\ee_{xy}=\ee_{xx}\IPR\ee_{yy}$ is then a graded submodule, so $\R\ee_{xy}=\bigoplus_gI_g\ee_{xy}$ with $I_g=\{r\in\R\mid r\ee_{xy}\in\IPR^g\}$, and $\R=\bigoplus_gI_g$ is a direct sum of ideals. Writing $1=\sum_gr_g$ with $r_g\in I_g$, we get $r_gr_h\in I_g\cap I_h=0$ for $g\neq h$, so the $r_g$ are orthogonal idempotents with sum $1$. As \R is indecomposable, exactly one of them is $1$, and $\ee_{xy}=r_g\ee_{xy}$ is homogeneous.
In a separate work \cite{Talpo-Schutzer}, we extended and studied this notion for arbitrary locally finite partially ordered sets by considering certain functions $\theta:\PO\to \G$ such that $\deg \ee_{xy}=\theta(x)^{-1}\theta(y)$.
Moreover, as the next examples show, \IPR might admit group gradings
which are not isomorphic to neither elementary nor good gradings (see also \cite[Example 1]{Miller-Spiegel} and \cite[Example 1]{Santulo-Souza-Yasumura}).
\begin{example}\label{exam:simplest-non-good-grading}
Let $\PO=\{1,2\}$ be the antichain with two elements, \G\ be a group containing an element $g$ of order 2, and assume that $2$ is a unit
in \R. The grading assignment
\begin{align*}
 \deg (\ee_{11}+\ee_{22}) &= 1, \\
 \deg (\ee_{11}-\ee_{22}) &= g,
\end{align*}
gives a well defined \G-grading on \IPR. If this grading were isomorphic to
a good grading, then the good grading would be trivial, which is obviously not the case, since $g\neq 1$.
\end{example}
\begin{example}\label{exam:simplest-non-balanced}Let \PO be the partially ordered set $\PO=\{1,2,3\}$ where
$1<3$ and $2<3$, but $1$ and $2$ are unrelated:
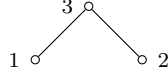
\begin{center}
\begin{tikzpicture}[level/.style={sibling distance = 5cm/#1,
  level distance = 1.5cm}]
\node [arn_n] (a) [label=left:{$\scriptstyle 3$}] {};
\node [arn_n] (b) [below left of=a,label=left:{$\scriptstyle 1$}] {};
\node [arn_n] (d) [below right of=a,label=right:{$\scriptstyle 2$}] {};
\draw []  (a) edge (b);
\draw []  (a) edge (d);
\end{tikzpicture}
\end{center}
Let $\G$ be the abelian 4-group $\{1,g,h,k\}$ with $g^2=1$, $gh=k$, $gk=h$,
and suppose that $o(g)=2$ is a unit in \R. Consider the grading assignment
\begin{align*}
\deg (\ee_{11} + \ee_{22}) &= 1 ,&
\deg (\ee_{11} - \ee_{22}) &= g, \\
\deg \ee_{33} &= 1, \\
\deg (\ee_{13} + \ee_{23}) &= h ,&
\deg (\ee_{13} - \ee_{23}) &= k.
\end{align*}
This gives a well-defined \G-grading on \IPR, which is not isomorphic to a
good grading because $\dim_{\R}\IPR^1<3$ (see also \Cref{cor:equiv:good:grading}).
\end{example}
In both examples above, if we drop the requirement that the order of $g$
is a unit in \R, the
assigment no longer defines a valid grading. For instance, if $\R=\mathbb{Z}/4\mathbb{Z}$, the element $\ee_{11}$ cannot be written as a combination of the given homogeneous elements.

Drawing from \cite{Miller-Spiegel}, the notion of support
and of minimum support
will be useful to the present work.

%
%
\begin{definition}\label{def:minsup}Let \Aalg be a subalgebra of \IPR. The \highlight{support} of $f\in\Aalg$ in \PO is the set
\[ \supf_\Aalg f = \{ x\in \PO \mid f(x,x)\neq 0\}. \]

If $\Aalg$ is a graded subalgebra of \IPR, the \highlight{minimum support} of $x\in\PO$ relative to \Aalg is the set
\[
\minsup_\Aalg(x) = \bigcap_{\mathclap{\substack{f\text{ homogeneous in  \Aalg}\\f(x,x)\neq 0}}}\supf_\Aalg f.
\]
\end{definition}%
When \Aalg is \IPR itself, we write simply $\supf f$ and $\minsup(x)$ for these sets.

Recall that two idempotents $\ee,\ff$ in a ring $\Aalg$ are \highlight{isomorphic},
written $\ee\cong \ff$, if $\ee\Aalg \cong \ff\Aalg$ as right
$\Aalg$-modules.
Also recall that a ring \R is \highlight{indecomposable} if it cannot be expressed
as a direct product of two nonzero rings.
Equivalently, such a ring contains no idempotent elements other than 0 and 1. Over an indecomposable ring the rank of a finitely generated projective module is constant. Whenever \mM is finitely generated projective, we write $\dim_\R\mM$ for this constant rank. Likewise, when \mM is free of finite rank over a corner algebra \Dalg, we write $\dim_\Dalg\mM$ for its free rank.

\begin{lemma}\label{lem:iso:idemp}Let \R\ be an indecomposable commutative ring with 1, and
$e,f$ two idempotents in \IPR. The following are equivalent:
\begin{enumerate}
\item\label{lem:iso:idemp:1} $\supf \ee = \supf \ff$;
\item\label{lem:iso:idemp:4} $\ee-\ff\in\zIPR$;
\item\label{lem:iso:idemp:5} $\ee+\zIPR = \ff+\zIPR$;
\item\label{lem:iso:idemp:2} $\ee\cong \ff$;
\item\label{lem:iso:idemp:6} There are elements $a,b\in\IPR$ such that $\ee=ab$ and $\ff=ba$;
\item\label{lem:iso:idemp:3} There is
an invertible element $h\in\IPR$ such that $\ee=h \ff h^{-1}$.
\end{enumerate}
Moreover, if $\ee$ and $\ff$ commute and satisfy any of the equivalent conditions above, then $\ee=\ff$.
\end{lemma}
\begin{proof}
\labelcref{lem:iso:idemp:1}$\iff$\labelcref{lem:iso:idemp:4} because \R is indecomposable, so that $\ee(x,x)$ and $\ff(x,x)$ are $0$ or $1$ and agree precisely on the common support.
\labelcref{lem:iso:idemp:4}$\implies$\labelcref{lem:iso:idemp:5} is clear, and the converse holds because $\IPR/\zIPR\cong\prod_{x\in\PO}\R$ is a commutative algebra.
\labelcref{lem:iso:idemp:6}$\iff$\labelcref{lem:iso:idemp:2} is the usual description of isomorphic idempotents, an isomorphism $\ee\IPR\to\ff\IPR$ of right modules being given by left multiplication by $b$ and its inverse by left multiplication by $a$, after replacing $a$ by $\ee a\ff$ and $b$ by $\ff b\ee$.
\labelcref{lem:iso:idemp:4}$\implies$\labelcref{lem:iso:idemp:3}: If $h=\ee\ff + (1-\ee)(1-\ff)$, then
\[
1-h = (\ee-\ff)(1-\ff) + (\ee-1)(\ee-\ff) \in \zIPR\subseteq\jIPR.
\]
It follows that $h$ is invertible and $\ee h = \ee\ff = h \ff$, as required.
\labelcref{lem:iso:idemp:3}$\implies$\labelcref{lem:iso:idemp:6}: simply take $a=h\ff$ and $b=h^{-1}$.
\labelcref{lem:iso:idemp:6}$\implies$\labelcref{lem:iso:idemp:4}: for all $x\in\PO$, we have $(\ee-\ff)(x,x)=a(x,x)b(x,x)-b(x,x)a(x,x)=0$, hence $\ee-\ff\in\zIPR$.

For the last assertion, let $u=\ee-\ff\in\zIPR$. If \ee and \ff commute, then $u^3=u$, hence $u=u^{2n+1}$ for all $n\geq0$. Since $u$ is strictly upper triangular, $u^{m}(x,y)$ is a sum over the strictly increasing chains $x=x_0<x_1<\cdots<x_m=y$ in the finite interval $[x,y]$, and therefore vanishes for $m\geq|[x,y]|$. Thus $u=0$.
\end{proof}

\begin{lemma}\label{lem:iso:corner}Let \R be \dR, \ee be an idempotent of \IPR, and $\QO=\supf\ee$. Then $\ee\IPR\ee\cong\IQR$ as \R-algebras. Furthermore, if \IPR is graded by a group (more generally by a semigroup with cancellation) and \ee is homogeneous, then \IQR inherits the grading.
\end{lemma}
\begin{proof}
Let $\ff=(\ee)_d$ be the diagonal part of \ee so that $\ee-\ff\in\zIPR$. Then the isomorphism is the map $a\mapsto (h^{-1}\ee a \ee h)|_{\QO\times\QO}$, where $h$ is the element in \Cref{lem:iso:idemp}\labelcref{lem:iso:idemp:3}, so that $h^{-1}\ee h=\ff$ and $h^{-1}(\ee\IPR\ee)h=\ff\IPR\ff$, and $\ff\IPR\ff\to\IQR$, $a\mapsto a|_{\QO\times\QO}$, is an isomorphism because \ff is diagonal. If \IPR is graded, and \ee is homogeneous, then $\deg\ee=1$, hence defining $\IQR^g$ as the image of $\ee(\IPR^g)\ee$ gives a well-defined grading on \IQR. Note that the isomorphism preserves diagonal entries: $(h^{-1}\ee a\ee h)(x,x)=a(x,x)$ for $x\in\QO$, while $a(x,x)=\ee(x,x)a(x,x)\ee(x,x)=0$ for $a\in\ee\IPR\ee$ and $x\notin\QO$. When \ee is diagonal, $h=1$ and the isomorphism is simply the restriction $a\mapsto a|_{\QO\times\QO}$.
\end{proof}

The following elementary fact about idempotents will be used repeatedly.
\begin{lemma}\label{lem:idemp:vanishing}Let $\ee$ be an idempotent in \IPR. If the interval $[x,y]$ does not meet $\supf\ee$, then $\ee(x,y)=0$.
\end{lemma}
\begin{proof}
Let $n=|[x,y]|$. Since \ee is idempotent, $\ee=\ee^{n}$, hence
\[
\ee(x,y)=\ee^{n}(x,y)=
\sum_{\mathclap{x=x_0\leq x_1\leq\cdots\leq x_n=y}}\ee(x_0,x_1)\ee(x_1,x_2)\cdots\ee(x_{n-1},x_n),
\]
the sum being over the chains of $n+1$ terms in $[x,y]$. Each such chain repeats some term, say $x_{i-1}=x_i$, and then $\ee(x_{i-1},x_i)=\ee(x_i,x_i)=0$ because $x_i\in[x,y]$. Every summand therefore vanishes.
\end{proof}
In particular, a nonzero idempotent has nonempty support. Suppose next that $\ee a\ff\neq0$ for idempotents $\ee,\ff$ and some $a\in\IPR$. Then there are $x\in\supf\ee$ and $y\in\supf\ff$ with $x\leq y$: indeed $(\ee a\ff)(p,q)=\sum \ee(p,u)a(u,v)\ff(v,q)\neq0$ for some $p\leq q$, and then $[p,u]\cap\supf\ee\neq\varnothing\neq[v,q]\cap\supf\ff$ for some $u\leq v$.

Throughout this work, the main objects of study are the corner subalgebras $\ee\IPR\ee$ and the \R-submodules $\ee\IPR\ff$ where \ee and \ff are idempotent elements. To fix the notation, we write $\Dalg_{\ee}=\ee\IPR\ee$ and $\mM_{\ee\ff}=\ee\IPR\ff$ and we observe that the $\mM_{\ee\ff}$ are in fact $(\Dalg_\ee,\Dalg_\ff)$-bimodules with respect to the multiplication in \IPR.

\begin{remark}\label{rem:diagonalization}
Idempotent elements in \IPR are diagonalizable by an inner automorphism and, if $\ee$ and $\ff$ are commuting idempotents, then they are simultaneously diagonalizable. Explicitly, if $\ee_1,\ldots,\ee_r$ are pairwise orthogonal idempotents with diagonal parts $\dd_1,\ldots,\dd_r$, and $\ee_0=1-\sum_i\ee_i$, $\dd_0=1-\sum_i\dd_i$, then $u=\sum_{i=0}^r\dd_i\ee_i$ has diagonal part $1$, hence is invertible, and $\dd_iu=u\ee_i$ for all $i$. When \IPR is graded by a group \G and the $\ee_i$ are homogeneous, the unit $u$ need not be homogeneous. The grading may nevertheless be transported along the inner automorphism determined by $u$, and in this sense homogeneous idempotents may be assumed to be diagonal without loss of generality.
\end{remark}

%
%
%

\section{Isomorphism between group and incidence algebras}\label{sec:groupalg}

In this section, necessary and sufficient
conditions are given for the existence of an isomorphism between an incidence algebra and a group algebra. The main result is this:

\begin{theorem}\label{thm:ipr:iso:comm:fin:grpalg}
Let \R be \dLR, \G a group and \PO \dLPO. Then the incidence algebra \IPR and the group algebra \RG are isomorphic if, and only if, the following holds:
\begin{enumerate}
\item\label{thm:ipr:iso:comm:fin:grpalg:1} \G is a finite abelian group;
\item\label{thm:ipr:iso:comm:fin:grpalg:2} \PO is a finite antichain and $\oP=\oG$;
\item\label{thm:ipr:iso:comm:fin:grpalg:3} $\RG\cong \R^\oP=\R^\oG$ as \R-algebras.
\end{enumerate}
\end{theorem}

To prove this, we need an elementary lemma:

\begin{lemma}\label{lem:iso:struct}
Assume that \PO and $\G$ are finite and let $\varphi\colon \IPR\rightarrow \RG$ be an \R-algebra isomorphism. Then $\oG\in\uR$, \G is abelian, \PO is an antichain of size \oG, and the isomorphism $\varphi$ and its inverse are respectively given by

\begin{displaymath}
	\ee_{xx} \mapsto	\sum\limits_{g \in \G} \oG^{-1}\chi_x(g^{-1}) g
	\qquad\qquad
	g	\mapsto		\sum\limits_{x\in \PO} \chi_x(g) \ee_{xx},
\end{displaymath}
where $\chi_x(g)=\varphi^{-1}(g)(x,x)$ are well-defined, pairwise distinct \R-valued characters of \G.
Consequently, if \R is indecomposable, up to relabelling the elements of \PO, the isomorphism $\varphi$ is unique.
\end{lemma}

\begin{proof}
Let $\varphi$ and its inverse be expressed by
\begin{displaymath}
	\ee_{xy} \mapsto	\sum_{g \in \G} c^{g}_{xy} g
	\qquad\qquad
	g	\mapsto		\sum_{x,y \in \PO} k^{xy}_{g} \ee_{xy}.
\end{displaymath}
where $c^g_{xy},k^{xy}_{g}\in\R$, and $c^g_{xy}=0$ for all but finitely many $g\in\G$. Of course $k^{xy}_{g}=\varphi^{-1}(g)(x,y)$, and evaluating $\varphi^{-1}(gh)=\varphi^{-1}(g)\varphi^{-1}(h)$ at $(x,x)$ gives the relation $k^{xx}_{gh}=k^{xx}_g k^{xx}_h$, and hence $\chi_x(gh)=\chi_x(g)\chi_x(h)$, for all $x\in\PO$ and all $g,h\in\G$. Further, $\varphi^{-1}(1)=1$, hence $k^{xy}_1 = 0$ for $y\neq x$, and $1=k^{xx}_1=\chi_x(1)$. In particular, this shows that the $\chi_x$ are well defined \R-valued characters of \G.

Next, we observe that in \IPR the following relation holds: $\ee_{xx}\varphi^{-1}(h)\ee_{xx}=k^{xx}_h\ee_{xx}$. Applying $\varphi$ to it, we obtain $\sum_{u,v\in\G}c_{xx}^u c_{xx}^v uhv = k^{xx}_h\sum_{w\in\G}c_{xx}^w w$, which gives $$\sum_{\substack{u,v\in\G \\ uhv=w}}c_{xx}^u c_{xx}^v=k^{xx}_h c_{xx}^w.$$
Setting $w=1$ in this identity, we obtain $\sum_{v\in\G}c_{xx}^{v^{-1}h^{-1}}c_{xx}^{v} = k^{xx}_h c_{xx}^1$, for all $h\in G$, and by substituting $w=h^{-1}$ and $h=1$, we obtain $\sum_{u\in\G}c_{xx}^u c_{xx}^{u^{-1}h^{-1}} = c_{xx}^{h^{-1}}$ (keep in mind that $k^{xx}_1=1$), for all $h\in\G$. From these two relations, we get $c_{xx}^{h^{-1}}=k^{xx}_h c_{xx}^{1}$, which gives $c_{xx}^h=k^{xx}_{h^{-1}}c_{xx}^1$. We conclude that
$\varphi(\ee_{xx})=c_{xx}^1\sum_{g\in\G}k_{g^{-1}}^{xx}g$. Applying $\varphi^{-1}$ to this relation and evaluating at $(x,x)$, we get
$1 = c_{xx}^1\sum_{g\in\G}k^{xx}_{g^{-1}}k^{xx}_g=c_{xx}^1\oG$. It follows that $\oG\in\uR$, $c^{1}_{xx}=\oG^{-1}$, and thus
$$
\varphi(\ee_{xx})=\frac{1}{|G|}\sum_{g\in\G}\chi_x(g^{-1})g,
$$
for all $x\in \PO$. It is worth noting that, evaluating that same relation at $(y,y)$ for $y\neq x$ gives
$$
\frac{1}{\oG}\sum_{g\in G}k^{xx}_{g^{-1}}k^{yy}_g = 0,
$$
which shows that the characters $\chi_x$ are pairwise orthogonal, hence distinct.

Next, we compute:
$$
\varphi(\ee_{xx})h = \frac{1}{\oG}\sum_{g\in\G}k^{xx}_{g^{-1}}gh=\frac{1}{\oG}\sum_{g\in\G}k^{xx}_{hg^{-1}}g=k^{xx}_h\varphi(\ee_{xx})
$$
and similarly
$$
h\varphi(\ee_{xx}) = \frac{1}{\oG}\sum_{g\in\G}k^{xx}_{g^{-1}}hg=\frac{1}{\oG}\sum_{g\in\G}k^{xx}_{g^{-1}h}g=k^{xx}_h\varphi(\ee_{xx})
$$
for all $x\in\PO$ and all $h\in\G$. It follows that $\{\varphi(\ee_{xx})\mid x\in\PO\}$ is a complete set of central orthogonal idempotents in \RG. In particular \G is an abelian group: $\RG=\bigoplus_x\varphi(\ee_{xx})\RG$ as algebras, and each factor $\varphi(\ee_{xx})\RG=\varphi(\ee_{xx}\IPR\ee_{xx})=\R\varphi(\ee_{xx})$ is commutative. But then \IPR is a commutative algebra, hence \PO is an antichain. This proves that $\varphi$ and $\varphi^{-1}$ are as claimed. Finally, $\varphi$ is determined by the labelling $x\mapsto\chi_x$ of \PO by characters. If \R is indecomposable, the $\oP=\oG$ distinct characters $\chi_x$ exhaust \hG, which has \oG elements by \Cref{lem:ipr:iso:comm:fin:grpalg} below (applied to $\RG\cong\IPR\cong\R^{\oP}$). Hence the set $\{\chi_x\}$ does not depend on $\varphi$, and $\varphi$ is unique up to relabelling. (Without indecomposability this fails: for $\R=\CC\times\CC$ and $\G$ of order $2$, \hG has four elements, and different pairs of orthogonal characters give isomorphisms $\Inc{\PO}{\R}\to\RG$ which are not related by a permutation of the two points of \PO.) The proof is complete.
\end{proof}

Now we prove \Cref{thm:ipr:iso:comm:fin:grpalg}:

\begin{proof}
If conditions \labelcref{thm:ipr:iso:comm:fin:grpalg:1,thm:ipr:iso:comm:fin:grpalg:2,thm:ipr:iso:comm:fin:grpalg:3} of the Theorem hold, then $\zIPR=0$ and $\IPR\cong\R^\oP=\R^\oG\cong \RG$ (cf. \Cref{lem:radicals}), and we are done. So we may assume that an isomorphism $\varphi\colon\IPR\to\RG$ exists.

For each $x\in\PO$, the map $\chi_x\colon\G\to\R$ sending $g$ to $\varphi^{-1}(g)(x,x)$ is multiplicative, because $(\varphi^{-1}(g)\varphi^{-1}(h))(x,x)=\varphi^{-1}(g)(x,x)\varphi^{-1}(h)(x,x)$, the interval $[x,x]$ being $\{x\}$, and $\chi_x(1)=1$. It follows that $\chi_x$ is an \R-valued character of \G, and $\ee_{xx}\varphi^{-1}(h)\ee_{xx}=\chi_x(h)\ee_{xx}$ holds for all $h\in\G$. These facts do not depend on the finiteness assumptions of \Cref{lem:iso:struct}, which will be invoked only after \G and \PO are known to be finite. Further, the map $\varphi^{-1}$ induces a \G-grading on \IPR with $\IPR^g=\R\varphi^{-1}(g)$ for all $g\in\G$ (and this grading is fine and strong).

For any $x\in\PO$ fixed and all $h\in\G$, if $\ee_{xx}=\sum_{i=1}^{n}r_i \varphi^{-1}(g_i)$ is the homogeneous decomposition of $\ee_{xx}$, where $r_i\in\R$ are all nonzero, then $\sum_{i,j=1}^{n}r_ir_j\varphi^{-1}(g_i h g_j)=\sum_{k=1}^{n}r_k\chi_x(h)\varphi^{-1}(g_k)$. For at least one triple $(i,j,k)$, we must have $g_i h g_j = g_k$, and thus $h=g_i^{-1}g_k g_j^{-1}$ belongs to a finite set. It follows that \G is a finite group.

But then, \IPR is a free \R-module of finite rank equal to $\oG$. Therefore \PO is finite: for a finite subset $F\subseteq\PO$, the submodule $\bigoplus_{x\in F}\R\ee_{xx}$ is a direct summand of \IPR, a complement being the kernel of $f\mapsto\sum_{x\in F}f(x,x)\ee_{xx}$. Reducing modulo a maximal ideal of \R gives $|F|\leq\oG$. By the Lemma, \G is an abelian group, \IPR is commutative, and \PO is antichain. Further, $\oP=\dim_{\R}\IPR=\dim_{\R}\RG=\oG$. To finish up, we notice again that $\zIPR=0$, hence $\RG\cong\IPR\cong\R^\oP=\R^\oG$, and thus conditions \labelcref{thm:ipr:iso:comm:fin:grpalg:1,thm:ipr:iso:comm:fin:grpalg:2,thm:ipr:iso:comm:fin:grpalg:3} of the Theorem hold, as required.
\end{proof}


Condition \labelcref{thm:ipr:iso:comm:fin:grpalg:3} of \Cref{thm:ipr:iso:comm:fin:grpalg} says that \R is a splitting ring for \G. We record the facts about split group algebras which will be used throughout.

Recall that a commutative ring \R is a \highlight{splitting ring} for a finite group \G if the group algebra \RG is isomorphic to a direct sum of matrix algebras over \R, namely $\RG\cong\bigoplus_{i=1}^n M_{d_i}(\R)$ for certain positive integers $d_i$. The following well-known lemma gives a characterization of splitting rings for finite groups, which is according to our intentions. Here $\hG\coloneqq \Hom(\G,\uR)$ is the dual group of \G, namely set of all \R-valued characters of \G. An element $\omega\in\R$ is a \highlight{primitive $n$-th root of unity} if it has order exactly $n$ in \uR. When \R is indecomposable and $n\in\uR$, the polynomial $x^n-1$ is separable over \R, so that a primitive $n$-th root of unity is a root of the $n$-th cyclotomic polynomial, and \R contains at most $n$ roots of $x^n-1$ \cite[Corollaries 2.4 and 2.5]{janusz}. Every divisor $d$ of $n$ is again a unit, so the same applies to $x^d-1$. The $n$-th roots of unity in \R therefore have at most $d$ elements of order dividing $d$ for each $d\mid n$, and a finite abelian group with that property is cyclic. In particular the $n$-th roots of unity in \R form a cyclic group.

\begin{lemma}\label{lem:ipr:iso:comm:fin:grpalg}
Let \R be \dR and \G be a finite group. Then, $\RG\cong \R^\oG$ as \R-algebras if, and only if:
	\begin{enumerate}

        \item\label{lem:ipr:iso:comm:fin:grpalg:1} \G is abelian;

        \item\label{lem:ipr:iso:comm:fin:grpalg:2}$\oG\in\uR$, that is \RG is a separable \R-algebra;

        \item\label{lem:ipr:iso:comm:fin:grpalg:3}\R contains a primitive $\Exp(\G)$-th root of the unity;



    \end{enumerate}
    Under these conditions, we have $\G\cong\hG$. Furthermore, \RG is a Galois extension of \R with Galois group \hG and with normal \R-basis consisting of the complete set of pairwise orthogonal idempotents $\{\bee_\chi\mid \chi\in\hG\}$, where $\bee_\chi=\frac{1}{|\G|}\sum_{g\in\G}\chi(g^{-1})g$ (the \highlight{Fourier idempotents} of \RG).
\end{lemma}

\begin{proof}First assume that $\RG\cong\R^\oG$ as \R-algebras, so that \R is a splitting ring for \G. \Cref{lem:ipr:iso:comm:fin:grpalg:1} is obvious. To see \cref{lem:ipr:iso:comm:fin:grpalg:2}, let \iM be an arbitrary maximal ideal of \R and let $\F=\R/\iM$. Since $\RG\otimes_\R \F\cong \R^\oG\otimes_\R\F$, it follows that $\FG\cong\F^\oG$ as \F-algebras. In particular, \FG is semisimple, hence from Maschke's Theorem (\cite[Theorem 3.4.7]{Polcino-Sehgal}) it follows that $\oG\cdot 1_\F$ is a unit in \F, so $\oG\cdot 1_\R\notin\iM$. Since \iM is arbitrary, $\oG\cdot 1_\R$ must be a unit in \R. Now the element $\bee=\oG^{-1}\sum_{g\in\G}g^{-1}\otimes g\in \RG\otimes_\R\RG$ is a separability idempotent, since
$\mu(\bee)=\oG^{-1}\sum_{g\in\G}g^{-1}g=1_\G$. Therefore, \RG is a separable \R-algebra (\cite[proof of Lemma 1.2]{Chase}).

To prove \cref{lem:ipr:iso:comm:fin:grpalg:3}, we first show that $|\G|=|\hG|$, where $\hG=\Hom(\G,\uR)$ is the set of \R-valued characters of \R.  Let $\pi_i\colon\R^\oG\to \R$ be the natural projection onto the $i$-th coordinate. The composition $\pi_i\circ\varphi$ is a surjective \R-algebra map from \RG to \R, whose restriction to \G defines the character $\chi_i\colon\G\to\uR$. These must be distinct. Indeed, if  $(\pi_i-\pi_j)\circ\varphi$ were the zero map, then the image of $\varphi$, namely $\R^\oG$, would be contained in $\ker(\pi_i-\pi_j)$, which is impossible since $\pi_i-\pi_j\neq 0$. This contradiction shows that the $\chi_i$ are distinct.

If $f\colon\R^\oG\to\R$ is any \R-algebra map and $\ee_i=(0,\ldots,1,\ldots,0)\in\R^\oG$ is the central idempotent with 1 at the $i$-th position, then $f(\ee_i)=f(\ee_i^2)=f(\ee_i)^2$, hence $f(\ee_i)$ is an idempotent in \R. Since \R is indecomposable, we have $f(\ee_i)\in\{0,1\}$. Since $f(\ee_i)f(\ee_j)=f(\ee_i\ee_j)=0$ for $i\neq j$ and $f(\ee_1)+\cdots+f(\ee_{\oG})=f(1)=1$, we have $f(\ee_i)=1$ for precisely one index $i$ and $f(\ee_j)=0$ for all $j\neq i$. Therefore, $f=\pi_i$. It follows that $|\G|=|\Hom_{\R\text{-alg}}(\R^\oG,\R)|=|\Hom_{\R\text{-alg}}(\RG,\R)|=|\hG|$.

We have established that all characters $\chi\in\hG$ are of the form $\pi_i\circ\varphi$ and thus we may write $\varphi(g)=(\chi_1(g),\ldots,\chi_{\oG}(g))$ for all $g\in\G$. Let $n=\Exp(\G)$. Since \G is a finite abelian group, it contains an element $g$ of order $n$. Since $\varphi$ is an isomorphism, the order of $\varphi(g)=(\chi_1(g),\ldots,\chi_\oG(g))$ in $(\R^\oG)^\times$ is $n$, that is, the subgroup of \uR generated by the $\chi_i(g)$ has exponent $n$. Call this subgroup $S$. Every divisor $d$ of $n$ is again a unit in \R, so $x^d-1$ has at most $d$ roots in \R because \R is indecomposable (\cite[Corollary 2.5]{janusz}). Hence $S$ has at most $d$ elements of order dividing $d$ for every $d\mid n$, and a finite abelian group with that property is cyclic. As $S$ has exponent $n$, it is cyclic of order $n$. Taking $d=n$ shows that $S$ is the whole group of $n$-th roots of unity of \R, and one of the $\chi_i(g)$ is a primitive $n$-th root of unity $\omega$ with $S=\langle\omega\rangle$.

Under these conditions, the image of any character in \hG consists of $n$-th roots of unity, hence lies in $\langle\omega\rangle$ by the previous paragraph. This subgroup, in turn, is isomorphic to the subgroup $\langle e^{i2\pi/n}\rangle\subset\CC$, hence we may view $\hG$ as the usual Pontryagin dual of \G. In particular, $\G\cong\hG$, and the elements of \hG satisfy the usual Schur orthogonality relations. In turn, it follows that $\{\bee_\chi\mid\chi\in\hG\}$ forms a complete set of orthogonal idempotents for \RG.

Now we show that $\RG$ is a Galois extension of \R with Galois group \hG and having the idempotent basis as a normal basis. To see this, consider the map $\alpha_\chi\colon \RG\to\RG$ such that $g\mapsto \chi(g^{-1})g$. Since $\alpha_\chi(1)=\chi(1)1=1$ and $\alpha_\chi(g_1g_2)=\chi(g_2^{-1}g_1^{-1})g_1g_2=(\chi(g_1^{-1})g_1)(\chi(g_2^{-1})g_2)=\alpha_\chi(g_1)\alpha_\chi(g_2)$, we see that $\alpha_\chi\in\Aut_{\R\text{-alg}}(\R\G)$. Moreover, $\alpha_\chi(\bee_\rho)=\oG^{-1}\sum_{g\in\G}\rho(g^{-1})\alpha_\chi(g)=\oG^{-1}\sum_{g\in\G}\rho(g^{-1})\chi(g^{-1})g=\bee_{\chi\rho}$, hence \hG acts transitively on the idempotent basis. To compute the fixed ring of \hG,
for arbitrary $a\in\RG$, there are unique $a_\rho\in\R$ such that $a=\sum_{\rho\in\hG}a_\rho\bee_\rho$. Then, $\alpha_\chi(a)=a$ if and only if $a_{\chi^{-1}\rho}=a_\rho$ for all $\chi,\rho\in\hG$. In particular, for $\rho=\chi$, we get $a_{\hat{1}}=a_\chi$ for all $\chi\in\hG$, therefore $a=a_{\hat{1}}1_\G\in\R$.

With respect to the separability idempotent $\bee$ from before, we have
$$
\mu(\operatorname{id}\otimes\alpha_\chi)(\bee)=\frac{1}{\oG}\sum_{g\in\G}g^{-1}\alpha_\chi(g)=\frac{1}{\oG}\sum_{g\in\G}\chi(g^{-1})g^{-1}g=\delta_{\chi,\hat{1}}1_\G,
$$
for all $\chi\in\hG$. It follows from \cite[Theorem 1.3(b)]{Chase} that $\RG$ is a Galois extension of \R with Galois group \hG and the idempotent basis is a normal basis, as required.

It remains to prove the converse. If all three conditions are holding, then the $n$-th roots of unity in \R form a cyclic group of order $n=\Exp(\G)$, whence $\hG\cong\G$ (write \G as a product of cyclic groups), and the orthogonality relations $\sum_{g\in\G}\chi(g)\rho(g^{-1})=\oG\delta_{\chi\rho}$ hold: for $\chi\neq\rho$ the sum is invariant under multiplication by some value $\zeta\neq1$ of $\chi\rho^{-1}$, and $1-\zeta$ is a unit (\cite[Lemma 2.1]{janusz}). It follows that the character table $(\chi(g))$ is invertible over \R, and the map $g\mapsto(\chi(g))_{\chi\in\hG}$ gives an \R-algebra isomorphism between \RG and $\prod_{\chi\in\hG}\R=\R^\oG$.
\end{proof}

\begin{remark}\label{rem:galois:extensions}
Under the conditions of \Cref{thm:ipr:iso:comm:fin:grpalg}, \Cref{lem:iso:struct,lem:ipr:iso:comm:fin:grpalg} describe the isomorphism $\varphi\colon\IPR\to\RG$ completely. We fix the following conventions.
\begin{enumerate}
\item\label{rem:galois:points} \emph{Points and characters.} The map $x\mapsto\chi_x=\varphi^{-1}(\cdot)(x,x)$ is a bijection $\PO\to\hG$, since the $\chi_x$ are $\oP=\oG=|\hG|$ pairwise orthogonal, hence distinct, characters. Moreover $\varphi(\ee_{xx})=\bee_{\chi_x}$. We identify \PO with \hG by this map, and $\varphi^{-1}$ becomes the map $g\mapsto(\chi(g))_{\chi\in\hG}$ of $\RG$ onto $\R^{\hG}=\Inc{\hG}{\R}$. The identification depends on the choice of $\varphi$, as described in \labelcref{rem:galois:galois}.
\item\label{rem:galois:action} \emph{Action of \G on \IPR.} For $g\in\G$ and $f\in\IPR$ we write $g\,f=\varphi^{-1}(g)f$ and $f\,g=f\varphi^{-1}(g)$. Therefore $g\,\ee_{xy}=\chi_x(g)\ee_{xy}$ and $\ee_{xy}\,g=\chi_y(g)\ee_{xy}$. The element $\ee_{xx}$ acts on \RG as the idempotent $\bee_{\chi_x}$, and $\bee_\chi\,f\,\bee_\rho$ is the $(\chi,\rho)$-entry of $f$. The identification will be used in this form for the bimodules $\ee\IPR\ff$ in \Cref{sec:bimodules}.
\item\label{rem:galois:galois} \emph{Galois action.} The automorphism $\alpha_\chi$ of \RG, $g\mapsto\chi(g^{-1})g$, satisfies $\alpha_\chi(\bee_\rho)=\bee_{\chi\rho}$, so $\varphi^{-1}\alpha_\chi\varphi$ is the automorphism of \IPR which sends $\ee_{xx}$ to $\ee_{x'x'}$, where $\chi_{x'}=\chi\chi_x$. Under $\PO=\hG$ it is the translation $\rho\mapsto\chi\rho$. It follows that \hG acts simply transitively on \PO. Replacing $\varphi$ by $\alpha_\chi\circ\varphi$ multiplies every label $\chi_x$ by $\chi$. The $\alpha_\chi$ are precisely the automorphisms of \RG preserving its grading by \G, since such an automorphism has the form $g\mapsto t(g)g$ with $t\in\hG$, and equals $\alpha_{t^{-1}}$. Hence two isomorphisms $\varphi,\varphi'$ which differ by a graded automorphism of \RG give labellings of \PO differing by a translation.
\item\label{rem:galois:base:points} \emph{Base points.} In \Cref{sec:xmin} this remark is applied to the corner algebras $\Dalg_{\ee}=\ee\IPR\ee\cong\RH_\ee$ of a graded incidence algebra. There the isomorphism is normalized at a base point $x\in\supf\ee$, by $|\H_x|\ee_{xx}^h\mapsto h$. The bijection $\supf\ee\to\hH_\ee$ becomes $p\mapsto\chi_x^p$, and a change of base point from $x$ to $x'$ multiplies all labels by the character $\chi_{x'}^{x}$, in accordance with \labelcref{rem:galois:galois}. This dependence must be taken into account in \Cref{sec:classification}.
\end{enumerate}
\end{remark}

\begin{example}
Let $\R$ be either $\F_3$, the field with 3 elements, or the integers $\mathbb{Z}$, and $\G=\{1,g,g^2\}$ be the cyclic group with 3 elements. Since $3$ is not a unit in \R, in view of \Cref{thm:ipr:iso:comm:fin:grpalg,lem:ipr:iso:comm:fin:grpalg}, \RG cannot be isomorphic to any \IPR as \R-algebras.
\end{example}

\begin{example}
This example shows that the existence of roots of the unity in \R is essential. If $\G=\{1,g,g^2\}$, \PO is the 3-antichain and $\R=\RR$, then $\Inc{\PO}{\RR}\cong\RR^3$ contains (at least) 3 nontrivial orthogonal idempotents while $\RR\G$ contains only two, namely $\ee=\frac{1}{3}(1+g+g^2)$ and $(1-\ee)=\frac{1}{3}(2-g-g^2)$, hence $\Inc{\PO}{\RR}$ and $\RR\G$ are not isomorphic.
\end{example}

\begin{example}
If \PO is the 2-chain $\CO_2$, then $\Inc{\CO_2}{\mathbb{C}}$ is 3-dimensional. If $\Inc{\CO_2}{\mathbb{C}}\cong\mathbb{C}\G$ then $\G$
is the group with 3 elements. But this is not possible since $\mathbb{C}\G$
is commutative and $\Inc{\CO_2}{\mathbb{C}}$ is not.
\end{example}
\begin{remark} In general, one is not to expect $\RG\cong\IPR$ to be semisimple. In fact,
since $\G$ is finite abelian and $\oG$ is a unit in \R, we have that $\RG\cong\IPR$ is semisimple if, and only if, the ring \R itself is semisimple (cf. \cite[Theorem 3.4.7]{Polcino-Sehgal}).
\end{remark}

%
%
%

\section{Gradings and Primitive Idempotents}\label{sec:xmin}
From this section onward, we assume that \R is \dR and that \PO is finite, unless stated otherwise. In the present section finiteness enters only through the existence of minimal and maximal elements in every nonempty subposet of \PO. The main result is the existence of a primitive homogeneous idempotent \ee when \PO has extremal elements and a precise description of its corresponding corner subalgebra $\Dalg_{\ee}=\ee\IPR\ee$, assuming that \IPR is graded by a group \G. Recall that \ee is primitive homogeneous if it is a nonzero homogeneous idempotent which is not a sum of two nonzero homogeneous orthogonal idempotents. For an arbitrary element $f\in\IPR$, we denote its homogeneous decomposition by
$f = \sum_{g\in\G}f^g$, where it is understood that all but finitely many of the homogeneous components $f^g$ are zero. Subsets of group elements that participate in the homogeneous decomposition of the elementary functions $\ee_{xy}$ play a crucial role in the our description of the structure of \IPR. First we require some notation:

\begin{definition}\label{def:H:phi}
Let $p, q, x, y \in \PO$. We define the following objects:
\begin{enumerate}
    \item \emph{The support set:}
    \[
    \H_{xy}^{pq} \coloneqq \{\, h \in \G \mid \ee_{xy}^h(p,q) \neq 0 \}.
    \]
    That is, \(\H_{xy}^{pq}\) is the (finite, possibly empty) subset of \G corresponding to nonzero values of $\ee_{xy}^h(p,q)$.

    \item \emph{The map:}
    \[
    \phi_{xy}^{pq} \colon \G \to \R,\quad g \mapsto |\H_{xy}^{pq}|\cdot \ee_{xy}^g(p,q).
    \]
\end{enumerate}

For ease of notation, we also set
\[
\H_{xy} \coloneqq \Hsup{xy}{xy},\quad \Hsup{x}{p} \coloneqq \Hsup{xx}{pp},\quad \text{and} \quad \H_x \coloneqq \H_{xx} = \Hsup{x}{x}.
\]
Under appropriate conditions, it will later be shown that $\H_x$ is a group and the map $
\chi_{x}^{p} \coloneqq \phisup{xx}{pp}$
defines a character of $\H_x$ (which resembles the character $\chi_{x}$ introduced in \Cref{lem:iso:struct}).
\end{definition}

The main result is this:

\begin{theorem}\label{thm:xmin}
If \PO has a minimal (resp. maximal) element $x$, then
\begin{enumerate}
    \item\label{thm:xmin:1} $\H_x = \{ h\in\G \mid \ee_{xx}^h(x,x) \neq 0\}$ is a finite abelian subgroup of \G whose order $|\H_x|$ is a unit in \R; $\R$ contains a primitive $\Exp(\H_x)$-th root of $1_\R$; moreover, $\ee_{xx}^h(x,x) = |\H_x|^{-1}$ for all $h\in\H_x$, while $\ee_{xx}^h\ee_x=0$ (resp. $\ee_x\ee_{xx}^h=0$) for all $h\notin\H_x$, where $\ee_x$ is the idempotent of \labelcref{thm:xmin:2}.

    \item\label{thm:xmin:2} $\ee_{x} \coloneqq |\H_x|\ee_{xx}^1$ is a primitive homogeneous idempotent whose support $\supf\ee_x$ is an antichain of cardinality $|\H_x|$;

    \item\label{thm:xmin:3}  $\Dalg_{\ee_x} = \Span\{\ee_{xx}^h \mid h \in \H_x \} \cong I(\supf\ee_x,\R) \cong \RH_x \cong \R^{|\H_x|}$ are graded isomorphisms of \R-algebras and $\Dalg_{\ee_x}$ is a left (resp. right) ideal of $\IPR$.

    \item\label{thm:xmin:4} Each $y\in\supf\ee_x$ is also minimal (resp. maximal) in \PO, $\supf\ee_y=\supf\ee_x$, $\H_y=\H_x$, and, with $\chi_x^y(h)\coloneqq|\H_x|\ee_{xx}^h(y,y)$,
    \begin{gather*}
    \ee_{xx}^h\,\ee_y=\chi_x^y(h)\,\ee_{yy}^h\quad(\text{resp. }\ee_y\,\ee_{xx}^h=\chi_x^y(h)\,\ee_{yy}^h),\\
    \ee_{xx}^h(p,p)=\chi_x^y(h)\,\ee_{yy}^h(p,p)\quad(p\in\PO)
    \end{gather*}
    for all $h\in\H_x$. In particular $\ee_{xx}^h=\chi_x^y(h)\ee_{yy}^h$ whenever $\ee_x=\ee_y$.
    In particular, $\chi_{x}^{y}$ is a character of $\H_x = \H_x^y = \H_y^x = \H_y$ (so that $\chi_x^y$ agrees with the map $\phi_{xx}^{yy}$ of \Cref{def:H:phi}), $\chi_x^y\chi_y^z=\chi_x^z$ for $y,z\in\supf\ee_x$, and $y\mapsto\chi_x^y$ is a bijection from $\supf\ee_x$ onto $\hH_x$. Moreover, for all $p,y\in\supf\ee_x$, there exists $q\in\supf\ee_x$ such that $\chi_{x}^{y} = \chi_{p}^{q}$.
\end{enumerate}
\end{theorem}

\begin{proof}
Items \labelcref{thm:xmin:1} to \labelcref{thm:xmin:3}. For ease of notation, we denote $c_k = \ee_{xx}^k(x,x)$ ($k\in\G$) during the first part of this proof.

When $x$ is minimal in \PO, $f \ee_{xx}=f(x,x)\ee_{xx}$ for all $f\in\IPR$. Comparing homogeneous components when $f$ is itself homogeneous, one gets
\[
f(x,x) \ee_{xx}^g = f \ee_{xx}^{h^{-1}g}
\]
for all $g,h\in\G$ and $f\in\IPR^h$. In particular, by successively setting $f=\ee_{xx}^h$ and $f=\ee_{xx}^g$, and evaluating the resulting equation on $(x,x)$, one obtains $c_h c_g = c_h c_{h^{-1}g}$ and $c_g^2 = c_1c_g$ for all $g,h\in\G$. By the pigeonhole principle, $c_1$ divides each summand of
\[
\sum_{(m)\vdash |\H_x|+1}\binom{|\H_x|+1}{(m)} (c_{k_1}^{m_1} c_{k_2}^{m_2}\cdots) = \bigg(\sum_{k\in\H_x}c_k\bigg)^{|\H_x|+1} = 1,
\]
whereby it is a unit in \R. It follows successively that: $c_1^{-1}c_g$ is idempotent in \R for all $g\in\G$, hence equal to either $0$ or $1$ because \R is indecomposable; $c_h = c_1$ for all $h\in\H_x$; $1 = \sum_{k\in\H_x}c_k = |\H_x| c_1$; $c_{h^{-1}g} = c_g = |\H_x|^{-1}$ for all $g,h\in\H_x$; and finally $\H_x$ is a finite subgroup of \G.

Now since $|\H_x|f(x,x)\ee_{xx}^{\deg f} = \ee_x f \ee_x = f \ee_x$ for all $f$ homogeneous, $\Span\{ |\H_x|\ee_{xx}^h \mid h\in \H_x\} = \Dalg_{\ee_x}$ (for the reverse inclusion note that $f(x,x)=(f\ee_x)(x,x)=|\H_x|c_hf(x,x)$ for $f\in\IPR^h$, so that $f(x,x)=0$ unless $h\in\H_x$) as $\H_x$-graded algebras and $\Dalg_{\ee_x}$ is a left ideal of \IPR. In particular, when $f=\ff$ is any homogeneous idempotent whose support contains $x$, $\ee_x \ff \ee_x = \ee_x$ (i.e. $\ff\nless \ee_x$), whence $\ee_x$ is primitive. Moreover, $|\H_x|\ee_{xx}^{gh} = (|\H_x|\ee_{xx}^{g})(|\H_x|\ee_{xx}^{h})$ for all $g,h\in\H_x$, whence $\{|\H_x| \ee_{xx}^h \mid h\in \H_x\}$ is an \R-basis of $\Dalg_{\ee_x}$ (they lie in the distinct homogeneous components $\IPR^h$, $h\in\H_x$, so every term of a vanishing \R-combination vanishes, and $(|\H_x|\ee_{xx}^h)(x,x)=|\H_x|c_h=1$ then forces each coefficient to be zero) and the map $|\H_x|\ee_{xx}^{h} \mapsto h$ is a graded \R-algebra isomorphism from $\Dalg_{\ee_x}$ to $\RH_x$.

Applying \Cref{thm:ipr:iso:comm:fin:grpalg} to $\RH_x \cong \Dalg_{\ee_x} \cong I(\supf\ee_x,\R)$ (cf. \Cref{lem:ipr:iso:comm:fin:grpalg,lem:iso:corner}) then shows: $\H_x$ is abelian; $\supf\ee_x$ is an antichain; $|\supf\ee_x| = |\H_x|$; $\R$ contains an $\Exp(\H_x)$-th primitive root of $1_\R$; $\Dalg_{\ee_x} \cong \RH_x \cong \R^{|\H_x|}$ as commutative graded algebras, and $\ee_{xx}^g\ee_x = |\H_x|\ee_{xx}^g\ee_{xx}^1 = |\H_x| c_g \ee_{xx}^g = 0$ (resp. $\ee_x\ee_{xx}^g=0$) for all $g\notin \H_x$, by the displayed identity with $f=\ee_{xx}^g$ and $h=g$. Note also that the same identity with $f=\ee_{xx}^1$ gives $\ee_x\ee_{xx}^g=|\H_x|c_1\ee_{xx}^g=\ee_{xx}^g$ (resp. $\ee_{xx}^g\ee_x=\ee_{xx}^g$) for all $g\in\G$, and $\ee_{xx}^h\ee_x=\ee_{xx}^h$ (resp. $\ee_x\ee_{xx}^h=\ee_{xx}^h$) for $h\in\H_x$.

\Cref{thm:xmin:4}. Let $y\in\supf\ee_x$ and note that, since $\Dalg_{\ee_x}$ is a left ideal of \IPR, and a commutative algebra with identity $\ee_x$ (so that $\ee_{zy}\ee_x=\ee_x\ee_{zy}\ee_x$),
\begin{displaymath}
\ee_x(y,y) \ee_{zy} = \ee_{zy} \ee_x \ee_{yy} = \ee_{zz} \ee_x \ee_{zy} \ee_x \ee_{yy} = \ee_x(z,z) \ee_x(y,y) \ee_{zy}
\end{displaymath}
 for all $z \leq y$. Since $\supf\ee_x$ is an antichain, $z\nless y$, meaning $y=z$ is minimal in \PO. 

Repeating the same steps up to this point for $y$, which is extremal in \PO, yields a finite subgroup $\H_y$ of $\G$ and a primitive homogeneous idempotent $\ee_y = |\H_y|\ee_{yy}^1$. Since $y$ is minimal, $f\ee_{yy}^g=f(y,y)\ee_{yy}^{hg}$ for $f\in\IPR^h$ (as in the proof of \labelcref{thm:xmin:1}). With $f=\ee_{xx}^h$ and $g=1$ this gives
\[
\ee_{xx}^h\,\ee_y=|\H_y|\ee_{xx}^h(y,y)\,\ee_{yy}^h\qquad(h\in\G).
\]
For $h=1$ we get $\ee_x\ee_y=\ee_x(y,y)\ee_y=\ee_y$, so $\ee_y\in\ee_x\IPR$ and $\supf\ee_y\subseteq\supf\ee_x$. Since $\ee_y=\ee_y^2=\ee_y\ee_x\ee_y$, we have $\ee_y\ee_x\neq0$, and the identity $\ee_y\ee_x=\ee_y(x,x)\ee_x$ (from $x$ minimal) forces $\ee_y(x,x)=1$, hence $\ee_y\ee_x=\ee_x$ and $\supf\ee_x\subseteq\supf\ee_y$. It follows that $\supf\ee_y=\supf\ee_x$ and $|\H_y|=|\H_x|$ by \labelcref{thm:xmin:2}. (The idempotents $\ee_x$ and $\ee_y$ themselves need not coincide: they have the same diagonal part, but their strictly upper triangular parts may differ. They do coincide when $\ee_x$ is diagonal, and, more generally, whenever some diagonal homogeneous idempotent \ee has $\supf\ee=\supf\ee_x$, since then $\ee_{xx},\ee_{yy}\in\ee\IPR\ee$ and $\ee_x,\ee_y$ are homogeneous idempotents of degree $1$ in the fine-graded algebra $\ee\IPR\ee\cong\RH$, whose only such idempotents are $0$ and \ee. See also \Cref{rem:ex:equals:e} below.) Now let $h\in\H_x$. Since $|\H_y|=|\H_x|$, the displayed identity reads $\ee_{xx}^h\ee_y=\chi_x^y(h)\ee_{yy}^h$, which is the first identity of \labelcref{thm:xmin:4}. Taking diagonal parts in it, and using that the diagonal part of a product is the product of the diagonal parts,
\[
\ee_{xx}^h(p,p)\,\ee_y(p,p)=\chi_x^y(h)\,\ee_{yy}^h(p,p)\qquad(p\in\PO).
\]
For $p\in\supf\ee_y=\supf\ee_x$ we have $\ee_y(p,p)=1$, and the second identity of \labelcref{thm:xmin:4} follows. For $p\notin\supf\ee_x$ the right-hand side vanishes, and so does $\ee_{xx}^h(p,p)=(\ee_{xx}^h\ee_x)(p,p)=\ee_{xx}^h(p,p)\ee_x(p,p)$, because $\ee_{xx}^h=\ee_{xx}^h\ee_x$ for $h\in\H_x$ by \labelcref{thm:xmin:1}. This proves the two identities of \labelcref{thm:xmin:4} (the case of a maximal element is symmetric). Evaluating the second one at $p=x$ gives $|\H_x|^{-1}=\chi_x^y(h)\ee_{yy}^h(x,x)$, so $\ee_{xx}^h(y,y)\neq0\neq\ee_{yy}^h(x,x)$ for $h\in\H_x$. Moreover $\ee_y(x,x)=1$, being a nonzero idempotent of the indecomposable ring \R, as $x\in\supf\ee_y$, so $(\ee_{yy}^h\ee_y)(x,x)=\ee_{yy}^h(x,x)\ee_y(x,x)=\ee_{yy}^h(x,x)\neq0$ forces $h\in\H_y$ by \labelcref{thm:xmin:1} applied to $y$. Thus $\H_x\subseteq\H_x^y\cap\H_y^x\cap\H_y$, and by symmetry $\H_y\subseteq\H_x$. Conversely, $\H_x^p\subseteq\H_x$ for every $p\in\PO$: if $h\notin\H_x$, then $\ee_{xx}^h=\ee_x\ee_{xx}^h$ and $\ee_{xx}^h\ee_x=0$ by the identities at the end of the proof of \labelcref{thm:xmin:1} to \labelcref{thm:xmin:3}, so $\ee_{xx}^h=\ee_x\ee_{xx}^h(1-\ee_x)$ has diagonal entries $\ee_x(p,p)\ee_{xx}^h(p,p)(1-\ee_x(p,p))=0$, as $\ee_x(p,p)\in\{0,1\}$. Therefore $\H_x = \H_x^y = \H_y^x = \H_y$. Evaluating $\ee_{xx}^g \ee_{xx}^h = \ee_{xx}^g(x,x) \ee_{xx}^{gh}$ on $(y,y)$ shows $\chi_{x}^{y}$ to be an \R-valued character on $\H_x $. Moreover, for any $z\in\supf\ee_x$, $\ee_{xx}^h(p,p)=\chi_x^y(h)\ee_{yy}^h(p,p)=\chi_x^y(h)\chi_y^z(h)\ee_{zz}^h(p,p)$ for all $p$, hence $\chi_{x}^{y}\chi_{y}^{z} = \chi_{x}^{z}$.

Finally, the isomorphism $\Dalg_{\ee_x}\cong I(\supf\ee_x,\R)$ of \Cref{lem:iso:corner} preserves diagonal entries, so the characters $\chi_y$ of \Cref{lem:iso:struct} attached to the isomorphism $I(\supf\ee_x,\R)\cong\RH_x$ are the $\chi_x^y$, $y\in\supf\ee_x$. By that lemma they are pairwise distinct and, as $|\supf\ee_x|=|\H_x|=|\hH_x|$ (\Cref{lem:ipr:iso:comm:fin:grpalg}), $y\mapsto\chi_x^y$ is a bijection onto $\hH_x$. It follows that for any $p,y\in\supf\ee_x$, there exists a unique $q \in \supf\ee_x$ such that $\chi_{x}^{y} \chi_{x}^{p} = \chi_{x}^q$. Since $\chi_{p}^{x} = (\chi_{x}^{p})^{-1}$,
\[
\chi_{x}^{y} = \chi_{p}^{x} \chi_{x}^{p} \chi_{x}^{y} = \chi_{p}^{x} \chi_{x}^{q} = \chi_{p}^{q}.
\]
\end{proof}
Since every nonempty subposet of \PO has a minimal element, we obtain at once:
\begin{corollary}\label{cor:xmin:for:idemps}
Let $\ee$ be a primitive homogeneous idempotent and consider the corner subalgebra $\Dalg_\ee=\ee\IPR\ee$. Then $\Dalg_\ee\cong\Inc{\supf\ee}{\R}$ (cf. \Cref{lem:iso:corner}) is still \G-graded and all the conclusions of \Cref{thm:xmin} hold with \PO replaced by $\supf\ee$. In particular, $\ee$ is the preimage of $\ee_x$ in the isomorphism and $\supf\ee$ is an antichain of size $|\H_x|$, where $x$ is any element in $\supf\ee$.
\end{corollary}
\begin{proof}
Let $\QO=\supf\ee$, which is nonempty by \Cref{lem:idemp:vanishing}, and let $x$ be a minimal element of \QO. \Cref{thm:xmin} applied to the graded algebra $\IQR\cong\Dalg_\ee$ yields a primitive homogeneous idempotent $\ee_x\in\IQR$. Its preimage $\ee'$ in $\Dalg_\ee$ is a nonzero homogeneous idempotent, and $\ee'$, $\ee-\ee'$ are orthogonal homogeneous idempotents with sum \ee. Since \ee is primitive, $\ee'=\ee$, that is, $\ee_x=1_{\IQR}$. Therefore $\QO=\supf\ee_x$ is an antichain of size $|\H_x|$, every element of \QO is extremal, and \Cref{thm:xmin} applies to each of them.
\end{proof}

\begin{remark}\label{rem:xmin:supp_idemp}
In the situation of the Corollary, the support $\{h\in\G\mid\Dalg_\ee^h\neq0\}$ of the grading of $\Dalg_\ee$ is a finite abelian subgroup of \G, which we shall denote by $\H_e$. It is the group $\H_x$ of \Cref{thm:xmin} computed in $\Inc{\supf\ee}{\R}$ for any $x\in\supf\ee$, and $\Dalg_\ee\cong \Inc{\supf\ee}{\R}\cong\RH_\ee$ as \G-graded \R-algebras. Notice that all the conclusions of \Cref{lem:ipr:iso:comm:fin:grpalg} also hold with respect to this isomorphism.

\end{remark}

\begin{remark}\label{rem:ex:equals:e}
Let \ee be a diagonal primitive homogeneous idempotent, $\QO=\supf\ee$ and $x\in\QO$. Then $\ee_{xx}=\ee\ee_{xx}\ee\in\Dalg_\ee$, and the isomorphism $\Dalg_\ee\to\IQR$ of \Cref{lem:iso:corner} is the restriction $a\mapsto a|_{\QO\times\QO}$, which carries $\ee_{xx}$ and its homogeneous components $\ee_{xx}^h\in\Dalg_\ee$ to the elementary function of $x$ in \IQR and its homogeneous components. Consequently the objects $\H_x$, $\ee_x=|\H_x|\ee_{xx}^1$ and $\chi_x^y$ of \Cref{thm:xmin} computed in \IQR coincide with the ones computed in \IPR, and \Cref{cor:xmin:for:idemps} gives $\H_x=\H_\ee$ and $\ee_x=\ee$. In particular, for a diagonal primitive homogeneous idempotent the identities of \Cref{thm:xmin}\labelcref{thm:xmin:4} hold exactly: $\ee_{xx}^h=\chi_x^y(h)\ee_{yy}^h$ for $x,y\in\supf\ee$ and $h\in\H_\ee$, while $\ee_{xx}^h=0$ for $h\notin\H_\ee$. Moreover, under the isomorphism $\Dalg_\ee\cong\RH_\ee$ of \Cref{thm:xmin}\labelcref{thm:xmin:3}, which sends $|\H_\ee|\ee_{xx}^h$ to $h$, the idempotent $\ee_{yy}=\sum_{h}\chi_x^y(h)^{-1}\ee_{xx}^h$ corresponds to the Fourier idempotent $\bee_{\chi_x^y}$. In particular $\ee_{xx}$ corresponds to $\bee_{\hat1}$, and $y\mapsto\ee_{yy}$ matches the bijection $y\mapsto\chi_x^y$ of \Cref{thm:xmin}\labelcref{thm:xmin:4} with the bijection $\chi\mapsto\bee_\chi$ of \hH onto the primitive idempotents of $\RH_\ee$.
\end{remark}

\begin{proposition}\label{thm:graded:jr:corner}If $\ee$ is a primitive homogeneous idempotent, then:
\begin{enumerate}
    \item\label{thm:graded:jr:corner:1} $\Dalg_\ee\cap\zIPR = 0$,
    \item\label{thm:graded:jr:corner:2} $\JR{\Dalg_\ee}$ is a graded ideal, corresponding to $\JR\R\H_\ee$ under the graded isomorphism $\Dalg_\ee\cong\RH_\ee$ and to $\JR{\R}^{|\supf\ee|}$ under the Fourier isomorphism $\RH_\ee\cong\R^{\hH_\ee}$,
    \item\label{thm:graded:jr:corner:3} $\NR{\Dalg_\ee}$ is a graded ideal, corresponding likewise to $\NR\R\H_\ee$ and to $\NR{\R}^{|\supf\ee|}$,
\end{enumerate}
where $\H_\ee=\H_x$ for any $x\in\supf\ee$ (\Cref{cor:xmin:for:idemps,rem:xmin:supp_idemp}).
\end{proposition}
\begin{proof}
Let $\QO=\supf\ee$, so that $\Dalg_\ee\cong\IQR$ as \R-algebras (cf. \cref{lem:iso:corner}). It follows from \Cref{cor:xmin:for:idemps} that \QO is an antichain, there exists a finite abelian subgroup $\H_\ee$ of \G such that $|\H_\ee|$ is a unit in \R and $\IQR\cong \RH_\ee\cong\R^{|\QO|}$ is commutative. Since \QO is an antichain, $\ZR{\IQR}=0$. As the isomorphism $\Dalg_\ee\to\IQR$ preserves diagonal entries (\Cref{lem:iso:corner}), an element of $\Dalg_\ee\cap\zIPR$ is carried to an element of $\ZR{\IQR}$, whence \labelcref{thm:graded:jr:corner:1}. For \labelcref{thm:graded:jr:corner:2} we compute $\JR{\RH_\ee}$ through the Fourier isomorphism $\RH_\ee\to\R^{\hH_\ee}$, $a\mapsto(\sum_h a_h\chi(h))_{\chi}$, of \Cref{lem:ipr:iso:comm:fin:grpalg}. The Jacobson radical of $\R^{\hH_\ee}$ is $\JR\R^{\hH_\ee}$, so it is enough to identify the preimage of the latter. If $a_h\in\JR\R$ for all $h$, then each coordinate $\sum_h a_h\chi(h)$ lies in $\JR\R$, because the $\chi(h)$ are units. Conversely, Fourier inversion gives $a_h=|\H_\ee|^{-1}\sum_\chi\chi(h^{-1})\sum_{h'}a_{h'}\chi(h')$, so all $a_h$ lie in $\JR\R$ as soon as all coordinates do. Hence $\JR{\RH_\ee}=\JR\R\H_\ee$, which is a graded ideal, being spanned by homogeneous elements. The nilradical is treated by the same computation, the nilradical of $\R^{\hH_\ee}$ being $\NR\R^{\hH_\ee}$, and this proves \labelcref{thm:graded:jr:corner:3}. Transporting back along the graded isomorphism $\Dalg_\ee\cong\RH_\ee$ of \Cref{cor:xmin:for:idemps,rem:xmin:supp_idemp} gives the assertions for $\Dalg_\ee$.
\end{proof}

%
%

\section{The \ssets}\label{sec:stars}
In this section we define a maximal set of primitive orthogonal idempotents that we call
a \sset (\Cref{def:sset}) using terminology borrowed from \cite[Definition~1]{Miller-Spiegel}, which will play a central role in our classification.
We show that a \sset exists and is unique up to graded inner automorphisms, and show that the nonzero Peirce blocks $\ee\IPR\ff$ define a partial order on it. The conditions and notations remain the same as before.

\begin{definition}\label{def:sset}A \highlight{\sset} is any set of primitive pairwise orthogonal homogeneous idempotents whose supports partition \PO.
\end{definition}
For instance, with the trivial grading, $\{\ee_{ii}\mid i=1,\ldots,n\}$ is a \sset in $\Inc{\CO_n}{\R}$. The elements of a \sset \Eset add up to the identity: $\cc=\sum_{\ee\in\Eset}\ee$ is an idempotent with $\cc(x,x)=1$ for every $x\in\PO$, so $1-\cc$ is an idempotent lying in the nil ideal \zIPR, whence $\cc=1$.


Let $\ee\in\IPR^1$ be an idempotent. If this
element is not primitive, then there are nonzero idempotents $\ee_1,\ee_2\in\IPR^1$ such that $\ee=\ee_1+\ee_2$ and $\ee_1\ee_2=0=\ee_2\ee_1$. Since \R is indecomposable,
it must be the case that $\supf\ee_1\cup\supf\ee_2=\supf\ee$ and $\supf\ee_1\cap\supf\ee_2=\varnothing$. If \PO is finite, it follows by
induction on $|\supf\ee|$ that $\ee=\ee_1+\cdots+\ee_n$ is a
sum of primitive pairwise orthogonal idempotents. This proves the following:

\begin{proposition}\label{thm:existence:prim:idemps}
If \PO is \dPO, then $\IPR^1$ contains a finite \sset.
\end{proposition}

\begin{remark}\label{rem:sset:diag}It follows from \Cref{rem:diagonalization} that a finite \sset is simultaneously diagonalizable by an inner automorphism of \IPR. We shall sometimes assume, without loss of generality, that the elements of the \sset are diagonal.
\end{remark}

By \Cref{rem:ex:equals:e}, for a diagonal \sset \Eset and $\ee\in\Eset$ one has $\ee_{xx}\in\Dalg_\ee$ and $\ee_x=\ee$ for every $x\in\supf\ee$.

The supports of the elements of a \sset are determined by the grading alone:

\begin{lemma}\label{lem:minsup:subalg}
Let \Dalg be a graded subalgebra of \IPR with unit $1_\Dalg$ and $x\in\supf_\Dalg 1_\Dalg$. Then
$$
\minsup_\Dalg(x) = \minsup(x).
$$
In particular, if \ee is a primitive homogeneous idempotent and $\Dalg=\ee\IPR\ee$, then $\minsup(x)=\supf \ee$.
\end{lemma}
\begin{proof}
First notice that, for a homogeneous $f\in\Dalg$, we can view $f$ as a homogeneous element in \IPR, and $\supf_\Dalg f=\supf f$. Hence, $\minsup(x)\subseteq \minsup_\Dalg(x)$. On the other hand, for a homogeneous $f\in\IPR$, we have $1_\Dalg f 1_\Dalg\in \Dalg$ is still homogeneous, and $\supf_\Dalg(1_\Dalg f 1_\Dalg)\subseteq \supf 1_\Dalg\,\cap\,\supf f\subseteq \supf f$. As $1_\Dalg$ is idempotent and \R is \dR, $1_\Dalg(x,x)$ is $0$ or $1$, and it is $1$ because $x\in\supf_\Dalg 1_\Dalg$. It follows that $(1_\Dalg f1_\Dalg)(x,x)=f(x,x)$, so $1_\Dalg f1_\Dalg$ is one of the elements entering the intersection that defines $\minsup_\Dalg(x)$ whenever $f$ is one of those defining $\minsup(x)$. Therefore, $\minsup_\Dalg(x)\subseteq \minsup(x)$.

To finish the proof, we simply observe that $\Dalg_\ee$ is spanned by the homogeneous elements $|\H_\ee|\ee_{xx}^h$ ($h\in\H_\ee$), each of which is invertible in $\Dalg_\ee$ and has support equal to $\supf\ee$ (\Cref{thm:xmin}\labelcref{thm:xmin:3,thm:xmin:4} and \Cref{cor:xmin:for:idemps}). Hence $\minsup_{\Dalg_\ee}(x)=\supf\ee$.
\end{proof}

\begin{theorem}\label{lem:star:sets:conj}
Let \IPR be graded by a group \G. Then, up to conjugation by a graded inner automorphism, \IPR has essentially only one \sset.
\end{theorem}
\begin{proof}
Let \Eset and $\Eset'$ be two \ssets for \IPR. It follows from the previous lemma that there is a bijective correspondence $\ee\mapsto\ee'$ between the two \ssets where $\ee'$ is the unique idempotent in $\Eset'$ with $\supf{\ee}=\supf{\ee}'$. In particular, $|\Eset|=|\Eset'|$.

Now, consider the element $h=\sum_{\ee\in\Eset}\ee'\ee$. This element is in $\IPR^1$ and satisfies $\ee' h = \ee'\ee = h \ee$ for all $\ee\in\Eset$. Since
\[
1-h=\sum_{\ee\in\Eset}(\ee-\ee')\ee\in\zIPR\subset\jIPR,
\]
it follows that $h$ is invertible. It is clear that this element defines a graded inner automorphism of \IPR sending \Eset onto $\Eset'$, as required.
\end{proof}

Finally, the nonzero Peirce blocks define a partial order on a \sset. The following regularity property of the blocks is the key point.

\begin{lemma}\label{lem:regularity}
Let $\ee,\ff$ be diagonal primitive homogeneous idempotents with $\ee\IPR\ff\neq0$. Then
\[
\ee_{xx}\IPR\ff\neq0\quad\text{and}\quad \ee\IPR\ee_{yy}\neq0\qquad(x\in\supf\ee,\ y\in\supf\ff).
\]
Consequently, for every $y\in\supf\ff$ there is $x\in\supf\ee$ with $x\leq y$, and for every $x\in\supf\ee$ there is $y\in\supf\ff$ with $x\leq y$.
\end{lemma}
\begin{proof}
Choose a nonzero homogeneous $m\in\ee\IPR\ff$, say of degree $g$, and let $y\in\supf\ff$. By \Cref{rem:ex:equals:e} we have $\ee_{yy}\in\Dalg_\ff$, and under the isomorphism $\Dalg_\ff\cong\RH_\ff$ of \Cref{thm:xmin}\labelcref{thm:xmin:3} the element $\ee_{yy}$ is the Fourier idempotent $\bee_\chi=|\H_\ff|^{-1}\sum_{k\in\H_\ff}\chi(k^{-1})u_k$ of some $\chi\in\hH_\ff$, where $u_k\in\Dalg_\ff^k$ denotes the image of $k$ and $u_1=\ff$. Now $mu_k$ is homogeneous of degree $gk$, and the degrees $gk$ ($k\in\H_\ff$) are pairwise distinct. It follows that the terms of
\[
m\,\ee_{yy}=|\H_\ff|^{-1}\sum_{k\in\H_\ff}\chi(k^{-1})\,mu_k
\]
lie in distinct homogeneous components. If $m\ee_{yy}=0$, then $\chi(k^{-1})mu_k=0$ for every $k$, and since $\chi(k^{-1})\in\uR$ this gives $mu_k=0$ for every $k$. In particular $m=m\ff=mu_1=0$, a contradiction. Therefore $\ee\IPR\ee_{yy}\ni m\ee_{yy}\neq0$, and the argument on the left is the same. The last assertion follows from \Cref{lem:idemp:vanishing}: applied to the idempotents \ee and $\ee_{yy}$ it yields $x\in\supf\ee$ with $x\leq y$, and applied to $\ee_{xx}$ and \ff it yields $y\in\supf\ff$ with $x\leq y$.
\end{proof}

\begin{proposition}\label{lem:idemp:relation}Let the relation $\ee\preccurlyeq\ff$ on the set of nonzero idempotents of $\IPR$ be defined by $\ee\IPR\ff\neq 0$. Then this relation is reflexive, and $\ee\preccurlyeq\ff$ if and only if $x\leq y$ for some $x\in\supf\ee$ and some $y\in\supf\ff$. Furthermore, this relation restricts to a partial order on any set of pairwise orthogonal primitive homogeneous idempotents, in particular on a \sset.
\end{proposition}
\begin{proof}
First notice that, since \R is indecomposable, we have $\ee(x,x)\in\{0,1\}$ for all idempotents \ee and all $x\in\PO$. If $\ee\IPR\ee=0$, then $\ee=\ee 1 \ee=0$, which is not possible, hence $\ee\preccurlyeq \ee$ for all nonzero idempotents.

Now, for any $x, y$ in \PO, we have $(\ee\ee_{xy}\ff)(x,y)=\ee(x,x)\ee_{xy}(x,y)\ff(y,y)$, which is nonzero when $x\leq y$, $x\in\supf\ee$ and $y\in\supf\ff$. Conversely, if $\ee\preccurlyeq\ff$, there must exist $x\in\supf\ee$ and $y\in\supf\ff$ such that $x\leq y$ (\Cref{lem:idemp:vanishing}).

Let now \Eset be a set of pairwise orthogonal primitive homogeneous idempotents. By \Cref{rem:diagonalization} an inner automorphism of \IPR carries \Eset onto a set of diagonal idempotents, which are primitive and homogeneous for the transported grading. The relation $\preccurlyeq$ is defined without reference to the grading and is preserved by inner automorphisms, so we may assume that the elements of \Eset are diagonal. Being orthogonal and diagonal, they have pairwise disjoint supports.

Transitivity: let $\ee\preccurlyeq\ff$ and $\ff\preccurlyeq\hh$ in \Eset. By the previous paragraph there are $y\in\supf\ff$ and $z\in\supf\hh$ with $y\leq z$, and by \Cref{lem:regularity} there is $x\in\supf\ee$ with $x\leq y$. Then $x\leq z$, so $\ee\preccurlyeq\hh$.

Antisymmetry: let $\ee\preccurlyeq\ff$ and $\ff\preccurlyeq\ee$ with $\ee\neq\ff$. Choose $y\in\supf\ff$ and $z\in\supf\ee$ with $y\leq z$, and then $x\in\supf\ee$ with $x\leq y$ by \Cref{lem:regularity}. Since $\supf\ee$ is an antichain (\Cref{cor:xmin:for:idemps}) and $x\leq y\leq z$, we get $x=z$ and hence $x=y\in\supf\ee\cap\supf\ff$, contradicting the disjointness of the supports. It follows that $\preccurlyeq$ is a partial order on \Eset.
\end{proof}
The transitivity argument is not available for arbitrary idempotents, and indeed $\preccurlyeq$ is not transitive on them: if $a<b$ in \PO and $c$ is comparable to neither, then $\ee_{aa}\preccurlyeq\ee_{bb}+\ee_{cc}\preccurlyeq\ee_{cc}$, while $\ee_{aa}\IPR\ee_{cc}=0$.

It follows at once from this that:
\begin{corollary}\label{cor:triang}If $\ee$ and $\ff$ are distinct primitive orthogonal homogeneous idempotents, then $\ee\IPR\ff=0$ or $\ff\IPR\ee=0$.
\end{corollary}

%
%

\section{The Graded Peirce Decomposition of \IPR}\label{sec:peirce}

In this section, unless stated otherwise, \R is \dR, \PO is \dPO, and we assume that \IPR is graded by a group \G. The main result is \Cref{thm:graded:peirce} (see also \Cref{cor:grading:UTnR} and following remarks), which extends \cite[Theorem 1, Corollary 17]{Santulo-Souza-Yasumura}, and hereby follows as immediate consequence of
\Cref{lem:idemp:relation,cor:triang,thm:existence:prim:idemps,lem:star:sets:conj}. This result provides a detailed description of the graded structure on the algebra \IPR.
\begin{theorem}\label{thm:graded:peirce}
There exists a \sset \Eset in $\IPR^1$, partially ordered by relation $\preccurlyeq$ (cf. \Cref{lem:idemp:relation}), and a corresponding graded triangular Peirce decomposition of \IPR,
\begin{equation}
\IPR=\bigoplus_{\ee\in\Eset}\ee\IPR\ee\oplus\bigoplus_{\ee\prec\ff\in\Eset}\ee\IPR\ff.\label{thm:graded:peirce:decomp}
\end{equation}
Consequently, for each $\ee\in\Eset$, there exists a finite abelian subgroup $\H_\ee$ of \G such that $\Dalg_\ee=\ee\IPR\ee\cong \RH_\ee\cong\R^{|\H_\ee|}$ as graded \R-algebras (cf. \Cref{lem:ipr:iso:comm:fin:grpalg}) and $\supf\ee$ is an antichain in \PO of size $|\H_e|$.

Moreover, this decomposition is unique up to a graded inner automorphism of \IPR.
\qed
\end{theorem}
This result can be recast in the language of structural matrix algebras. More precisely:

\begin{corollary}\label{cor:struct:mat:alg}
Let $\Eset=\{\ee_1,\ldots,\ee_n\}$ be a \sset for \IPR. Then, as a \G-graded \R-algebra,
\[
\IPR \cong \left[ \begin{array}{cccc}
   \RH_1  & \mM_{12} & \cdots & \mM_{1n} \\
           & \RH_2     & \cdots & \mM_{2n} \\
           &            & \ddots & \vdots     \\
           &            &        & \RH_n
\end{array}\right]
\]
where the $\H_i$ are finite abelian subgroups of \G and $\mM_{ij}=\ee_i\IPR\ee_j$ are graded $(\RH_i,\RH_j)$-bimodules and free graded \R-modules of finite rank for all $i,j=1,\ldots,n$.\qed
\end{corollary}

This conclusion is the same as that of \cite[Theorem 1]{Santulo-Souza-Yasumura}), except that there is no restriction on the group \G, and we ask only the ring \R to be indecomposable.

\begin{remark}\label{rem:poset:character:identification}Fix two distinct idempotents $\ee,\ff$ in a diagonal \sset \Eset (\Cref{rem:sset:diag}), so that $\ee_{xx}\in\Dalg_\ee$ for $x\in\supf\ee$ (\Cref{rem:ex:equals:e}), and let $\QO=\supf\ee\,\dot\cup\,\supf\ff$. Choose base points $x_\ee\in\supf\ee$ and $x_\ff\in\supf\ff$. By \Cref{thm:xmin}\labelcref{thm:xmin:4} the maps $p\mapsto\chi_p\coloneqq\chi_{x_\ee}^p$ and $q\mapsto\chi_q\coloneqq\chi_{x_\ff}^q$ are bijections of $\supf\ee$ onto $\hH_\ee$ and of $\supf\ff$ onto $\hH_\ff$. Together they give a bijection $\QO\to\hH_\ee\dot\cup\hH_\ff$ under which $\ee_{pp}$ becomes the Fourier idempotent $\bee_{\chi_p}$ of $\Dalg_\ee\cong\RH_\ee$ (\Cref{rem:ex:equals:e}), and similarly for \ff. A change of base points translates the labels of $\supf\ee$ and of $\supf\ff$ by fixed characters, in accordance with \Cref{rem:galois:extensions}\labelcref{rem:galois:base:points}. We transport the order of \QO to $\hH_\ee\dot\cup\hH_\ff$ along this bijection. Thus, by definition,
\[
\chi_p\leq\chi_q\iff p\leq q\iff \ee_{pp}\IPR\ee_{qq}\neq 0\iff \bee_{\chi_p}\mM_{\ee\ff}\bee_{\chi_q}\neq 0,
\]
where $\mM_{\ee\ff}=\ee\IPR\ff$, the last two equivalences coming from \Cref{lem:idemp:relation} and the Peirce decomposition. The resulting labelled poset is what \Cref{sec:bimodules} describes in terms of double cosets.
\end{remark}

\begin{corollary}\label{cor:equiv:good:grading}The group grading on \IPR is isomorphic to a good grading if, and only if, $\IPR^1$ contains a \sset whose supports are singletons.
\end{corollary}
\begin{proof}
If the grading is good, then each $\ee_{xx}$ is homogeneous, necessarily of degree $1$ (an idempotent of degree $g$ satisfies $g=g^2$), and $\ee_{xx}\IPR\ee_{xx}=\R\ee_{xx}$ has no idempotents other than $0$ and $\ee_{xx}$ because \R is indecomposable. Therefore $\{\ee_{xx}\mid x\in\PO\}$ is a \sset in $\IPR^1$ whose supports are singletons. Conversely, let \Eset be such a \sset. By \Cref{rem:diagonalization} there is an inner automorphism $\iota$ of \IPR carrying \Eset onto a set of diagonal idempotents, necessarily $\{\ee_{xx}\mid x\in\PO\}$, since the supports are singletons. Transporting the grading along $\iota$ therefore makes every $\ee_{xx}$ homogeneous of degree $1$, that is, produces an isomorphic good grading. This is equivalent to the description of Miller and Spiegel in \cite[Theorem 1]{Miller-Spiegel}, our definition of a \sset reducing to theirs when the supports are singletons.
\end{proof}

The following corollary is immediate (compare with \cite[Corollary 1]{Miller-Spiegel}).
\begin{corollary}If \G is torsion-free then any group grading on \IPR is isomorphic to a good grading.
\end{corollary}

Similarly, we have the following immediate corollary first stated in \cite[Corollary 2]{Miller-Spiegel} for any bounded countable locally finite partially ordered set:
\begin{corollary}\label{cor:integers}
If $\mathcal{O}$ is the ring of integers of a number field, then every \G-grading on $\Inc{\PO}{\mathcal{O}}$ is isomorphic to a good grading.
\end{corollary}

\begin{example}
Any \G-grading of $\Inc{\PO}{\ZZ}$ must be isomorphic to a good grading due to \Cref{cor:integers}.
\end{example}

\begin{theorem}\label{cor:graded:radicals}
Let \PO be \dPO and assume that \IPR is graded by a group \G. Let \Eset be a \sset of \IPR. Then
\begin{enumerate}
\item\label{cor:graded:radicals:1} $\zIPR=\bigoplus_{\ee\prec\ff}\ee\IPR\ff$,
\item\label{cor:graded:radicals:2} $\jIPR=\bigoplus_{e\in\Eset}\JR{\ee\IPR\ee}\oplus \zIPR$,
\item\label{cor:graded:radicals:3} $\nIPR=\bigoplus_{\ee\in\Eset}\NR{\ee\IPR\ee}\oplus\zIPR$,
\end{enumerate}
are graded ideals of \IPR.
\end{theorem}
\begin{proof}
Consider the Peirce decomposition given in \Cref{thm:graded:peirce} \cref{thm:graded:peirce:decomp}.

\Cref{cor:graded:radicals:1}: Let $Z_0=\bigoplus_{\ee\prec\ff}\ee\IPR\ff$. If $\ee\prec\ff$ and $a\in\ee\IPR\ff$, then $a(x,x)=\ee(x,x)a(x,x)\ff(x,x)=0$ for all $x\in\PO$, hence $Z_0\subseteq\zIPR$. On the other hand, $\ee\zIPR\ee\subseteq\zIPR\cap\ee\IPR\ee=0$, hence $\zIPR\subseteq Z_0$. It follows that $\zIPR=Z_0$ is a graded ideal.

\Cref{cor:graded:radicals:2} and \Cref{cor:graded:radicals:3}: Since \PO is finite, $\zIPR$ is nilpotent, hence contained in $\jIPR$ and in $\nIPR$. By \Cref{lem:radicals} the quotient $\IPR/\zIPR$ is isomorphic to $\prod_{x\in\PO}\R$ by $f+\zIPR\mapsto(f(x,x))_{x\in\PO}$, whose Jacobson radical is $\prod_{x\in\PO}\JR{\R}$ and whose nilradical is $\prod_{x\in\PO}\NR{\R}$. Since $\zIPR$ lies in both radicals,
\begin{gather*}
\JR{\IPR/\zIPR}=\jIPR/\zIPR,\\
\NR{\IPR/\zIPR}=\nIPR/\zIPR.
\end{gather*}
Therefore
\[
\jIPR=\{a\in\IPR\mid a(x,x)\in\JR{\R}\text{ for all }x\in\PO\},
\]
and the same description holds for $\nIPR$ with $\NR\R$ in place of $\JR\R$.

Let $\ee\in\Eset$ and $\QO=\supf\ee$. Applying this description to the finite poset \QO and transporting it along the isomorphism $\Dalg_\ee\to\IQR$ of \Cref{lem:iso:corner}, which preserves diagonal entries and kills those outside \QO, gives
\[
\JR{\Dalg_\ee}=\{a\in\Dalg_\ee\mid a(x,x)\in\JR\R\text{ for all }x\in\QO\}=\Dalg_\ee\cap\jIPR,
\]
and likewise $\NR{\Dalg_\ee}=\Dalg_\ee\cap\nIPR$. Now let $a\in\IPR$ and write $a=\sum_{\ee\in\Eset}\ee a\ee+\sum_{\ee\prec\ff}\ee a \ff$. The supports of the idempotents in \Eset partition \PO and $\ee(x,x)=1$ for $x\in\supf\ee$, so $(\ee a\ee)(x,x)=a(x,x)$ for $x\in\supf\ee$ and $(\ee a\ee)(x,x)=0$ otherwise. The condition $a(x,x)\in\JR\R$ for all $x\in\PO$ therefore holds if and only if $\ee a\ee\in\JR{\Dalg_\ee}$ for every $\ee\in\Eset$. This proves \labelcref{cor:graded:radicals:2}, and \labelcref{cor:graded:radicals:3} follows in the same way. Both are graded ideals because each corner radical is graded by \Cref{thm:graded:jr:corner} and $\zIPR$ is graded by \labelcref{cor:graded:radicals:1}.
\end{proof}

The next \nameCref{cor:antichain:grading} partially generalizes \cite[Theorem 4]{Miller-Spiegel}.
\begin{corollary}\label{cor:antichain:grading}
Let \PO be \dPO and assume that \IPR is graded by a group \G, and let \Eset be a \sset. Then
$\IPR/\zIPR\cong \prod_{\ee\in\Eset}\RH_\ee$ as graded \R-algebras, and this quotient is isomorphic to $\R^\oP$ as an \R-algebra. The support of the grading on $\IPR/\zIPR$ is therefore the finite set $\bigcup_{\ee\in\Eset}\H_\ee$, a union of finite abelian subgroups of \G. In particular, if \PO is an antichain, then there exists a partition $\PO=\dot\cup_{i=1}^{n}\PO_i$ and corresponding finite abelian subgroups $\H_i$ of \G, such that $\Inc{\PO_i}{R}\cong\RH_i$ as graded \R-algebras, $|\PO_i|=|\H_i|$ is a unit in \R, and \R contains a primitive $\Exp\H_i$-root of the unity, for all $i$.
\end{corollary}

The next corollary was first proved in \cite[Theorem 7]{Valenti-Zaicev}  for $\UT_n\F$, where \F is a field.
\begin{corollary}\label{cor:grading:UTnR}Every group grading on $\UT_n\R$ is isomorphic to an elementary grading. Consequently, $\UT_n\R/\ZR{\UT_n\R}\cong\R^n$ always bears the trivial grading.
\end{corollary}
\begin{proof}
$\UT_n\R\cong\IPR$ where \PO is a chain of length $n$. By \Cref{thm:graded:peirce}, there exists a \sset \Eset in $\IPR^1$ whose supports must be singletons (being antichains in \PO). It follows from \Cref{cor:equiv:good:grading} that the group grading is isomorphic to a good grading.

Since \Eset is finite, this set is simultaneously diagonalizable and, up to an inner automorphism, we may assume that it consists of the elements $\ee_{11},\ldots,\ee_{nn}$. Further, we may assume that this set is totally ordered so that $\ee_{ii}\preccurlyeq\ee_{jj}\iff i\leq j$. Let
$g_1=\deg\ee_{11},\ldots,g_n=\deg\ee_{1n}$. Then, for each $i\leq j$, since $\ee_{1i}\ee_{ij}=\ee_{1j}$, we have $\deg\ee_{ij}=(\deg\ee_{1i})^{-1}\deg\ee_{1j}=g_i^{-1}g_j$. It follows that the sequence $(g_1,\ldots,g_n)$ defines an elementary \G-grading on $\UT_n\R$. The rest follows from the fact that the ideal $\zIPR$ is graded.
\end{proof}

\begin{remark}
New and interesting phenomena occur when \R is decomposable. If $\R=\R_1\oplus\cdots\oplus\R_k$ with each $\R_i$ is indecomposable (for instance, when $\R$ is artinian), then $\UT_n\R\cong\UT_n\R_1\oplus\cdots\oplus \UT_n\R_k$ and each $(\UT_n\R_i)=\ee_i(\UT_n\R)$ (for some idempotent $\ee_i\in\R$) inherits the grading. From the \namecref{cor:grading:UTnR}, each $\UT_n\R_i$ is graded by an elementary grading, and these gradings need not be given by the same sequence of group elements. It follows that the grading on $\UT_n\R$ is equivalent to a sum of (possibly distinct) elementary gradings.
\end{remark}

%
%

\section{Graded structure of the bimodules $\mM_{\ee\ff}$}\label{sec:bimodules}

The purpose of this section is to unravel the details of the graded Peirce decomposition (\cref{thm:graded:peirce:decomp}) by giving a detailed description of the structure of the \G-graded $(\Dalg_\ee,\Dalg_\ff)$-bimodules $\mM_{\ee\ff}=\ee\IPR\ff$ for a pair of primitive homogeneous orthogonal idempotents $\ee\prec\ff$. The notation is the same as that introduced in the previous sections, and we recall that there are finite abelian subgroups $\H_\ee,\H_\ff$ of \G such that $\Dalg_\ee\cong \RH_\ee$ and $\Dalg_\ff\cong \RH_\ff$ as graded \R-algebras (see \Cref{thm:xmin}).

Throughout this section \H and \K denote finite abelian subgroups of \G whose group algebras are split, $\RH\cong\R^{\oH}$ and $\RK\cong\R^{\oK}$. In the application $\H=\H_\ee$ and $\K=\H_\ff$. We put
\[
\Q=\H\times\K .
\]
An $(\RH,\RK)$-bimodule \mM is the same thing as a left \RQ-module, via $(h,k)\cdot m=hmk$, and a \G-grading on \mM is a decomposition $\mM=\bigoplus_{g\in\G}\mM^g$ into \R-submodules such that $h\mM^gk\subseteq\mM^{hgk}$ for all $h\in\H$, $k\in\K$ (equality holds, since $h$ and $k$ are units). The \highlight{support} of \mM in the grading is $\supf_\G\mM=\{g\in\G\mid\mM^g\neq0\}$. Accordingly, \Q acts on the set \G by $(h,k)\cdot g=hgk$. The orbits of this action are the double cosets $\H g\K$, and the stabilizers are the groups
\[
\Sset_g=\Sset_g(\H,\K)=\{(h,k)\in\H\times\K\mid hgk=g\}.
\]
Since \Q is finite abelian and $\RQ\cong\RH\otimes_\R\RK\cong\R^{\oQ}$ is split, \Cref{lem:ipr:iso:comm:fin:grpalg} applies to \Q and, by \Cref{lem:split:subgroups} below, to every subgroup of \Q. In particular $\hQ=\hH\times\hK$, the pair $\xi=(\chi,\rho)$ being identified with the character $\xi(h,k)=\chi(h)\rho(k)$, and the Fourier idempotent $\bee_\xi=\bee_\chi\otimes\bee_\rho\in\RQ$ acts on a bimodule as the projection $m\mapsto\bee_\chi m\bee_\rho$ onto its $(\chi,\rho)$-\highlight{Peirce component}. A graded homomorphism of graded bimodules preserves each homogeneous component. We write $\cong_\gr$ for a graded isomorphism.

To state the main result of this section, we shall need a few definitions.

\begin{definition}\label{def:extensions}
For a subgroup $\S\leq\Q$ and a character $\lambda\in\hS$, let
\[
\Ext{\S}{\lambda}=\{\xi\in\hQ\mid \xi|_{\S}=\lambda\}
\]
be the set of \highlight{extensions} of $\lambda$ to \Q. Two characters $\lambda_1\in\hat{\Sset}_{g_1}(\H,\K)$ and $\lambda_2\in\hat{\Sset}_{g_2}(\H,\K)$ are \highlight{strongly distinct} if $\Ext{\Sset_{g_1}}{\lambda_1}\cap\Ext{\Sset_{g_2}}{\lambda_2}=\varnothing$, that is, if there is no $(\chi,\rho)\in\hH\times\hK$ such that $(\chi,\rho)|_{\Sset_{g_1}}=\lambda_1$ and $(\chi,\rho)|_{\Sset_{g_2}}=\lambda_2$. A family of characters is strongly distinct if its members are pairwise strongly distinct.
\end{definition}
In other words, strongly distinct characters are not just different: they cannot be restrictions of the same character of $\H\times\K$. (When $g_1=g_2$, this simply means $\lambda_1\neq\lambda_2$.)

\begin{definition}Let \mM be an \RH-module (group \H any) and $\lambda$ a character of \H. We say that a vector $m\in\mM$ is a \highlight{weight vector of weight $\lambda$} if it satisfies: $h\cdot m = \lambda(h)m$ for all $h\in\H$. If, in addition, the \R-module $\R m$ is free of rank 1, we shall say that $m$ is a \highlight{free weight vector}.
The \R-submodule $\mM_\lambda=\{m\in\mM\mid m\text{ is a weight vector of weight }\lambda\}$ is called a \highlight{weight space of weight $\lambda$}.
\end{definition}

We shall be interested in using this concept in the case where the group \H is replaced by the finite abelian group $\Sset_g(\H,\K)$ and \mM is a \G-graded $(\RH,\RK)$-bimodule, seen as an $\Sset_g(\H,\K)$-module by restriction of the $\H\times\K$-action as follows:

\begin{definition}\label{def:atomic}We say that a \G-graded $(\RH,\RK)$-bimodule \mM is \highlight{atomic (of weight $\lambda$)} if it is cyclic, that is $\mM=\RH m \RK$, generated by a free weight vector $m\in\mM$ which is homogeneous of some degree $g\in\G$ and whose weight $\lambda$ is a character of $\Sset_g(\H,\K)$.
\end{definition}
Since $\H\times\K$ is abelian, every element of an atomic $(\RH,\RK)$-bimodule $\mM=\RH m \RK$ is itself a (not necessarily free) weight vector for the same character $\lambda$, namely $\mM = \mM_\lambda$, so the suffix ``of weight $\lambda$'' makes sense. Moreover, when \R is a field, an atomic $(\RH,\RK)$-bimodule is exactly the same as a graded irreducible $(\RH,\RK)$-bimodule.

\begin{definition}\label{def:types}
Fix once and for all a representative $g_C\in C$ of each double coset $C\in\H\backslash\G/\K$ and put $\Sset_C=\Sset_{g_C}$ (by \Cref{lem:stabilizer:basic} below, $\Sset_t=\Sset_{C}$ for every $t\in C$). A \highlight{type} is a pair $\theta=(C,\lambda)$ with $C\in\H\backslash\G/\K$ and $\lambda\in\hat{\Sset}_C$. We write
\[
\Types=\coprod_{C\in\H\backslash\G/\K}\hat{\Sset}_C
\]
for the set of types, and $E_\theta=\Ext{\Sset_C}{\lambda}$ for the set of extensions of a type. It follows that two types are strongly distinct precisely when their extension sets are disjoint. The \highlight{type} of an atomic bimodule $\mM=\RH m\RK$ of weight $\lambda$ (\Cref{def:atomic}) is
\[
\Theta(\mM)=(\H(\deg m)\K,\ \lambda)\in\Types,
\]
where $\lambda$ is a character of $\Sset_{\deg m}=\Sset_{\H(\deg m)\K}$. That $\Theta(\mM)$ does not depend on the choice of the generator $m$ will be shown in \Cref{thm:bimod:classification}\labelcref{thm:bimod:classification:5}.
\end{definition}
The type of an atomic bimodule does not depend on the choice of the representatives $g_C$. The models $\mA_\theta$ of \Cref{def:induced} do depend on it (\Cref{rem:representatives}).

The main result of this section gives a very detailed multiplicity-free decomposition of the \G-graded $(\R\H_\ee,\R\H_\ff)$-bimodule $\mM_{\ee\ff}$ into atomic sub-bimodules, and also a decomposition into indecomposable (not graded) bimodules. We identify $\supf\ee$ with $\hH_\ee$ and $\supf\ff$ with $\hH_\ff$ as in \Cref{rem:galois:extensions,rem:poset:character:identification}, so that $\ee_{xx}$ acts on $\mM_{\ee\ff}$ as $\bee_{\chi_x}$ and $\ee_{yy}$ as $\bee_{\chi_y}$, and we put
\[
B_{\ee\ff}=\{(\chi_x,\chi_y)\in\hH_\ee\times\hH_\ff\mid x\in\supf\ee,\ y\in\supf\ff,\ x<y\}
\]
be the set of pairs of characters corresponding to the elementary functions $\ee_{xy}\in\mM_{\ee\ff}$.

\begin{theorem}\label{teo:structure:bimod}Let $\ee\prec\ff$ be fixed, both diagonal (\Cref{rem:sset:diag}), and put $\H=\H_\ee$, $\K=\H_\ff$, $\mM=\mM_{\ee\ff}$. Then:
\begin{enumerate}
\item\label{teo:structure:bimod:1} There is a unique set of types $\Theta_{\ee\ff}\subseteq\Types$ such that
\[
\mM=\bigoplus_{\theta\in\Theta_{\ee\ff}}\mM_\theta,
\]
where each $\mM_\theta$ is an atomic sub-bimodule of type $\theta$. Explicitly, $\mM_\theta=\bigoplus_{(\chi_x,\chi_y)\in E_\theta}\R\ee_{xy}$. The types in $\Theta_{\ee\ff}$ are strongly distinct and
\[
B_{\ee\ff}=\mathop{\dot\bigcup}_{\theta\in\Theta_{\ee\ff}}E_\theta .
\]
\item\label{teo:structure:bimod:2} For each $g\in\G$, $\mM^g=\bigoplus_{\theta}\mM_\theta^g$ is a free \R-module of rank $|\{\theta\in\Theta_{\ee\ff}\mid g\in C_\theta\}|$. If $m_\theta$ is a homogeneous generator of $\mM_\theta$ of degree $g_{C_\theta}$, the elements $q\cdot m_\theta$, $\theta\in\Theta_{\ee\ff}$, $q$ ranging over a transversal of $\Q/\Sset_{C_\theta}$, form a homogeneous \R-basis of \mM consisting of free weight vectors.
\item\label{teo:structure:bimod:3} The Peirce components of \mM are $\bee_{\chi_x}\mM\bee_{\chi_y}=\ee_{xx}\mM\ee_{yy}=\R\ee_{xy}$ if $x<y$, and $0$ otherwise. Each is an indecomposable $(\RH,\RK)$-bimodule, free of rank $1$ over \R, and every character of \hH and every character of \hK occurs in \mM.
\item\label{teo:structure:bimod:4} Two atomic bimodules are graded isomorphic if and only if they have the same type, and the graded isomorphism class of $\mM_{\ee\ff}$ is determined by $\Theta_{\ee\ff}$.
\end{enumerate}
\end{theorem}

The proof is given in \Cref{subsec:proof:structure}, after the general theory of graded bimodules over split group algebras has been developed in \Cref{subsec:graded:bimodules}.

%
%

\subsection{Modules over group algebras}\label{subsect:modules:group:algebras}
All the results in this subsection are well-known and available elsewhere in the literature in the more general form of graded modules over Galois extensions admitting a normal basis. We state them here in a form which will be more convenient for our purposes.

Let \R be \dR, \H a finite abelian subgroup of \G such that $\RH\cong R^{|\H|}$, namely a Galois extension of \R with  Galois group \hH (see \Cref{lem:ipr:iso:comm:fin:grpalg}). In this situation, $|\H|$ is a unit in \R, the ring contains a primitive $\Exp(\H)$-root of the unity (\R is a splitting ring for \RH), $\hH\cong\H$, and the elements of \hH give a complete set of central orthogonal idempotents $\bee_\chi$ (for $\chi\in\hH$) for \RH. Here we study the structure of \G-graded $\RH$-modules $\mM$ which are finitely generated and projective over \R. We recall that the action of \hH on \RH by automorphisms $\alpha_\chi$ is given by $\alpha_\chi(h)=\chi(h^{-1})h$ (equivalently $\alpha_\chi(\bee_\varrho)=\bee_{\chi\varrho}$), for all $\chi\in\hH$ and all $h\in\H$.

We study certain \RH-submodules of \mM which are associated with a character $\chi\in\hH$, namely the \RH-submodules $\mM_\chi=\{m\in\mM\mid hm=\chi(h)m,\text{for all }h\in\H\}$. It is easy to see that $\mM_\chi=\bee_{\chi}\mM$. We say that $\chi$ \highlight{occurs} in \mM whenever $\mM_\chi\neq 0$, and we say that $\dim_\R\mM_\chi$ is the multiplicity of $\chi$ in \mM (this makes sense when \mM is projective over \R, because \R is indecomposable).

Now, if $m^g\in \mM^g$ is homogeneous, then $\bee_\chi m^g=0$ implies $m^g=0$ (the components of $\bee_\chi m^g$ in the distinct degrees $hg$, $h\in\H$, are the elements $\oH^{-1}\chi(h^{-1})hm^g$), hence if $m^g\neq 0$, then $\bee_{\chi}m^g\neq 0$ for all $\chi\in\hH$. In particular, we see that $\mM_\chi \cap \mM^g=0$ for all $g\in\G$ unless \hH is trivial, and that all characters of \hH must occur in \mM provided $\mM\neq0$. In that case $\mM=\bigoplus_{\chi\in\hH}\mM_{\chi}$ where all terms are nonzero.

These simple observations lead us to the following:

\begin{lemma}\label{lem:galois:graded:modules}
Let \mM be a \G-graded left \RH-module. Then the following holds:
\begin{enumerate}
\item\label{lem:galois:graded:modules:0} If \mM is free of finite rank over \R having an \R-basis of homogeneous elements, then
$\bee_{\chi}(\mM^{g})$ is free of finite rank over \R, for all $\chi\in\hH$, $g\in \G$;

\item\label{lem:galois:graded:modules:1} For each $\chi\in\hH$, we have $\bee_{\chi}(\mM^{hg}) = \bee_{\chi}(\mM^{g})$, for all $h\in\H$, $g\in \G$;
\item\label{lem:galois:graded:modules:2} The Galois action of \hH on \RH via automorphisms extends to isomorphisms of \R-modules $\alpha_\varrho\colon\bee_{\chi}(\mM^g)\to \bee_{\chi\varrho}(\mM^g)$ via $\alpha_\varrho(\bee_{\chi}m^g)=\bee_{\varrho\chi}m^g$, for all $\chi,\varrho\in\hH$, $g\in \G$;
\item\label{lem:galois:graded:modules:3} If $T=\H\backslash\G$ is a right transversal of \G by \H, then $\bee_\chi\mM=\bigoplus_{t\in T}\bee_\chi(\mM^t)$. Hence the Galois action of \hH extends to automorphisms of \mM such that $\alpha_\varrho(\bee_\chi\mM)=\bee_{\varrho\chi}\mM$ for all $\chi,\varrho\in\hH$;
\item\label{lem:galois:graded:modules:4} If $\bee_\chi(\mM^g)$ has (projective) dimension $n_g=\dim_\R \bee_\chi(\mM^g)$, then this dimension is independent of $\chi\in\hH$. Consequently,
$\dim_{\R}{\bee_{\chi}\mM}=\sum_{t\in T}n_t$ is
also independent of $\chi\in \hH$.
%

\item\label{lem:galois:graded:modules:5}Set $\mN=\bee_{\hat{1}}\mM$, where $\hat{1}\in\hH$ is the identity. Then the map $\omega\colon \RH\otimes_{\R} \mN\to\mM$ given by $\omega(\bee_{\chi}\otimes n)=\bee_{\chi}\alpha_{\chi}(n)$ is an \RH-module isomorphism.  In particular, if \mN is free over \R, then \mM is free over \RH.

\end{enumerate}
Similar statements can be made with respect to $\mM\bee_\chi$ and $(\mM^g)\bee_\chi$ if $\mM$ is a \G-graded right \RH-module.
\end{lemma}
Here it is important to notice that none of these facts need to be true if \mM is not graded, as one can easily construct counterexamples.

\begin{proof}
First notice that all the facts in \Cref{lem:ipr:iso:comm:fin:grpalg} are holding. See also \Cref{rem:galois:extensions}.

\labelcref{lem:galois:graded:modules:0}: A homogeneous \R-basis of \mM restricts to an \R-basis of each $\mM^g$, so $\mM^g$ is free of finite rank. The map $m\mapsto\bee_\chi m$ is injective on $\mM^g$ by the observation preceding the lemma, whence $\bee_\chi(\mM^g)\cong\mM^g$.

\labelcref{lem:galois:graded:modules:1}: Since $h\mM^g=\mM^{hg}$, we have $\bee_\chi\mM^{g}=\bee_{\chi}h^{-1}\mM^{hg}=\chi(h^{-1})\bee_\chi\mM^{hg}=\bee_\chi\mM^{hg}$.

\labelcref{lem:galois:graded:modules:2}: To see that the map $\bee_{\chi} m^g\mapsto \alpha_\varrho(\bee_{\chi}m^g)=\bee_{\varrho\chi}m^g$, for arbitrary $m^g\in\mM^g$, is a well-defined isomorphism of \R-modules, it suffices to observe that, for any $\varrho\in\hH$, we have $0=\bee_\varrho m^g=\oH^{-1}\sum_{h\in\H}\varrho(h^{-1})h m^g$ implies $\oH^{-1}\varrho(h^{-1})h m^g=0$ for all $h\in\H$, hence $m^g=0$.

\labelcref{lem:galois:graded:modules:3}: For $g,t\in \G$ with $\H g\neq\H t$, the degrees occurring in $\bee_\chi\mM^g$ and $\bee_\chi\mM^t$ are distinct, hence $\bee_\chi\mM^g\,\cap\,\bee_\chi\mM^t=0$. Otherwise, $\H g=\H t$, $t=hg$ for some $h\in\H$, and $\bee_\chi\mM^g=\bee_\chi\mM^{hg}=\bee_\chi\mM^t$. It follows that $\bee_\chi\mM=\bigoplus_{t\in T}\bee_\chi\mM^t$. The automorphism action on \mM is clear.

\labelcref{lem:galois:graded:modules:4}: Follows at once from \labelcref{lem:galois:graded:modules:1,lem:galois:graded:modules:2,lem:galois:graded:modules:3}.

\labelcref{lem:galois:graded:modules:5}: It is clear that $\omega$ is well defined and \RH-linear. Let $\eta\colon \mM\to \RH\otimes_\R \mN$ be given by $\eta(m)=\sum_{\chi\in\hH}\bee_{\chi}\otimes \alpha_{\chi^{-1}}(\bee_{\chi}m)$. Then $\omega\eta=id_{\mM}$ and $\eta\omega=id_{\RH\otimes_\R \mN}$. Indeed,
$$
(\omega\eta)(m)=\sum_{\chi\in\hH}\bee_{\chi}\alpha_{\chi}\bigl(\alpha_{\chi^{-1}}(\bee_{\chi}m)\bigr)=\sum_{\chi\in\hH}\bee_{\chi}\bee_{\chi}m=\sum_{\chi\in\hH}\bee_{\chi}m=m,
$$
and
$$
(\eta\omega)(\bee_{\chi}\otimes n)=\eta(\bee_{\chi}\alpha_{\chi}(n))=\sum_{\psi\in\hH}\bee_{\psi}\otimes \alpha_{\psi^{-1}}(\bee_{\psi} \bee_{\chi}\alpha_{\chi}(n))=\bee_{\chi}\otimes n.
$$
Here we note that $\alpha_{\chi}(n)$ is not equal to $\bee_{\chi}n$, unless $n$ is homogeneous.
\end{proof}

\begin{remark}\label{rem:galois:module:properties}\Cref{lem:galois:graded:modules:4} of this \namecref{lem:galois:graded:modules} is really striking: it says that \textit{every character of \hH has to occur in \mM with the same multiplicity}, and thus $\mM=\bigoplus_{\chi\in\hH}\mM_\chi$, where each term is not just nonzero but \R-isomorphic to every other term. Furthermore, if \mN is a graded \RH-submodule of \mM, then the same conclusions of this lemma apply to \mN. In particular, when $\mN=\RH m$ is cyclic and $m$ is a homogeneous element with $\R m$ free over \R, then $\{h m\mid h\in\H\}$ forms a homogeneous \R-basis for \mN, $\mN\cong\RH$ as \R-modules, and \mN is free over \R of rank  $\dim_\R\mN=|\H|$. Notice that $\{\bee_\chi m\mid\chi\in\hH\}$ is another \R-basis, albeit not homogeneous, with respect to which $\mN$ further decomposes into indecomposable \RH-submodules as $\mN=\bigoplus_{\chi\in\hH}\mN_\chi$ with $\mN_\chi = \R\bee_\chi m$.
\end{remark}

As an immediate consequence of the previous \namecref{lem:galois:graded:modules}s and this \namecref{rem:galois:module:properties}, we obtain:

\begin{proposition}\label{prop:module:structure}If \mM is a \G-graded left \RH-module that is free over \R of finite rank and has a basis of homogeneous elements, then \mM is free over \RH with $\dim_\R\mM=\oH\dim_{\RH}\mM$, and, if $n=\dim_{\RH}\mM$, from its basis we may extract elements $m_1^{g_1},\ldots,m_n^{g_n}$, such that \mM decomposes into indecomposable \RH-submodules as
$$
\mM = \bigoplus_{i=1}^n\left(\bigoplus_{\chi\in\hH}\R\bee_{\chi}m_i^{g_i} \right)
$$
where all terms are nonzero.
\end{proposition}
\begin{proof}
Let $T$ be a right transversal of \H in \G and, for $t\in T$, let $m_1^t,\ldots,m_{n_t}^t$ be the elements of the given homogeneous basis lying in $\mM^t$. By \labelcref{lem:galois:graded:modules:1,lem:galois:graded:modules:3} of \Cref{lem:galois:graded:modules}, $\bee_\chi\mM=\bigoplus_{t\in T}\bee_\chi\mM^t$ and $\bee_\chi\mM^t$ is spanned by $\bee_\chi m_1^t,\ldots,\bee_\chi m_{n_t}^t$, a basis of it by \labelcref{lem:galois:graded:modules:0}. Listing the $m_i^t$ ($t\in T$, $1\leq i\leq n_t$) as $m_1^{g_1},\ldots,m_n^{g_n}$, we conclude that $\{\bee_\chi m_i^{g_i}\}_i$ is an \R-basis of $\bee_\chi\mM$ for every $\chi$, and $\mM=\bigoplus_\chi\bee_\chi\mM$ gives the displayed decomposition. Each $\R\bee_\chi m_i^{g_i}=\bee_\chi\RH m_i^{g_i}$ is an \RH-submodule, and it is indecomposable because it is free of rank one over $\R\cong\bee_\chi\RH$. Moreover $\mM\cong\RH\otimes_\R\bee_{\hat1}\mM$ is free over \RH of rank $n$ by \Cref{lem:galois:graded:modules}\labelcref{lem:galois:graded:modules:5}, so $\dim_\R\mM=\oH n$.
\end{proof}

%
%
%
%
\subsection{Graded bimodules over split group algebras}\label{subsec:graded:bimodules}

In this subsection we classify the \G-graded $(\RH,\RK)$-bimodules, \H and \K being as fixed at the beginning of the section. The incidence algebra plays no role until \Cref{subsec:proof:structure}.

\begin{lemma}\label{lem:stabilizer:basic}
Let $g\in\G$. The $\Q$-orbit of $g$ is the double coset $\H g\K$, the coordinate projections $\Sset_g\to\H$ and $\Sset_g\to\K$ are injective, and
\[
\Sset_g\cong\H\cap g\K g^{-1}\cong\K\cap g^{-1}\H g .
\]
Moreover $\Sset_t=\Sset_g$ for every $t\in\H g\K$.
\end{lemma}
\begin{proof}
The orbit is $\{hgk\}=\H g\K$. If $(h,k)\in\Sset_g$ then $h=gk^{-1}g^{-1}$ and $k=g^{-1}h^{-1}g$, so each coordinate determines the other, and the images of the projections are $\H\cap g\K g^{-1}$ and $\K\cap g^{-1}\H g$. If $t=h_0gk_0$, then, \H and \K being abelian, $htk=t$ if and only if $hh_0gk_0k=h_0gk_0$, if and only if $hgk=g$. Hence $\Sset_t=\Sset_g$.
\end{proof}

\begin{lemma}\label{lem:split:subgroups}
For every subgroup $\S\leq\Q$ one has $\R\S\cong\R^{|\S|}$ and $\R(\Q/\S)\cong\R^{[\Q:\S]}$. Consequently, the restriction map $\hQ\to\hS$ is surjective, its kernel is $\S^\perp=\{\xi\in\hQ\mid\xi|_\S=1\}\cong\widehat{\Q/\S}$, and every $\lambda\in\hS$ has exactly $[\Q:\S]$ extensions to \Q: $\Ext{\S}{\lambda}$ is a coset of $\S^\perp$ in \hQ.
\end{lemma}
\begin{proof}
By \Cref{lem:ipr:iso:comm:fin:grpalg}, applied to \H and to \K, the integers \oH and \oK are units in \R, and \R contains primitive roots of unity of orders $\Exp(\H)$ and $\Exp(\K)$. Put $n=\oH\oK$. Every divisor $d$ of $n$ is a unit in \R, so $x^d-1$ is separable over \R and has at most $d$ roots in \R because \R is indecomposable \cite[Corollary 2.5]{janusz}. The group of $n$-th roots of unity in \R therefore has at most $d$ elements of order dividing $d$ for each $d\mid n$, and is therefore cyclic. It follows that \R contains a primitive root of unity of order $\operatorname{lcm}(\Exp\H,\Exp\K)=\Exp(\Q)$. Now $|\S|$ and $[\Q:\S]$ divide $|\Q|=n$, $\Exp(\S)$ and $\Exp(\Q/\S)$ divide $\Exp(\Q)$, and \S, $\Q/\S$ are abelian, so the three conditions of \Cref{lem:ipr:iso:comm:fin:grpalg} hold for \S and for $\Q/\S$, whence $\R\S\cong\R^{|\S|}$ and $\R(\Q/\S)\cong\R^{[\Q:\S]}$. The same lemma gives $|\hS|=|\S|$ and $|\widehat{\Q/\S}|=[\Q:\S]$. The kernel of the restriction $\hQ\to\hS$ consists of the characters of \Q that factor through $\Q/\S$, so it is $\S^\perp\cong\widehat{\Q/\S}$, of order $[\Q:\S]$. The image therefore has $|\Q|/[\Q:\S]=|\S|$ elements and is all of \hS, and each fibre is a coset of $\S^\perp$.
\end{proof}

In particular, for each stabilizer $\Sset_g$ the Fourier idempotents
\[
\bee_\lambda=\frac{1}{|\Sset_g|}\sum_{s\in\Sset_g}\lambda(s^{-1})s\in\R\Sset_g\qquad(\lambda\in\hat{\Sset}_g)
\]
form a complete set of primitive orthogonal idempotents of $\R\Sset_g$, and $s\bee_\lambda=\lambda(s)\bee_\lambda$ for $s\in\Sset_g$. Therefore every $\R\Sset_g$-module \mN is the direct sum of its weight spaces $\mN_\lambda=\bee_\lambda\mN$.

\begin{lemma}\label{lem:orbit:induction}
Let \mM be a \G-graded $(\RH,\RK)$-bimodule, let $C=\H g\K$ and put $\mM_C=\bigoplus_{t\in C}\mM^t$. Then $\mM_C$ is a graded sub-bimodule, $\mM=\bigoplus_{C\in\H\backslash\G/\K}\mM_C$, and the map
\[
\Phi_g\colon\RQ\otimes_{\R\Sset_g}\mM^g\longrightarrow\mM_C,\qquad q\otimes m\longmapsto q\cdot m,
\]
is an isomorphism of graded bimodules, where the left-hand side is graded by $\deg(q\otimes m)=q\cdot g$.
\end{lemma}
\begin{proof}
Since $C$ is a \Q-orbit, $\mM_C$ is a graded sub-bimodule and \mM is the direct sum of the $\mM_C$. The grading on the tensor product is well defined because $\Sset_g$ fixes $g$, and $\Phi_g$ is well defined because $\Phi_g(qs\otimes m)=q\cdot(s\cdot m)=\Phi_g(q\otimes s\cdot m)$ for $s\in\Sset_g$. If $T$ is a transversal of $\Q/\Sset_g$, then $\RQ=\bigoplus_{t\in T}t\,\R\Sset_g$ is free as a right $\R\Sset_g$-module, so $\RQ\otimes_{\R\Sset_g}\mM^g=\bigoplus_{t\in T}t\otimes\mM^g$. For each $t\in T$, multiplication by the unit $t$ is an \R-isomorphism $\mM^g\to\mM^{t\cdot g}$, and $t\Sset_g=t'\Sset_g$ if and only if $t\cdot g=t'\cdot g$. It follows that $\Phi_g$ is the direct sum of these isomorphisms, and it is graded.
\end{proof}

\begin{definition}\label{def:induced}
Let $\theta=(C,\lambda)$ be a type and \mP an \R-module. Write $\mP_\lambda$ for \mP with the $\R\Sset_C$-module structure $s\cdot p=\lambda(s)p$, and define the \highlight{induced bimodule}
\[
\mA_\theta(\mP)=\RQ\otimes_{\R\Sset_C}\mP_\lambda,\qquad \deg(q\otimes p)=q\cdot g_C .
\]
We put $\mA_\theta=\mA_\theta(\R)$. The degree-$g_C$ component of $\mA_\theta(\mP)$ is $1\otimes\mP\cong\mP$, since only the elements of $\Sset_C$ fix $g_C$. As an \R-module, $\mA_\theta(\mP)\cong\mP^{[\Q:\Sset_C]}$. Finally $\mA_\theta$ is atomic of weight $\lambda$, generated by the homogeneous free weight vector $1\otimes1$ of degree $g_C$, with $\supf_\G\mA_\theta=C$.
\end{definition}

\begin{theorem}\label{thm:bimod:classification}
Let \mM be a \G-graded $(\RH,\RK)$-bimodule. For a type $\theta=(C,\lambda)$ put
\[
\mP_\theta(\mM)=\bee_\lambda\mM^{g_C}\qquad\text{and}\qquad \mM_\theta=\RQ\cdot\mP_\theta(\mM)=\RH\,\mP_\theta(\mM)\,\RK .
\]
\begin{enumerate}
\item\label{thm:bimod:classification:1} $\mM=\bigoplus_{\theta\in\Types}\mM_\theta$, and $\mM_\theta\cong_\gr\mA_\theta(\mP_\theta(\mM))$ for every type $\theta$.
\item\label{thm:bimod:classification:2} For \R-modules $\mP,\mP'$ and types $\theta,\varphi$,
\[
\Homgr_{(\RH,\RK)}\bigl(\mA_\theta(\mP),\mA_\varphi(\mP')\bigr)\cong
\begin{cases}\Hom_\R(\mP,\mP'),&\theta=\varphi,\\ 0,&\theta\neq\varphi.\end{cases}
\]
\item\label{thm:bimod:classification:3} Two graded bimodules \mM and \mN are graded isomorphic if and only if $\mP_\theta(\mM)\cong_\R\mP_\theta(\mN)$ for every type $\theta$.
\item\label{thm:bimod:classification:4} \mM is finitely generated over \R if and only if only finitely many $\mP_\theta(\mM)$ are nonzero and each of them is finitely generated over \R. In that case, \mM is a finite direct sum of atomic bimodules if and only if every $\mP_\theta(\mM)$ is free of finite rank $m_\theta$, and then $\mM\cong_\gr\bigoplus_\theta\mA_\theta^{\,m_\theta}$ with the multiplicities $m_\theta$ uniquely determined.
\item\label{thm:bimod:classification:5} An atomic bimodule $\RH m\RK$ of weight $\lambda$ with $\deg m\in C$ is graded isomorphic to $\mA_{(C,\lambda)}$, and its only nonzero module $\mP_\theta$ is $\mP_{(C,\lambda)}\cong\R$. In particular, atomic bimodules are graded isomorphic if and only if they have the same type, every graded homomorphism between atomic bimodules of different types is zero, $\End^\gr_{(\RH,\RK)}(\mA_\theta)\cong\R$ and $\Aut^\gr_{(\RH,\RK)}(\mA_\theta)\cong\uR$.
\end{enumerate}
\end{theorem}
\begin{proof}
\labelcref{thm:bimod:classification:1}: By \Cref{lem:orbit:induction}, $\mM=\bigoplus_C\mM_C$ with $\mM_C\cong_\gr\RQ\otimes_{\R\Sset_C}\mM^{g_C}$. Since $\R\Sset_C$ is split, $\mM^{g_C}=\bigoplus_{\lambda\in\hat{\Sset}_C}\bee_\lambda\mM^{g_C}$ as $\R\Sset_C$-modules, and $\Sset_C$ acts on $\bee_\lambda\mM^{g_C}$ through $\lambda$. Tensoring with \RQ gives $\mM_C=\bigoplus_\lambda\RQ\cdot\bee_\lambda\mM^{g_C}$ with $\RQ\cdot\bee_\lambda\mM^{g_C}=\Phi_{g_C}(\RQ\otimes\bee_\lambda\mM^{g_C})\cong_\gr\mA_{(C,\lambda)}(\bee_\lambda\mM^{g_C})$.

\labelcref{thm:bimod:classification:2}: If the double cosets of $\theta$ and $\varphi$ differ, the supports in the grading are disjoint and every graded homomorphism is zero. Let $\theta=(C,\lambda)$, $\varphi=(C,\mu)$. A graded homomorphism $\psi\colon\mA_\theta(\mP)\to\mA_\varphi(\mP')$ maps the degree-$g_C$ component $1\otimes\mP$ into $1\otimes\mP'$, and it is determined by this restriction because $\mA_\theta(\mP)$ is generated by $1\otimes\mP$ under \Q. The restriction is an $\R\Sset_C$-linear map $\mP_\lambda\to\mP'_\mu$, and conversely every such map $f$ extends to the graded homomorphism $q\otimes p\mapsto q\otimes f(p)$. If $\lambda\neq\mu$, the idempotent $\bee_\lambda$ acts as the identity on $\mP_\lambda$ and as zero on $\mP'_\mu$, so $f=0$. If $\lambda=\mu$, then $\Sset_C$ acts by the same scalars on both sides, and every \R-linear map is $\R\Sset_C$-linear.

\labelcref{thm:bimod:classification:3}: A graded isomorphism $\mM\to\mN$ restricts to \R-isomorphisms $\mP_\theta(\mM)\to\mP_\theta(\mN)$. Conversely, given such isomorphisms, \labelcref{thm:bimod:classification:2} produces graded isomorphisms $\mM_\theta\to\mN_\theta$ whose direct sum is a graded isomorphism, by \labelcref{thm:bimod:classification:1}.

\labelcref{thm:bimod:classification:4}: Each $\mP_\theta(\mM)$ is an \R-direct summand of \mM (project onto $\mM^{g_C}$ and apply $\bee_\lambda$), and $\mM\cong_\R\bigoplus_\theta\mP_\theta(\mM)^{[\Q:\Sset_C]}$ by \labelcref{thm:bimod:classification:1}. Suppose \mM is finitely generated over \R. Replacing a finite generating set by the homogeneous components of its elements shows that only finitely many degrees, hence finitely many double cosets and finitely many types, occur. Each $\mP_\theta(\mM)$ is then finitely generated, being a direct summand of \mM. The converse is clear. If every $\mP_\theta(\mM)$ is free of finite rank $m_\theta$, then $\mM_\theta\cong_\gr\mA_\theta^{m_\theta}$, and the $m_\theta$ are determined by \mM since $\mP_\theta(\mA_\varphi)=0$ for $\varphi\neq\theta$ and $\mP_\theta(\mA_\theta)=\R$. The converse follows from \labelcref{thm:bimod:classification:5}.

\labelcref{thm:bimod:classification:5}: Let $\mM=\RH m\RK$ be atomic of weight $\lambda$, $\deg m\in C$. Replacing $m$ by $q\cdot m$ for a suitable $q\in\Q$ we may assume $\deg m=g_C$. The weight is unchanged, because \Q is abelian: $s\cdot(q\cdot m)=q\cdot(s\cdot m)=\lambda(s)\,q\cdot m$. Every element of $\mM^{g_C}$ is an \R-combination of elements $hmk$ with $hg_Ck=g_C$, that is, with $(h,k)\in\Sset_C$, each of which equals $\lambda(h,k)m$. Hence $\mM^{g_C}=\R m\cong\R$ and $\mM^{g_C}=\bee_\lambda\mM^{g_C}$. By \Cref{lem:orbit:induction}, $\mM\cong_\gr\RQ\otimes_{\R\Sset_C}\R m\cong_\gr\mA_{(C,\lambda)}$, and $\mP_\varphi(\mM)=0$ for $\varphi\neq(C,\lambda)$. The remaining assertions follow from \labelcref{thm:bimod:classification:2}, since $\Hom_\R(\R,\R)=\R$.
\end{proof}

\begin{remark}\label{def:type:atomic}
By \Cref{thm:bimod:classification}\labelcref{thm:bimod:classification:5}, the type $\Theta(\mM)$ of an atomic bimodule (\Cref{def:types}) is well defined, being independent of the generator $m$ and of the representatives $g_C$, and it determines \mM up to graded isomorphism.
\end{remark}

\begin{remark}\label{rem:representatives}
The models $\mA_\theta$ depend on the representatives: if $g_C$ is replaced by $t=q_0\cdot g_C$, the new model is obtained from the old one by the shift $q\otimes p\mapsto qq_0^{-1}\otimes p$ of the grading. The two are graded isomorphic by \labelcref{thm:bimod:classification:5}, but they are different gradings on the same bimodule. The distinction matters when several bimodules are combined into an algebra (\Cref{sec:classification}).
\end{remark}

The homogeneous free generators of an atomic bimodule are the unit multiples of the translates of any one of them:

\begin{lemma}\label{lem:atomic:bimod:hom:gens}Let $\mM=\RH m \RK$ be an atomic $(\RH,\RK)$-bimodule and let $n\in\mM$ be homogeneous with $\R n$ free over \R of rank $1$. Then there are $(h,k)\in\H\times\K$ and $a\in\R$ such that $\deg n = h(\deg m)k$ and $n = a h m k$. If moreover $n$ generates, that is $\mM=\RH n \RK$, then $a$ is a unit. In particular the homogeneous free generators of \mM are exactly the elements $a\,hmk$ with $a\in\uR$ and $(h,k)\in\H\times\K$.
\end{lemma}
\begin{proof}
Assume that $\deg m=g$, and notice that the grading on \mM is fine and supported on the double coset $\H g \K$. Since $n$ is homogeneous, we have $n\in\mM^t$ for some $t\in\H g\K$, so that $t = h g k$ for some $(h,k)\in\H\times \K$.
Now $h^{-1}nk^{-1}\in\mM^g=\R m$, so $h^{-1}nk^{-1}=a m$ for some $a\in \R$, that is $n = a h m k$. Suppose $\mM=\RH n\RK$. Then $am=h^{-1}nk^{-1}$ also generates \mM, so its translates $h_0(am)k_0$ span \mM. Those of degree $g$ are the ones with $(h_0,k_0)\in\Sset_g$, and $h_0(am)k_0=a\lambda(h_0,k_0)m$ because $m$ is a weight vector of weight $\lambda$. Comparing with $\mM^g=\R m$ gives $a\R=\R$, so $a$ is a unit. Conversely $a\,hmk$ generates \mM for every $a\in\uR$, since $m=a^{-1}h^{-1}(a\,hmk)k^{-1}$.
\end{proof}
The hypothesis that $n$ generates cannot be dropped: for $\R=\ZZ$, $\H=\K=\G=1$ and $\mM=\ZZ$ with $m=1$, the element $n=2$ is homogeneous with $\R n$ free of rank $1$, and $n=2m$ with $2$ not a unit.

\begin{corollary}[Schur's lemma for atomic bimodules]\label{lem:homo:atomic:bimodules}Let $\mM=\RH m \RK$ and $\mN=\RH n \RK$ be two atomic bimodules of weights $\lambda\in\hat{\Sset}_g(\H,\K)$ and $\sigma\in\hat{\Sset}_t(\H,\K)$, respectively, and let $\psi\colon\mM\to\mN$ be a graded $(\RH,\RK)$-bimodule homomorphism. If $\supf_\G\mM\neq\supf_\G\mN$ (that is, $t\notin \H g \K$) or $\lambda\neq\sigma$, then $\psi= 0$. Otherwise, there are $a\in\R$ and $(h,k)\in\H\times \K$ such that $\psi(m)=a h n k$, and $\psi$ is an isomorphism if and only if $a$ is a unit in \R.
\end{corollary}
\begin{proof}
The first assertion is \Cref{thm:bimod:classification}\labelcref{thm:bimod:classification:2,thm:bimod:classification:5}. If $t\in\H g\K$ and $\lambda=\sigma$, then $t=hgk$ for some $(h,k)$, so $h^{-1}nk^{-1}$ is again a homogeneous free generator of \mN, now of degree $g$ and $\psi(m)\in\mN^g=\R h^{-1}nk^{-1}$, say $\psi(m)=a\,h^{-1}nk^{-1}$. Renaming $(h,k)$ gives the displayed form. Finally $\psi$ is an isomorphism if and only if it maps the free generator $m$ to a free generator, if and only if $a\in\uR$.
\end{proof}

We turn to the Peirce components. The component $\bee_\chi\mM\bee_\rho=\{m\in\mM\mid hmk=\chi(h)\rho(k)m\text{ for all }(h,k)\in\Q\}$ is the isotypic component of \mM of character $\xi=(\chi,\rho)$. It follows that $\mM=\bigoplus_{\xi\in\hQ}\bee_\chi\mM\bee_\rho$, and an \R-linear map between two $\xi$-isotypic components is a bimodule homomorphism. Unless \H and \K are trivial, the isotypic components contain no nonzero homogeneous elements (\Cref{rem:galois:module:properties}).

\begin{proposition}\label{prop:peirce:formula}
Let $\theta=(C,\lambda)$ be a type, \mP an \R-module and $\xi=(\chi,\rho)\in\hQ$. Then
\[
\bee_\chi\,\mA_\theta(\mP)\,\bee_\rho\cong_\R
\begin{cases}\mP,&\xi\in E_\theta,\\ 0,&\xi\notin E_\theta.\end{cases}
\]
Consequently, for every graded bimodule \mM,
\[
\bee_\chi\mM\bee_\rho\cong_\R\bigoplus_{\theta\,:\,\xi\in E_\theta}\mP_\theta(\mM).
\]
\end{proposition}
\begin{proof}
Put $\S=\Sset_C$. Since $\bee_\xi\RQ=\R\bee_\xi\cong\R$, we have $\bee_\xi\mA_\theta(\mP)=\bee_\xi\RQ\otimes_{\R\S}\mP_\lambda$, and $\bee_\xi\otimes p=\bee_\xi\bee_\lambda\otimes p$ because $\bee_\lambda$ acts as the identity on $\mP_\lambda$. Using $s\bee_\xi=\xi(s)\bee_\xi$ for $s\in\S$ and the orthogonality relations in \S,
\[
\bee_\xi\bee_\lambda=\frac{1}{|\S|}\sum_{s\in\S}\lambda(s^{-1})\xi(s)\,\bee_\xi=
\begin{cases}\bee_\xi,&\xi|_\S=\lambda,\\ 0,&\xi|_\S\neq\lambda.\end{cases}
\]
Hence the component vanishes unless $\xi\in E_\theta$, in which case $p\mapsto\bee_\xi\otimes p$ is an isomorphism $\mP\to\bee_\xi\mA_\theta(\mP)$ (every element of the target has this form, and $r\bee_\xi\otimes p\mapsto rp$ is a well-defined inverse since $\bee_\xi s=\lambda(s)\bee_\xi$ for $s\in\S$). The second statement follows by applying $\bee_\xi$ to the decomposition of \Cref{thm:bimod:classification}\labelcref{thm:bimod:classification:1}.
\end{proof}

For an atomic bimodule the Peirce components can be written down explicitly.

\begin{lemma}\label{lem:echi:erho:restricted}
Let \mM be a \G-graded $(\RH,\RK)$-bimodule, let $g\in\G$ and let $T$ be a transversal of $\Sset_g=\Sset_g(\H,\K)$ in $\H\times\K$. Then, for any $m\in\mM$ and $(\chi,\rho)\in\hH\times\hK$,
$$
\bee_\chi m \bee_\rho = \frac{|\Sset_g|}{|\H||\K|}\sum_{(h,k)\in T}\chi(h^{-1})\rho(k^{-1})h (\bee_{(\chi,\rho)|_{\Sset_g}} m) k.
$$
Moreover, if $m$ is a weight vector of weight $\lambda\in\hat{\Sset}_g$, then
$$
\bee_\chi m \bee_\rho = \begin{cases}
    \dfrac{|\Sset_g|}{|\H||\K|}\sum_{(h,k)\in T}\chi(h^{-1})\rho(k^{-1})h m k, & \text{if }(\chi,\rho)|_{\Sset_g}=\lambda, \\
    0, & \text{otherwise.}
\end{cases}
$$
In the first case, if $m$ is homogeneous of degree $g$, then $\bee_\chi m \bee_\rho=0$ if and only if $m=0$.
\end{lemma}
\begin{proof}
Write $\bee_\chi m\bee_\rho=\frac{1}{|\H||\K|}\sum_{(h,k)\in\H\times\K}\chi(h^{-1})\rho(k^{-1})hmk$ and group the terms according to the cosets $(h,k)\Sset_g$. For $s=(s_1,s_2)\in\Sset_g$,
\[
\chi((hs_1)^{-1})\rho((ks_2)^{-1})\,hs_1\,m\,ks_2=\chi(h^{-1})\rho(k^{-1})\,(\chi,\rho)(s^{-1})\,h(s\cdot m)k,
\]
and $\sum_{s\in\Sset_g}(\chi,\rho)(s^{-1})\,s\cdot m=|\Sset_g|\bee_{(\chi,\rho)|_{\Sset_g}}m$. This gives the first formula, and the second follows since $\bee_\mu m=\delta_{\mu\lambda}m$ for a weight vector $m$ of weight $\lambda$. If $m$ is homogeneous of degree $g$, the terms $hmk$, $(h,k)\in T$, lie in the distinct homogeneous components $\mM^{hgk}$ (for $hgk=h'gk'$ means $(h,k)\Sset_g=(h',k')\Sset_g$), so their combination with unit coefficients vanishes only if $m=0$.
\end{proof}

\begin{corollary}\label{prop:atomic:bimod:decomposition}
Let $\mM=\RH m \RK$ be an atomic $(\RH,\RK)$-bimodule of type $\Theta(\mM)=(\H g \K,\lambda)$, and let $(\chi,\rho)\in\hH\times\hK$. Then $\bee_\chi \mM \bee_\rho=\R\bee_\chi m\bee_\rho$ if $(\chi,\rho)\in E_{\Theta(\mM)}$ and $0$ otherwise. Each nonzero $\bee_\chi\mM\bee_\rho$ is free of rank $1$ over \R, and
$$
\mM = \bigoplus_{(\chi,\rho)\in E_{\Theta(\mM)}}\R\bee_\chi m \bee_\rho
$$
is a decomposition into (non-graded) indecomposable $(\RH,\RK)$-bimodules of rank 1 over \R.
\end{corollary}
\begin{proof}
By \Cref{thm:bimod:classification}\labelcref{thm:bimod:classification:5} we may assume $\mM=\mA_\theta$ and $m=1\otimes1$. Indeed $m$ generates \mM, so \Cref{lem:atomic:bimod:hom:gens} makes it a unit multiple of a translate of $1\otimes1$, and such a translate rescales each Peirce component generator by a unit, because $\bee_\chi(hmk)\bee_\rho$ is a unit multiple of $\bee_\chi m\bee_\rho$. The isomorphism $\R\to\bee_\xi\mA_\theta$ of the proof of \Cref{prop:peirce:formula} is $r\mapsto r\bee_\xi\otimes1=r\,\bee_\chi m\bee_\rho$, so $\bee_\chi\mM\bee_\rho=\R\bee_\chi m\bee_\rho\cong\R$ when $(\chi,\rho)\in E_\theta$, and it is zero otherwise. The rest is the isotypic decomposition of \mM.
\end{proof}

The next proposition decides when the extension sets of two types meet.

\begin{proposition}\label{prop:extensions:intersection}
Let $\S,\Tset\leq\Q$ be subgroups, $\lambda\in\hS$ and $\mu\in\widehat{\Tset}$. Then $\Ext{\S}{\lambda}\cap\Ext{\Tset}{\mu}\neq\varnothing$ if and only if $\lambda|_{\S\cap\Tset}=\mu|_{\S\cap\Tset}$. In that case there is a unique $\nu\in\widehat{\S\Tset}$ extending both $\lambda$ and $\mu$, and $\Ext{\S}{\lambda}\cap\Ext{\Tset}{\mu}=\Ext{\S\Tset}{\nu}$ has $[\Q:\S\Tset]$ elements. In particular, for types $\theta=(C,\lambda)$ and $\varphi=(D,\mu)$,
\[
E_\theta\cap E_\varphi=\varnothing\iff\lambda|_{\Sset_C\cap\Sset_D}\neq\mu|_{\Sset_C\cap\Sset_D},
\]
and if $\Sset_C\cap\Sset_D=1$, then $E_\theta\cap E_\varphi\neq\varnothing$ for all types $\theta,\varphi$ based on $C$ and $D$.
\end{proposition}
\begin{proof}
A common extension has equal restrictions to $\S\cap\Tset$. Conversely, if $\lambda|_{\S\cap\Tset}=\mu|_{\S\cap\Tset}$, then $\nu(st)=\lambda(s)\mu(t)$ is well defined (if $st=s't'$ then $s'^{-1}s=t't^{-1}\in\S\cap\Tset$) and is the unique character of $\S\Tset$ extending both. By \Cref{lem:split:subgroups} it extends to \Q, and the common extensions of $\lambda$ and $\mu$ are exactly the extensions of $\nu$.
\end{proof}

\begin{definition}\label{def:multiplicity:free}
A finitely generated \G-graded $(\RH,\RK)$-bimodule \mM is \highlight{multiplicity-free} if every nonzero Peirce component $\bee_\chi\mM\bee_\rho$ is free of rank $1$ over \R. For $\xi=(\chi,\rho)\in\hQ$ let $\R_\xi$ denote the rank-one (ungraded) bimodule \R with $h\cdot r\cdot k=\chi(h)\rho(k)r$, and for $B\subseteq\hQ$ put $\mM_B=\bigoplus_{\xi\in B}\R_\xi$.
\end{definition}

\begin{theorem}\label{thm:multiplicity:free}
Let \mM be a finitely generated \G-graded $(\RH,\RK)$-bimodule.
\begin{enumerate}
\item\label{thm:multiplicity:free:1} \mM is multiplicity-free if and only if $\mM=\bigoplus_{\theta\in\Theta}\mM_\theta$ with $\mM_\theta\cong_\gr\mA_\theta$ for a finite set $\Theta\subseteq\Types$ of pairwise strongly distinct types (that is, with pairwise disjoint extension sets $E_\theta$). In that case $\Theta=\{\theta\mid\mP_\theta(\mM)\neq0\}$ is uniquely determined, $\mM_\theta=\bigoplus_{\xi\in E_\theta}\bee_\xi\mM$, and the underlying ungraded bimodule of \mM is isomorphic to $\mM_B$ with $B=\dot\bigcup_{\theta\in\Theta}E_\theta$.
\item\label{thm:multiplicity:free:2} For $B\subseteq\hQ$, the graded isomorphism classes of multiplicity-free gradings on the bimodule $\mM_B$ correspond bijectively to the finite sets $\Theta\subseteq\Types$ such that $B=\dot\bigcup_{\theta\in\Theta}E_\theta$. Distinct types with the same extension set (necessarily based on different double cosets) give non-isomorphic gradings.
\end{enumerate}
\end{theorem}
\begin{proof}
\labelcref{thm:multiplicity:free:1}: Let \mM be multiplicity-free and $\theta$ a type with $\mP_\theta(\mM)\neq0$, and choose $\xi\in E_\theta$. By \Cref{prop:peirce:formula}, $\mP_\theta(\mM)$ is a nonzero direct summand of $\bee_\xi\mM\cong\R$. A direct summand of \R is an ideal $e\R$ with $e$ idempotent, so, \R being indecomposable, $\mP_\theta(\mM)\cong\R$ and no other type $\varphi$ with $\xi\in E_\varphi$ has $\mP_\varphi(\mM)\neq0$. By \Cref{thm:bimod:classification}\labelcref{thm:bimod:classification:4}, only finitely many $\mP_\theta(\mM)$ are nonzero. This gives the decomposition with pairwise disjoint extension sets. Conversely, if $\mM=\bigoplus_\Theta\mM_\theta$ with pairwise disjoint $E_\theta$, \Cref{prop:peirce:formula} shows that every Peirce component is $\R$ or $0$. The description of $\mM_\theta$ and of the ungraded bimodule follows from \Cref{prop:peirce:formula} and the isotypic decomposition.

\labelcref{thm:multiplicity:free:2}: By \labelcref{thm:multiplicity:free:1}, a multiplicity-free graded bimodule whose underlying bimodule is $\mM_B$ is $\bigoplus_{\theta\in\Theta}\mM_\theta$ with $\dot\bigcup E_\theta=B$. Conversely, given such $\Theta$, the graded bimodule $\bigoplus_{\theta\in\Theta}\mA_\theta$ has underlying bimodule isomorphic to $\mM_B$, and transporting its grading along an ungraded isomorphism gives a multiplicity-free grading on $\mM_B$. Two such gradings are graded isomorphic if and only if their sets of types coincide, by \Cref{thm:bimod:classification}\labelcref{thm:bimod:classification:3}.
\end{proof}

\begin{remark}
Each extension set $E_\theta$ is a coset of $\Sset_C^\perp$ in \hQ. Thus a multiplicity-free grading on $\mM_{\hQ}$, with all Peirce components nonzero, amounts to a partition of $\hQ=\hH\times\hK$ into cosets of subgroups of the form $\Sset_C^\perp$, one type being attached to each coset. When \G is abelian, $\Sset_g=\{(d,d^{-1})\mid d\in\H\cap\K\}$ for all $g$, so all stabilizers coincide and a type is a coset of $\H\K$ together with a character of $\H\cap\K$. Compare \cite[Section 4]{Santulo-Souza-Yasumura}. For nonabelian \G the stabilizers vary with the double coset, and by \Cref{prop:extensions:intersection} types based on different double cosets can occur in the same multiplicity-free bimodule only when their characters differ on the intersection of the stabilizers.
\end{remark}

We conclude with the tensor product of two atomic bimodules over a common group algebra. The multiplication of a graded incidence algebra will be described in these terms in \Cref{sec:classification}. Let \L be a third finite abelian subgroup of \G with \RL split, and let $\theta=(C,\lambda)\in\Tset_{\H,\K}(\G)$ and $\theta'=(C',\lambda')\in\Tset_{\K,\L}(\G)$. Then $\mA_\theta\otimes_{\RK}\mA_{\theta'}$ is a \G-graded $(\RH,\RL)$-bimodule, with $\deg(x\otimes y)=\deg x\deg y$. Put
\begin{gather*}
\K_1=\pi_\K(\Sset_C)=\K\cap g_C^{-1}\H g_C,\qquad \K_2=\pi_\K(\Sset_{C'})=\K\cap g_{C'}\L g_{C'}^{-1},\\
t_k=g_Ckg_{C'}\quad(k\in\K),
\end{gather*}
and let $\U=\U_{\theta\theta'}\leq\H\times\L$ be the \highlight{composable subgroup}
\begin{gather*}
\U=\{(h,l)\in\H\times\L\mid \text{there is }k\in\K\text{ with }(h,k)\in\Sset_C\text{ and }(k^{-1},l)\in\Sset_{C'}\},\\
\nu(h,l)=\lambda(h,k)\lambda'(k^{-1},l).
\end{gather*}
The element $k$ is unique (\Cref{lem:stabilizer:basic}), $k\mapsto(h,l)$ is an isomorphism $\K_1\cap\K_2\to\U$, and $\nu$ is a character of \U. Indeed, $(h,k)\in\Sset_C$ and $(k^{-1},l)\in\Sset_{C'}$ amount to $h=g_Ck^{-1}g_C^{-1}$ and $l=g_{C'}^{-1}kg_{C'}$, which require $k\in\K_1\cap\K_2$ and then determine $h$ and $l$. The resulting map is a homomorphism because $\Sset_C$ and $\Sset_{C'}$ are groups, and $\nu$ is a character because $\lambda$ and $\lambda'$ are.

\begin{proposition}\label{prop:mackey}
With the above notation, $\U\leq\Sset_{t_k}(\H,\L)$ for every $k\in\K$, and there is an isomorphism of graded $(\RH,\RL)$-bimodules
\[
\mA_\theta\otimes_{\RK}\mA_{\theta'}\ \cong_\gr\ \bigoplus_{\bar k\in\K/\K_1\K_2}\ \ \bigoplus_{\substack{\lambda''\in\hat{\Sset}_{t_k}(\H,\L)\\ \lambda''|_{\U}=\nu}}\mA_{(\H t_k\L,\ \lambda'')},
\]
where $k$ denotes any representative of $\bar k$. The inner sum does not depend on the choice of the representative, because $t_{k\kappa}\in\H t_k\L$ for $\kappa\in\K_1\K_2$, so that $\Sset_{t_{k\kappa}}=\Sset_{t_k}$. Consequently, for $\theta''=(C'',\lambda'')\in\Tset_{\H,\L}(\G)$,
\begin{align*}
m(\theta,\theta',\theta'')&\coloneqq\operatorname{rank}_\R\mP_{\theta''}(\mA_\theta\otimes_{\RK}\mA_{\theta'})\\
&=\begin{cases}
|\{\bar k\in\K/\K_1\K_2\mid \H t_k\L=C''\}|,&\lambda''|_\U=\nu,\\
0,&\lambda''|_\U\neq\nu,
\end{cases}
\end{align*}
and $\Homgr_{(\RH,\RL)}(\mA_\theta\otimes_{\RK}\mA_{\theta'},\mA_{\theta''})$ is a free \R-module of rank $m(\theta,\theta',\theta'')$.
\end{proposition}
\begin{proof}
Write $x_{h,k}=(h,k)\otimes1\in\mA_\theta$ and $y_{k,l}=(k,l)\otimes1\in\mA_{\theta'}$. Therefore $x_{h,k}k'=x_{h,kk'}$, $k'y_{k,l}=y_{k'k,l}$, $x_{hh_0,kk_0}=\lambda(h_0,k_0)x_{h,k}$ for $(h_0,k_0)\in\Sset_C$ and $y_{k_0'k,ll_0'}=\lambda'(k_0',l_0')y_{k,l}$ for $(k_0',l_0')\in\Sset_{C'}$. Since $x_{h,k}\otimes y_{k',l}=x_{h,1}\otimes y_{kk',l}$, the tensor product is spanned by the elements $z_{h,k,l}=x_{h,1}\otimes y_{k,l}$, of degree $ht_kl$, and $(h,l)\cdot z_{h_1,k,l_1}=z_{hh_1,k,l_1l}$. Let $\Sset\leq\H\times\K\times\L$ be the image of the homomorphism $\Sset_C\times\Sset_{C'}\to\H\times\K\times\L$, $((h_0,k_0),(k_0',l_0'))\mapsto(h_0,k_0k_0',l_0')$. It is injective, because $(1,k_0)\in\Sset_C$ forces $k_0=1$, so $\tilde\nu(h_0,k_0k_0',l_0')=\lambda(h_0,k_0)\lambda'(k_0',l_0')$ is a well-defined character of \Sset. The relations above say that $z_{s\cdot(h,k,l)}=\tilde\nu(s)z_{h,k,l}$ for $s\in\Sset$, so $(h,k,l)\otimes1\mapsto z_{h,k,l}$ is a well-defined surjective homomorphism of graded $\R(\H\times\L)$-modules $\R(\H\times\K\times\L)\otimes_{\R\Sset}\R_{\tilde\nu}\to\mA_\theta\otimes_{\RK}\mA_{\theta'}$, the source being graded by $\deg((h,k,l)\otimes1)=ht_kl$. The source is free over \R of rank $|\H||\K||\L|/|\Sset|=|\H||\K||\L|/(|\Sset_C||\Sset_{C'}|)$. The target is free over \R of rank $(|C|/|\K|)\cdot|C'|$: the elements $x_{h,k}$, one for each coset of $\Sset_C$ in $\H\times\K$, form a homogeneous \R-basis of $\mA_\theta$, and $x_{h,k}k'=x_{h,kk'}$ is a unit multiple of a basis element. Hence, if $T$ is a set of representatives of the cosets of $\Sset_C\cdot(1\times\K)$ in $\H\times\K$, then $\mA_\theta=\bigoplus_{(h,k)\in T}x_{h,k}\RK$ with each summand free over \RK, so that $\mA_\theta$ is free as a right \RK-module of rank $|\H||\K|/(|\Sset_C||\K|)=|C|/|\K|$. Moreover $|C|=|\H||\K|/|\Sset_C|$ and $|C'|=|\K||\L|/|\Sset_{C'}|$. The ranks agree, so the surjection is an isomorphism: $\mA_\theta\otimes_{\RK}\mA_{\theta'}\cong_\gr\R(\H\times\K\times\L)\otimes_{\R\Sset}\R_{\tilde\nu}$.

The projection of \Sset to \K is $\K_1\K_2$. Hence, for a coset $\bar k\in\K/\K_1\K_2$, the elements $z_{h,k',l}$ with $k'\in\bar k$ span a graded sub-bimodule $\mN_{\bar k}$, $\mA_\theta\otimes_{\RK}\mA_{\theta'}=\bigoplus_{\bar k}\mN_{\bar k}$, and every $z_{h,k',l}$ with $k'\in\bar k$ is a unit multiple of some $z_{h',k,l'}$ (use the relations for $(h_0,k_0)\in\Sset_C$ and $(k_0',l_0')\in\Sset_{C'}$ with $k'=k_0k_0'k$). It follows that $\mN_{\bar k}=\R(\H\times\L)\cdot z_{1,k,1}$, and $z_{hh_0,k,ll_0}$ is proportional to $z_{h,k,l}$ exactly when $(h_0,1,l_0)\in\Sset$, that is, when $(h_0,l_0)\in\U$, the factor being $\tilde\nu(h_0,1,l_0)=\nu(h_0,l_0)$. Therefore $\mN_{\bar k}\cong\R(\H\times\L)\otimes_{\R\U}\R_\nu$ with $(h,l)\otimes1\mapsto z_{h,k,l}$ of degree $ht_kl$. If $(h_0,l_0)\in\U$ with $k_0$ as in the definition of \U, then $h_0g_C=g_Ck_0^{-1}$ and $g_{C'}l_0=k_0g_{C'}$, so $h_0t_kl_0=g_Ck_0^{-1}kk_0g_{C'}=t_k$, because \K is abelian. It follows that $\U\leq\Sset''\coloneqq\Sset_{t_k}(\H,\L)$. By \Cref{lem:orbit:induction}, $\mN_{\bar k}\cong_\gr\R(\H\times\L)\otimes_{\R\Sset''}\mN_{\bar k}^{t_k}$, and the degree-$t_k$ component is $\mN_{\bar k}^{t_k}=\R\Sset''\otimes_{\R\U}\R_\nu$. Its $\lambda''$-weight space, for $\lambda''\in\hat{\Sset''}$, is $\bee_{\lambda''}\R\Sset''\otimes_{\R\U}\R_\nu=\R\bee_{\lambda''}\otimes_{\R\U}\R_\nu$, which is free of rank one if $\lambda''|_\U=\nu$ and zero otherwise (the computation in the proof of \Cref{prop:peirce:formula}, with $\bee_{\lambda''}\bee_\nu=\bee_{\lambda''}$ or $0$ according as $\lambda''|_\U=\nu$ or not). As $\R\Sset''$ is split, $\mN_{\bar k}^{t_k}=\bigoplus_{\lambda''|_\U=\nu}\R_{\lambda''}$, and tensoring with $\R(\H\times\L)$ over $\R\Sset''$ gives $\mN_{\bar k}\cong_\gr\bigoplus_{\lambda''|_\U=\nu}\mA_{(\H t_k\L,\lambda'')}$ (\Cref{def:induced}, together with \Cref{thm:bimod:classification}\labelcref{thm:bimod:classification:5} to pass from the representative $t_k$ to $g_{C''}$). The formula for $m(\theta,\theta',\theta'')$ follows from \Cref{thm:bimod:classification}\labelcref{thm:bimod:classification:4}, and the last assertion from \labelcref{thm:bimod:classification:2}.
\end{proof}

\begin{remark}
When \G is abelian, all stabilizers are ``diagonal'', $\Sset_g(\H,\K)=\{(d,d^{-1})\mid d\in\H\cap\K\}$, so that $\K_1=\H\cap\K$, $\K_2=\K\cap\L$, $\U=\{(d,d^{-1})\mid d\in\H\cap\K\cap\L\}$, and $\H t_k\L=t_k\H\L$ depends on $k$ only through its class modulo $\K\cap\H\L$. The channels $\bar k\in\K/(\H\cap\K)(\K\cap\L)$ mapping to the same coset of $\H\L$ are the ones responsible for multiplicities $m(\theta,\theta',\theta'')>1$. In \Cref{ex:octahedron} there are two of them.
\end{remark}

%
%
\subsection{Proof of \Cref{teo:structure:bimod}}\label{subsec:proof:structure}

Let $\ee\prec\ff$ be diagonal, $\H=\H_\ee$, $\K=\H_\ff$ and $\mM=\mM_{\ee\ff}=\ee\IPR\ff$. Then \mM is the free \R-module with basis $\{\ee_{xy}\mid x\in\supf\ee,\ y\in\supf\ff,\ x<y\}$, and is in particular finitely generated over \R. Under the identifications $\supf\ee=\hH$, $\supf\ff=\hK$ of \Cref{rem:poset:character:identification}, the idempotent $\ee_{xx}$ acts on \mM as $\bee_{\chi_x}$ and $\ee_{yy}$ as $\bee_{\chi_y}$, so the Peirce components of \mM are
\[
\bee_{\chi_x}\mM\bee_{\chi_y}=\ee_{xx}\,\ee\IPR\ff\,\ee_{yy}=\ee_{xx}\IPR\ee_{yy}=
\begin{cases}\R\ee_{xy},& x<y,\\ 0,&\text{otherwise},\end{cases}
\]
which proves \labelcref{teo:structure:bimod:3} except for the last assertion, and shows that \mM is multiplicity-free with $\{\xi\in\hQ\mid\bee_\xi\mM\neq0\}=B_{\ee\ff}$. \Cref{thm:multiplicity:free}\labelcref{thm:multiplicity:free:1} now gives a unique finite set $\Theta_{\ee\ff}$ of pairwise strongly distinct types with $\mM=\bigoplus_{\theta\in\Theta_{\ee\ff}}\mM_\theta$, $\mM_\theta\cong_\gr\mA_\theta$ atomic of type $\theta$, $\mM_\theta=\bigoplus_{\xi\in E_\theta}\bee_\xi\mM=\bigoplus_{(\chi_x,\chi_y)\in E_\theta}\R\ee_{xy}$ and $B_{\ee\ff}=\dot\bigcup_\theta E_\theta$. This is \labelcref{teo:structure:bimod:1}. Every element of $\hH$ (resp. \hK) occurs as a first (resp. second) coordinate of some pair in $B_{\ee\ff}$. Indeed, given $\chi\in\hH$ and a type $\theta=(C,\lambda)\in\Theta_{\ee\ff}$, the rule $k\mapsto\lambda(h,k)\chi(h)^{-1}$, $(h,k)\in\Sset_C$, is a well-defined character of the image of the injective projection $\Sset_C\to\K$ (\Cref{lem:stabilizer:basic}). It extends to a character $\rho$ of \K (\Cref{lem:split:subgroups}), so that $(\chi,\rho)\in E_\theta\subseteq B_{\ee\ff}$. Therefore every character of \hH and of \hK occurs in \mM, completing \labelcref{teo:structure:bimod:3}. For \labelcref{teo:structure:bimod:2}, $\mM^g=\bigoplus_\theta\mM_\theta^g$ and $\mM_\theta^g\cong\mA_\theta^g$ is $\R$ if $g\in C_\theta$ and $0$ otherwise. If $m_\theta$ is a homogeneous generator of $\mM_\theta$ of degree $g_{C_\theta}$, then $\mM_\theta=\bigoplus_{q\in T_\theta}\R\,q\cdot m_\theta$ for a transversal $T_\theta$ of $\Q/\Sset_{C_\theta}$ (\Cref{lem:orbit:induction}), and each $q\cdot m_\theta$ is a homogeneous free weight vector. Finally, \labelcref{teo:structure:bimod:4} is \Cref{thm:bimod:classification}\labelcref{thm:bimod:classification:3,thm:bimod:classification:5}. \qed

\begin{remark}\label{rem:structure:consequences}
Two remarks on \Cref{teo:structure:bimod} are in order. First, the decomposition into atomic sub-bimodules is unique as a decomposition, not merely up to permutation: $\mM_\theta$ is the span of the $\ee_{xy}$ with $(\chi_x,\chi_y)\in E_\theta$. The relation $x<y$ between $\supf\ee$ and $\supf\ff$ is thus partitioned into classes, one for each type, each class being a coset of $\Sset_C^\perp$ in $\hH\times\hK$ with $|\Sset_C^\perp|=[\Q:\Sset_C]=|\H g\K|$ elements. These classes are described by the notation of \Cref{def:H:phi}. Indeed $\H_{xy}$ is the double coset $C_\theta$ of the type of the summand containing $\ee_{xy}$, and two pairs $(x,y)$ and $(p,q)$ lie in the same class precisely when $\H_{xy}^{pq}\neq\varnothing$. The maps $\phi_{xy}^{pq}$ are the transition functions $\ee_{xy}^g=\phi_{xy}^{pq}(g)\ee_{pq}^g$ between the elementary functions of one class, and involve no base point. Second, $\Theta_{\ee\ff}$ determines the bimodule $\mM_{\ee\ff}$ up to graded isomorphism but not its grading: by \Cref{thm:multiplicity:free}\labelcref{thm:multiplicity:free:2} the gradings of type $\Theta_{\ee\ff}$ on $\mM_{\ee\ff}$ are conjugate under the group $\prod_{\xi\in B_{\ee\ff}}\uR$ of bimodule automorphisms $\ee_{xy}\mapsto u_{xy}\ee_{xy}$, which need not extend to automorphisms of \IPR. The gradings of the various blocks are tied together by the multiplication $\mM_{\ee\ff}\mM_{\ff\hh}\subseteq\mM_{\ee\hh}$. This is the subject of \Cref{sec:classification}.
\end{remark}

\subsection{Further properties of the bimodules $\mM_{\ee\ff}$}

We draw some consequences of \Cref{teo:structure:bimod} for the poset \PO.

The next result gives a necessary condition for the existence of a group grading on \IPR which depends directly on the partially ordered set. Recall that a \highlight{bipartite graph} is a graph whose vertex set can be divided into two disjoint sets such that no two vertices within the same set are adjacent. A bipartite graph is \highlight{biregular} if all the vertices within the same subset have the same \highlight{degree} (number of adjacent edges).

\begin{theorem}\label{thm:biregular:graph}Let $\ee\prec\ff$ be two primitive homogeneous orthogonal idempotents.
Then, as a subpartially ordered subset of \PO seen as a graph, the Hasse diagram of $\supf\ee\,\cup\supf\ff$ is a biregular bipartite graph.
\end{theorem}
\begin{proof}
Let $\QO=\supf\ee\,\cup\supf\ff$ be considered as a subpartially ordered set of \PO. Since $\ee\neq\ff$, this union is disjoint and the elements in each part are not related (see \Cref{cor:xmin:for:idemps,rem:xmin:supp_idemp}), hence the Hasse diagram of \QO is in fact a bipartite graph.

Also, let $\Dalg_\ee=\ee\IPR\ee$, $\Dalg_\ff=\ff\IPR\ff$, and let $\mM=\ee\IPR\ff$, which is a \G-graded $(\Dalg_\ee,\Dalg_\ff)$-bimodule. Here, it is useful to recall that $\Dalg_\ee\cong\RH_\ee$ and $\Dalg_\ff\cong\RH_\ff$ (cf. \Cref{rem:xmin:supp_idemp}) for certain finite abelian subgroups $\H_\ee$ and $\H_\ff$ of \G. Without loss of generality, we may assume that $\ee$ and $\ff$ are diagonal (cf. \Cref{rem:diagonalization}).

For each $(x,y)\in\supf\ee\times\supf\ff$ with $x\leq y$, let $\bee_x\in \RH_\ee$ be the primitive (central) idempotent corresponding to $x$, and similarly for $\bee_y\in\RH_\ff$ (the preimages of these elements in \IPR are $\ee_{xx}$ and $\ee_{yy}$, respectively).
Consider the left $\RH_\ee$-submodule $\mM_x=\bee_x\mM$ and the right $\RH_\ff$-submodule $\mM^y=\mM\bee_y$ of \mM. Since \mM is free over \R, these modules are projective and have a finite dimension. It follows from \Cref{lem:galois:graded:modules}\labelcref{lem:galois:graded:modules:4} that the ranks of these modules are independent of $x$ and $y$.
But since $\mM_x=\bigoplus_{\substack{y\in\supf\ff \\ x\leq y}}\R\ee_{xy}$,
it follows that $|\{y\in\supf\ff\mid x\leq y\}|=\dim_{\R}\mM_x$ is constant for all $x\in\supf\ee$.
Similarly, $|\{x\in\supf\ee\mid x\leq y\}|=\dim_{\R}\mM^y$ is constant for all $y\in\supf\ff$. These are the degrees of the vertices in the Hasse diagram of \QO. Indeed a relation $x<y$ with $x\in\supf\ee$ and $y\in\supf\ff$ is a cover in \QO. An intermediate $x<z<y$ with $z\in\QO$ lies in one of the two parts, and is then comparable with $x$ or with $y$ inside that part, which is an antichain. This shows that the vertices in each part of the Hasse diagram of \QO all have the same degree, and thus this graph is biregular, as required.

\end{proof}

\begin{example}Let $\PO=\dot\cup_{k=1}^n \CO_k$ ($n>1$) be the disjoint union of increasingly long chains. The only way to partition \PO into antichains such that, taken in pairs as subpartially ordered sets of \PO, their Hasse diagram become bipartite biregular graphs is to partition into singletons. Indeed, an antichain meets a chain in at most one element, so a part meets each $\CO_k$ in at most one element and every vertex of the Hasse diagram of a pair of parts has degree $0$ or $1$. Such a pair is biregular exactly when the comparabilities between the two parts form a perfect matching or there are none. Write $\ell_k=|\CO_k|$. Exactly $\ell_k$ parts meet $\CO_k$, and two consecutive elements of $\CO_k$ lie in distinct parts which are therefore matched, so all the parts meeting $\CO_k$ have the same size and meet the same set $S_k$ of chains. If $\CO_j\in S_k$, these $\ell_k$ parts also meet $\CO_j$, whence $\ell_k\leq\ell_j$, and they are among the $\ell_j$ parts meeting $\CO_j$, so that $S_j=S_k$ and $\ell_j\leq\ell_k$ by symmetry. The lengths $\ell_1,\ldots,\ell_n$ being distinct, $\ell_j=\ell_k$ gives $j=k$, so $S_k=\{\CO_k\}$, that is, every part is a singleton. The hypothesis on the lengths is essential, as the two parts $\{x_1,y_1\}$ and $\{x_2,y_2\}$ of $\CO\dot\cup\CO'$, $\CO=\{x_1<x_2\}$ and $\CO'=\{y_1<y_2\}$, already show. Therefore, from \Cref{thm:biregular:graph,cor:equiv:good:grading} it follows that \IPR admits only good gradings.
\end{example}

Fix $\ee\prec\ff$ two primitive orthogonal homogeneous idempotents, which we may assume are diagonal, and let $\mM=\mM_{\ee\ff}$. Let $\QO=\supf(\ee)\,\cup\supf(\ff)$ as a subpartially ordered set of \PO, so that $\IQR=\Dalg_\ee\oplus\mM\oplus\Dalg_\ff$ (cf. \Cref{thm:graded:peirce}). If \mN is a nonzero graded sub-bimodule of \mM which is free over \R, let \RO be the same as \QO with a new ordering given by $x\leq y\iff \ee_{xx}\mN\ee_{yy}\neq 0$. This defines a partial order on \RO called \mN-restricted, which is a subpartial order of that on \QO. If $\ee_{xy}\in\mN$ whenever $\ee_{xx}\mN\ee_{yy}\neq 0$, then with respect to this order, we have $\IRR=\Dalg_\ee\oplus \mN\oplus\Dalg_\ff$, a graded subalgebra of \IQR.
\Cref{thm:biregular:graph} applies to \IRR and proves the following:

\begin{corollary}\label{cor:bimodule:biregular:bipartite}Let $\ee\prec\ff$ be two primitive orthogonal homogeneous idempotents. If \mN is a nonzero graded sub-bimodule of $\mM_{\ee\ff}$ which is free over \R, $\ee_{xy}\in\mN$ whenever $\ee_{xx}\mN\ee_{yy}\neq 0$, and $\RO=\supf(\ee)\,\cup\,\supf(\ff)$ is equiped with the \mN-restricted partial order, then, seen as a graph, the Hasse diagram of \RO is biregular bipartite. In particular, if  $\supf(\ee)$ or $\supf(\ff)$ is a singleton, then $\mM_{\ee\ff}$ is atomic: if, say, $\H_\ee=1$, all stabilizers $\Sset_g(\H_\ee,\H_\ff)$ are trivial, every extension set is all of $\hH_\ee\times\hH_\ff$, and \Cref{teo:structure:bimod}\labelcref{teo:structure:bimod:1} leaves room for a single type.
\end{corollary}

\begin{proposition}\label{lem:RH:free:module}Let $\ee\prec\ff$ be two primitive orthogonal homogeneous idempotents, and let $\mN\subseteq\mM_{\ee\ff}$ be a nonzero graded sub-bimodule all of whose Peirce components $\bee_\chi\mN\bee_\rho$ ($\chi\in\hH_\ee$, $\rho\in\hH_\ff$) are free over \R, as is the case for $\mN=\mM_{\ee\ff}$ and for any sum of atomic summands of $\mM_{\ee\ff}$. Then \mN is free as a left $\Dalg_\ee$-module and as a right $\Dalg_\ff$-module, and
\begin{align*}\dim_{\R} (\mN) &= \dim_\R (\Dalg_\ee) \dim_{\Dalg_\ee}(\mN)\\
&=\dim_{\Dalg_\ff}(\mN)\dim_{\R}(\Dalg_\ff).\end{align*}
\end{proposition}
\begin{proof}
By \Cref{lem:galois:graded:modules}\labelcref{lem:galois:graded:modules:5}, applied to \mN as a graded left $\RH_\ee$-module, $\mN\cong\RH_\ee\otimes_\R\bee_{\hat1}\mN$, and $\bee_{\hat1}\mN=\bigoplus_{\rho\in\hH_\ff}\bee_{\hat1}\mN\bee_\rho$ is free over \R by hypothesis. It follows that \mN is free over $\RH_\ee\cong\Dalg_\ee$ and $\dim_\R\mN=|\H_\ee|\dim_{\RH_\ee}\mN$. The right-hand version is symmetric. The Peirce components of $\mM_{\ee\ff}$, and of any sum of its atomic summands, are $\R$ or $0$ by \Cref{teo:structure:bimod}\labelcref{teo:structure:bimod:3} and \Cref{prop:peirce:formula}.
\end{proof}

\section{Graded Isomorphisms}\label{sec:isos}
In this section we describe the graded isomorphisms between graded incidence algebras and derive from them the invariants of a grading. The ring \R is still assumed to be \dR and \PO a finite partially ordered set.

The following result extends \cite[Proposition 31]{Santulo-Souza-Yasumura}:

\begin{theorem}\label{thm:isomorphism:classes}
Let \PO and \QO be two finite partially ordered sets, $\Aalg=\IPR=\bigoplus_{g\in\G}\Aalg^g$ and $\Balg=\IQR=\bigoplus_{g\in\G}\Balg^g$ be two incidence algebras graded by the same group \G, and assume that $\varphi\colon\Aalg\to\Balg$ is a graded isomorphism of \R-algebras.
Then \PO and \QO are isomorphic and, if $\Eset$ and $\Eset'$ are respective \ssets, there exists an isomorphism of partially ordered sets $\sigma\colon\Eset\to\Eset'$ such that $\ee\Aalg\ff\cong \sigma(\ee)\Balg\sigma(\ff)$ for all $\ee\preccurlyeq\ff$ in \Eset. Consequently, we have
\[
\Aalg\cong\bigoplus_{\ee\in\Eset}\RH_\ee\oplus\bigoplus_{\ee\prec\ff\in\Eset}\mM_{\ee\ff}\cong\bigoplus_{\sigma(\ee)\in\Eset'}\RH_{\sigma(\ee)}\oplus\bigoplus_{\sigma(\ee)\prec\sigma(\ff)\in\Eset'}\mM_{\sigma(\ee)\sigma(\ff)}\cong\Balg,
\]
where $\H_\ee = \H_{\sigma(\ee)}$ for all $\ee\in\Eset$.
\end{theorem}
\begin{proof}
That $\PO\cong\QO$ follows from \cite[Corollary 7.2.11]{Spiegel-ODonnell}. From \Cref{lem:star:sets:conj}, the \ssets $\varphi(\Eset)$ and $\Eset'$ are conjugate by a graded inner automorphism $\psi$ of \Balg, so $\sigma(\Eset)=\Eset'$, where $\sigma=\psi\circ\varphi$. In particular, $|\Eset|=|\Eset'|$. We have
\[
\ee\preccurlyeq\ff\in\Eset\iff 0\neq \sigma(\ee\Aalg\ff)=\sigma(\ee)\Balg\sigma(\ff)\iff \sigma(\ee)\preccurlyeq\sigma(\ff)\in\Eset',
\]
therefore $\sigma\colon\Eset\to\Eset'$ is an isomorphism of partially ordered sets. Since $\RH_{\ee}\cong \ee\Aalg\ee\cong\sigma(\ee)\Balg\sigma(\ee)\cong\RH_{\sigma(\ee)}$ as graded \R-algebras, it follows that $\H_\ee = \H_{\sigma(\ee)}$.
\end{proof}

The converse of \Cref{thm:isomorphism:classes} is false: two gradings may have isomorphic posets of idempotents, the same corner algebras and blockwise graded isomorphic bimodules $\mM_{\ee\ff}$ without being graded isomorphic (\Cref{ex:octahedron}). The blockwise isomorphisms must in addition be compatible with the multiplication $\mM_{\ee\ff}\mM_{\ff\hh}\subseteq\mM_{\ee\hh}$. The next theorem makes this precise and shows that graded isomorphisms have a rigid form. We use the identifications of \Cref{rem:galois:extensions}: for a diagonal primitive homogeneous idempotent \ee with base point $x_\ee\in\supf\ee$, the map $p\mapsto\chi_p=\chi_{x_\ee}^p$ identifies $\supf\ee$ with $\hH_\ee$, $\Dalg_\ee$ with $\RH_\ee$ (by $|\H_\ee|\ee_{x_\ee x_\ee}^h\mapsto h$), and $\ee_{pp}$ with $\bee_{\chi_p}$.

\begin{theorem}\label{thm:graded:iso:structure}
Let $\Aalg=\IPR$ and $\Balg=\IQR$ be graded by \G, with diagonal \ssets \Eset and $\Eset'$ (\Cref{rem:sset:diag}). The following are equivalent:
\begin{enumerate}
\item\label{thm:graded:iso:structure:1} \Aalg and \Balg are isomorphic as graded \R-algebras.
\item\label{thm:graded:iso:structure:2} There exist an isomorphism of partially ordered sets $\sigma\colon(\Eset,\preccurlyeq)\to(\Eset',\preccurlyeq)$ and \R-linear isomorphisms $\psi_\ee\colon\Dalg_\ee\to\Dalg_{\sigma(\ee)}$ ($\ee\in\Eset$) and $\psi_{\ee\ff}\colon\mM_{\ee\ff}\to\mM_{\sigma(\ee)\sigma(\ff)}$ ($\ee\prec\ff$ in \Eset), all preserving degrees, such that $\psi=\bigoplus_\ee\psi_\ee\oplus\bigoplus_{\ee\prec\ff}\psi_{\ee\ff}$ is multiplicative: $\psi(xy)=\psi(x)\psi(y)$ for all $x\in\ee\Aalg\ff$, $y\in\ff\Aalg\hh$ and $\ee\preccurlyeq\ff\preccurlyeq\hh$ in \Eset.
\item\label{thm:graded:iso:structure:3} There exist an isomorphism of partially ordered sets $\tau\colon\PO\to\QO$ and units $u_{pq}\in\uR$ ($p<q$ in \PO) such that
\begin{enumerate}[label=\textup{(\alph*)}]
\item $u_{pq}u_{qr}=u_{pr}$ whenever $p<q<r$;
\item $\tau$ carries support blocks to support blocks: for every $\ee\in\Eset$ there is a unique $\sigma(\ee)\in\Eset'$ with $\tau(\supf\ee)=\supf\sigma(\ee)$. The resulting map $\sigma\colon\Eset\to\Eset'$ is a bijection, and $\H_\ee=\H_{\sigma(\ee)}$, and in the labellings by characters $\tau|_{\supf\ee}$ is a translation: there is $t_\ee\in\hH_\ee$ with $\chi_{\tau(p)}=t_\ee\chi_p$ for all $p\in\supf\ee$;
\item the \R-algebra isomorphism $\varphi_{\tau,u}\colon\Aalg\to\Balg$, $\ee_{pq}\mapsto u_{pq}\ee_{\tau(p)\tau(q)}$ ($u_{pp}=1$), is graded.
\end{enumerate}
\end{enumerate}
Moreover, every graded isomorphism $\Aalg\to\Balg$ is of the form $\iota\circ\varphi_{\tau,u}$ with $\iota$ a graded inner automorphism of \Balg (conjugation by a unit of $\Balg^1$) and $(\tau,u)$ as in \labelcref{thm:graded:iso:structure:3}.
\end{theorem}
\begin{proof}
\labelcref{thm:graded:iso:structure:1}$\Rightarrow$\labelcref{thm:graded:iso:structure:3}: Let $\varphi\colon\Aalg\to\Balg$ be a graded isomorphism. Then $\varphi(\Eset)$ is a \sset of \Balg, hence conjugate to $\Eset'$ by a graded inner automorphism $\iota^{-1}$ of \Balg (\Cref{lem:star:sets:conj}). Replacing $\varphi$ by $\iota^{-1}\varphi$ we may assume $\varphi(\Eset)=\Eset'$, and we let $\sigma=\varphi|_\Eset$. As in the proof of \Cref{thm:isomorphism:classes}, $\sigma$ is an isomorphism of posets. For $\ee\in\Eset$, $\varphi$ restricts to a graded algebra isomorphism $\Dalg_\ee\to\Dalg_{\sigma(\ee)}$. Both algebras are split group algebras with their fine gradings, $\Dalg_\ee^h=\R\cdot|\H_\ee|\ee_{x_\ee x_\ee}^h$ for $h\in\H_\ee$ and $0$ otherwise. Since $\varphi$ maps nonzero components of degree $h$ onto nonzero components of degree $h$, we get $\H_\ee=\H_{\sigma(\ee)}=:\H$ and $\varphi(h)=\alpha_\ee(h)\,h$ for a function $\alpha_\ee\colon\H\to\uR$ which is multiplicative, that is, $\alpha_\ee\in\hH$. Consequently
\[
\varphi(\ee_{pp})=\varphi(\bee_{\chi_p})=\frac1{|\H|}\sum_{h\in\H}\chi_p(h^{-1})\alpha_\ee(h)\,h=\bee_{\chi_p\alpha_\ee^{-1}}=\ee_{p'p'},
\]
where $\chi_{p'}=\alpha_\ee^{-1}\chi_p$.
Thus $\varphi$ permutes the diagonal elementary functions, $\varphi(\ee_{pp})=\ee_{\tau(p)\tau(p)}$ for a bijection $\tau\colon\PO\to\QO$ which maps $\supf\ee$ onto $\supf\sigma(\ee)$ and is a translation there, by $t_\ee\coloneqq\alpha_\ee^{-1}$. We keep the two symbols apart throughout: $\alpha_\ee$ is the character of the corner automorphism and $t_\ee=\alpha_\ee^{-1}$ is the translation of the point labels. For $p\leq q$, $\varphi(\ee_{pq})=\varphi(\ee_{pp})\varphi(\ee_{pq})\varphi(\ee_{qq})\in\ee_{\tau(p)\tau(p)}\Balg\,\ee_{\tau(q)\tau(q)}=\R\ee_{\tau(p)\tau(q)}$, so $\varphi(\ee_{pq})=u_{pq}\ee_{\tau(p)\tau(q)}$ with $u_{pq}\in\R$. Applying the same argument to $\varphi^{-1}$ shows that $u_{pq}$ is a unit and that $\ee_{\tau(p)\tau(q)}\neq0$, i.e.\ $\tau(p)\leq\tau(q)$, if and only if $p\leq q$. Hence $\tau$ is an isomorphism of posets, $u_{pp}=1$, and $\ee_{pq}\ee_{qr}=\ee_{pr}$ gives $u_{pq}u_{qr}=u_{pr}$. So $\varphi=\varphi_{\tau,u}$, which is graded by assumption. This also proves the last assertion.

\labelcref{thm:graded:iso:structure:3}$\Rightarrow$\labelcref{thm:graded:iso:structure:1} is clear, and so is \labelcref{thm:graded:iso:structure:1}$\Rightarrow$\labelcref{thm:graded:iso:structure:2} (take $\psi_\ee$, $\psi_{\ee\ff}$ to be the restrictions of $\varphi_{\tau,u}$, the bijection $\sigma$ being an isomorphism of posets by the first part). Finally, assume \labelcref{thm:graded:iso:structure:2}. Since $\Aalg=\bigoplus_\ee\Dalg_\ee\oplus\bigoplus_{\ee\prec\ff}\mM_{\ee\ff}$ (\Cref{thm:graded:peirce}) and likewise for \Balg, and since $\sigma$ preserves the order, so that $\ee\prec\ff$ in \Eset if and only if $\sigma(\ee)\prec\sigma(\ff)$ in $\Eset'$ and the two decompositions are indexed compatibly, $\psi$ is a degree-preserving \R-linear bijection $\Aalg\to\Balg$. It is multiplicative on products of elements lying in blocks $\ee\Aalg\ff$, $\ff\Aalg\hh$ by hypothesis, and products $(\ee\Aalg\ff)(\ff'\Aalg\hh)$ with $\ff\neq\ff'$ vanish on both sides, because $\sigma(\ff)\sigma(\ff')=0$. It follows that $\psi$ is a graded algebra isomorphism.
\end{proof}

\begin{remark}\label{rem:sigma:poset}
The requirement that $\sigma$ preserve the order in \labelcref{thm:graded:iso:structure:2} cannot be weakened to a bijection. Let \PO be a two element antichain and \QO a two element chain, both with the trivial grading. Then $\Aalg=\IPR\cong\R\times\R$ is commutative and $\Balg=\IQR\cong\UT_2\R$ is not, so they are not isomorphic. Yet the two \ssets have two elements each, all four corners are \R, and \Eset has no pair $\ee\prec\ff$, so any bijection $\sigma$ together with the identity maps $\R\to\R$ meets every other requirement. \Cref{thm:classification} is not affected, equivalence of data requiring an isomorphism of posets by \Cref{lem:datum:equivalence}.
\end{remark}

\begin{corollary}\label{cor:graded:iso:data}
Let $\varphi\colon\Aalg\to\Balg$ be a graded isomorphism, normalised as in \Cref{thm:graded:iso:structure}, with $\sigma$, $\tau$, $(t_\ee)$ and $u$ as there. Then, for $\ee\prec\ff$ in \Eset, $\varphi$ maps each atomic summand $(\mM_{\ee\ff})_\theta$ onto an atomic summand of $\mM_{\sigma(\ee)\sigma(\ff)}$, and the induced bijection of type sets is the \emph{translation}
\[
\Theta_{\ee\ff}\longrightarrow\Theta_{\sigma(\ee)\sigma(\ff)},\qquad (C,\lambda)\longmapsto\bigl(C,\ \lambda\cdot(t_\ee,t_\ff)|_{\Sset_C}\bigr).
\]
On the generators, $\varphi(m_\theta)=s_\theta\,m_{\theta'}$ for units $s_\theta\in\uR$, one for each type. In particular the following are invariants of the graded isomorphism class of \Aalg: the poset \Eset (up to isomorphism), the subgroups $\H_\ee$, and the type sets $\Theta_{\ee\ff}$ up to translations by $\hH_\ee\times\hH_\ff$.
\end{corollary}
\begin{proof}
By \Cref{teo:structure:bimod}\labelcref{teo:structure:bimod:1}, $(\mM_{\ee\ff})_\theta=\bigoplus_{(\chi_p,\chi_q)\in E_\theta}\R\ee_{pq}$, and $\varphi(\ee_{pq})=u_{pq}\ee_{\tau(p)\tau(q)}$ with $(\chi_{\tau(p)},\chi_{\tau(q)})=(t_\ee\chi_p,t_\ff\chi_q)$. Therefore $\varphi((\mM_{\ee\ff})_\theta)=\bigoplus_{\xi\in(t_\ee,t_\ff)E_\theta}\R\ee_\xi$, and $(t_\ee,t_\ff)E_\theta=E_{\theta'}$ with $\theta'=(C,\lambda\cdot(t_\ee,t_\ff)|_{\Sset_C})$, because $E_\theta$ is a coset of $\Sset_C^\perp$ and restriction to $\Sset_C$ is multiplicative. The remaining assertions follow from \Cref{thm:bimod:classification}\labelcref{thm:bimod:classification:5}, since $\varphi$ restricts to a graded bimodule isomorphism $(\mM_{\ee\ff})_\theta\to(\mM_{\sigma\ee\sigma\ff})_{\theta'}$ up to the identifications of the corner algebras.
\end{proof}

The invariants of \Cref{cor:graded:iso:data} do not determine the graded isomorphism class. The missing invariant is the multiplication between the blocks, which is analysed in the next section. There we also decide which graded structural matrix algebras of the form displayed in \Cref{cor:struct:mat:alg} are graded isomorphic to graded incidence algebras.

The number of nonisomorphic elementary gradings on \IPR were counted in \cite{Talpo-Schutzer} when \R is a field. The situation is more complicated for good gradings, as exemplified in \cite{Jones}. Of course, we may expect it to be increasingly difficult to try and count all the possible group gradings in the general setting.

\section{The classification of \G-gradings on incidence algebras}\label{sec:classification}

In this section we identify a complete invariant of a \G-graded incidence algebra $\Aalg=\IPR$ up to graded isomorphism, and we determine which values of the invariant occur. Together these classify the \G-graded incidence algebras of finite posets over \R. \Cref{thm:graded:peirce,teo:structure:bimod} furnish the following partial invariants: the poset $(\Eset,\preccurlyeq)$ of a \sset, the subgroups $\H_\ee$, and for $\ee\prec\ff$ the set $\Theta_{\ee\ff}$ of types of the atomic summands of $\mM_{\ee\ff}$, the latter up to the translations of \Cref{cor:graded:iso:data}. By \Cref{ex:octahedron} these do not suffice: the multiplication between the blocks $\mM_{\ee\ff}$ is an independent invariant. Since the blocks are direct sums of atomic bimodules, whose graded endomorphisms are scalars, the multiplication is described by finitely many structure constants.

\begin{definition}\label{def:datum}
A \highlight{\G-graded incidence datum} $\mathcal D=(\Eset,\preccurlyeq,\H,\Theta,\mu)$ consists of
\begin{enumerate}
\item\label{def:datum:1} a finite partially ordered set $(\Eset,\preccurlyeq)$;
\item\label{def:datum:2} for each $\ee\in\Eset$, a finite abelian subgroup $\H_\ee\leq\G$ such that $\RH_\ee\cong\R^{|\H_\ee|}$;
\item\label{def:datum:3} for each $\ee\prec\ff$ in \Eset, a nonempty finite set $\Theta_{\ee\ff}\subseteq\Tset_{\H_\ee,\H_\ff}(\G)$ of pairwise strongly distinct types. We put $\mM_{\ee\ff}(\mathcal D)=\bigoplus_{\theta\in\Theta_{\ee\ff}}\mA_\theta$, a multiplicity-free \G-graded $(\RH_\ee,\RH_\ff)$-bimodule (\Cref{thm:multiplicity:free}), and $B_{\ee\ff}=\dot\bigcup_{\theta\in\Theta_{\ee\ff}}E_\theta\subseteq\hH_\ee\times\hH_\ff$;
\item\label{def:datum:4} for each chain $\ee\prec\ff\prec\hh$ in \Eset, a homomorphism of graded $(\RH_\ee,\RH_\hh)$-bimodules
\[
\mu_{\ee\ff\hh}\colon\mM_{\ee\ff}(\mathcal D)\otimes_{\RH_\ff}\mM_{\ff\hh}(\mathcal D)\longrightarrow\mM_{\ee\hh}(\mathcal D),
\]
\end{enumerate}
subject to the requirement that the \G-graded \R-module
\[
\Aalg(\mathcal D)=\bigoplus_{\ee\in\Eset}\RH_\ee\ \oplus\ \bigoplus_{\ee\prec\ff}\mM_{\ee\ff}(\mathcal D),
\]
with the multiplication given by the algebra structure of the $\RH_\ee$, the bimodule structures of the $\mM_{\ee\ff}(\mathcal D)$, the maps $\mu_{\ee\ff\hh}$, and zero on all remaining pairs of summands, be associative. Since the bimodule axioms take care of all products involving a diagonal factor, this amounts to
\[
\mu_{\ee\hh\kk}(\mu_{\ee\ff\hh}(x,y),z)=\mu_{\ee\ff\kk}(x,\mu_{\ff\hh\kk}(y,z))\qquad(\ee\prec\ff\prec\hh\prec\kk).
\]
Then $\Aalg(\mathcal D)$ is a \G-graded \R-algebra with identity $\sum_\ee1_{\RH_\ee}$, and $\{1_{\RH_\ee}\mid\ee\in\Eset\}$ is a set of pairwise orthogonal homogeneous idempotents with $1_{\RH_\ee}\Aalg(\mathcal D)1_{\RH_\ff}=\mM_{\ee\ff}(\mathcal D)$ for $\ee\prec\ff$. We call the maps $\mu_{\ee\ff\hh}$ the \highlight{structure maps} of $\mathcal D$.
\end{definition}

\begin{corollary}\label{cor:structure:constants}
Let $\ee\prec\ff\prec\hh$ in \Eset, $\theta=(C,\lambda)\in\Theta_{\ee\ff}$, $\theta'=(C',\lambda')\in\Theta_{\ff\hh}$, and let $\K_1,\K_2$, $t_k$, $\U=\U_{\theta\theta'}$ and $\nu$ be as in \Cref{prop:mackey}, with $\H=\H_\ee$, $\K=\H_\ff$, $\L=\H_\hh$. For every $k\in\K$ fix an element $q_k\in\H\times\L$ with $q_k\cdot g_{\H t_k\L}=t_k$. The same $q_k$ serves every $\theta''=(C'',\lambda'')\in\Theta_{\ee\hh}$ with $C''=\H t_k\L$, the base point $g_{C''}$ depending only on that double coset. Then the restriction of the structure map $\mu_{\ee\ff\hh}$ to $\mA_\theta\otimes_{\RK}\mA_{\theta'}$ is determined by the scalars $c_k(\theta,\theta';\theta'')\in\R$ defined by
\[
\mu_{\ee\ff\hh}\bigl(\bee_{\lambda''}(m_\theta\otimes k\cdot m_{\theta'})\bigr)=c_k(\theta,\theta';\theta'')\ q_k\cdot m_{\theta''},
\]
one for each $k\in\K$ and each $\theta''\in\Theta_{\ee\hh}$ with $C''=\H t_k\L$ and $\lambda''|_\U=\nu$, the remaining ones being zero. The double coset $\H t_k\L$ and the group $\Sset_{t_k}$ depend only on the \highlight{channel} $\bar k\in\K/\K_1\K_2$, and the scalars attached to two elements of one channel determine one another, so one $k$ per channel suffices. We write $c(\theta,\theta',\bar k;\theta'')$ for the scalar attached to a chosen representative $k$ of $\bar k$. Here $\bee_{\lambda''}\in\R\Sset_{t_k}(\H,\L)$ is the Fourier idempotent. Conversely, every family of such scalars defines a homomorphism of graded bimodules $\mA_\theta\otimes_{\RK}\mA_{\theta'}\to\mM_{\ee\hh}(\mathcal D)$. In particular, the structure maps of a datum are determined by finitely many scalars.
\end{corollary}
\begin{proof}
By \Cref{prop:mackey} and its proof, $\mA_\theta\otimes_{\RK}\mA_{\theta'}=\bigoplus_{\bar k}\mN_{\bar k}$, where $\mN_{\bar k}$ is generated by $z_k=m_\theta\otimes k\cdot m_{\theta'}$ of degree $t_k$, $\mN_{\bar k}^{t_k}=\bigoplus_{\lambda''|_\U=\nu}\R\,\bee_{\lambda''}z_k$, and $\R(\H\times\L)\cdot\bee_{\lambda''}z_k$ is atomic of type $(\H t_k\L,\lambda'')$ with homogeneous free generator $\bee_{\lambda''}z_k$ of degree $t_k$. By \Cref{thm:bimod:classification}\labelcref{thm:bimod:classification:2,thm:bimod:classification:5}, a graded bimodule homomorphism from this summand into $\mM_{\ee\hh}(\mathcal D)=\bigoplus_{\theta''\in\Theta_{\ee\hh}}\mA_{\theta''}$ is zero on all summands of a type different from $(\H t_k\L,\lambda'')$ and is a scalar multiple of the isomorphism $\bee_{\lambda''}z_k\mapsto q_k\cdot m_{\theta''}$ on the summand of that type, if present.
\end{proof}

Changing the representatives $k$ or $q_k$ multiplies the scalars by roots of unity ($\lambda''(s)$ for $s\in\Sset_{t_k}$, and the factors of \Cref{prop:mackey}). The products $\xi(q_k)c_k(\theta,\theta';\theta'')$, $\xi\in E_{\theta''}$, appearing in \Cref{prop:incidence:explicit} below do not depend on the choice of $q_k$, since replacing $q_k$ by $q_kq$ with $q\in\Sset_{t_k}$ divides $c_k$ by $\lambda''(q)$ and multiplies $\xi(q_k)$ by $\xi(q)=\lambda''(q)$.

\begin{definition}\label{def:datum:equivalence}
Two data $\mathcal D$ and $\mathcal D'$ are \highlight{equivalent} if there is an isomorphism of \G-graded \R-algebras $\psi\colon\Aalg(\mathcal D)\to\Aalg(\mathcal D')$ and a bijection $\sigma\colon\Eset\to\Eset'$ such that $\psi(\RH_\ee)=\RH'_{\sigma(\ee)}$ for all $\ee\in\Eset$.
\end{definition}

\begin{lemma}\label{lem:datum:equivalence}
An equivalence $(\psi,\sigma)$ as in \Cref{def:datum:equivalence} is the same thing as the following collection of data:
\begin{enumerate}
\item an isomorphism of posets $\sigma\colon\Eset\to\Eset'$ with $\H_\ee=\H'_{\sigma(\ee)}$ for all $\ee\in\Eset$,
\item characters $\alpha_\ee\in\hH_\ee$, giving $\psi(h)=\alpha_\ee(h)h$ on $\RH_\ee$, with label translations $t_\ee=\alpha_\ee^{-1}$ as in \Cref{thm:graded:iso:structure},
\item for each $\ee\prec\ff$, the translation $\Theta_{\ee\ff}\to\Theta'_{\sigma(\ee)\sigma(\ff)}$, $\theta=(C,\lambda)\mapsto\theta'=(C,\lambda\cdot(\alpha_\ee,\alpha_\ff)^{-1}|_{\Sset_C})$, which is the translation by $(t_\ee,t_\ff)$ of \Cref{cor:graded:iso:data} and must be a bijection,
\item units $s_\theta\in\uR$ with $\psi(m_\theta)=s_\theta m'_{\theta'}$ for the standard generators $m_\theta=1\otimes1$ of degree $g_C$,
\item the compatibility $\psi\circ\mu_{\ee\ff\hh}=\mu'_{\sigma\ee\sigma\ff\sigma\hh}\circ(\psi\otimes\psi)$ for all chains $\ee\prec\ff\prec\hh$.
\end{enumerate}
\end{lemma}
\begin{proof}
Since $1_{\RH_\ee}\Aalg(\mathcal D)1_{\RH_\ff}\neq0$ if and only if $\ee\preccurlyeq\ff$, $\sigma$ is an isomorphism of posets. The restriction of $\psi$ to $\RH_\ee$ is a graded algebra isomorphism $\RH_\ee\to\RH'_{\sigma(\ee)}$ between fine-graded split group algebras, hence $\H_\ee=\H'_{\sigma(\ee)}$ and $\psi(h)=\alpha_\ee(h)h$ for a character $\alpha_\ee$, as in the proof of \Cref{thm:graded:iso:structure}. Next $\psi$ maps $\mM_{\ee\ff}(\mathcal D)$ onto $\mM_{\sigma\ee\sigma\ff}(\mathcal D')$ and, for $q=(h,k)\in\H_\ee\times\H_\ff$, $\psi(q\cdot x)=\alpha_\ee(h)\alpha_\ff(k)\,q\cdot\psi(x)$. It follows that $\psi(m_\theta)$ is a homogeneous element of degree $g_C$ with $\R\psi(m_\theta)$ free of rank one, and $s\cdot\psi(m_\theta)=\lambda(s)(\alpha_\ee,\alpha_\ff)(s)^{-1}\psi(m_\theta)$ for $s\in\Sset_C$. It is therefore a free weight vector of weight $\lambda\cdot(\alpha_\ee,\alpha_\ff)^{-1}|_{\Sset_C}$. So $\psi$ maps $\mA_\theta$ onto the summand $\mA_{\theta'}$ of $\mM_{\sigma\ee\sigma\ff}(\mathcal D')$ (\Cref{thm:multiplicity:free}\labelcref{thm:multiplicity:free:1}). As $m_\theta$ generates $\mA_\theta$, its image generates $\mA_{\theta'}$, and both $\psi(m_\theta)$ and $m'_{\theta'}$ lie in the degree $g_C$ component of $\mA_{\theta'}$, which is free of rank one. Therefore $\psi(m_\theta)=s_\theta m'_{\theta'}$ with $s_\theta\in\uR$ by \Cref{lem:atomic:bimod:hom:gens}. The compatibility with the structure maps is the multiplicativity of $\psi$ on the off-diagonal blocks. Conversely, such a collection defines a degree-preserving \R-linear bijection $\Aalg(\mathcal D)\to\Aalg(\mathcal D')$ which is multiplicative on every pair of summands, hence an equivalence.
\end{proof}

\begin{definition}\label{def:datum:of:grading}
Let $\Aalg=\IPR$ be \G-graded, with \PO finite. Choose a diagonal \sset \Eset, base points $x_\ee\in\supf\ee$, and for each $\ee\prec\ff$ and $\theta\in\Theta_{\ee\ff}$ a homogeneous free generator $m_\theta$ of $(\mM_{\ee\ff})_\theta$ of degree $g_{C_\theta}$. Identify $\Dalg_\ee=\RH_\ee$ and $(\mM_{\ee\ff})_\theta=\mA_\theta$ via $q\cdot m_\theta\mapsto q\otimes1$ (\Cref{thm:bimod:classification}\labelcref{thm:bimod:classification:5}). The multiplication of \Aalg then induces bimodule maps $\mu_{\ee\ff\hh}\colon\mM_{\ee\ff}(\mathcal D)\otimes_{\RH_\ff}\mM_{\ff\hh}(\mathcal D)\to\mM_{\ee\hh}(\mathcal D)$, and
\[
\mathcal D(\Aalg)=(\Eset,\preccurlyeq,(\H_\ee),(\Theta_{\ee\ff}),(\mu_{\ee\ff\hh}))
\]
is a \G-graded incidence datum, the \highlight{datum of the grading}, with $\Aalg(\mathcal D(\Aalg))\cong\Aalg$ as graded algebras (\Cref{thm:graded:peirce}). By \Cref{lem:datum:equivalence}, a different choice of \sset (conjugate by a graded inner automorphism, \Cref{lem:star:sets:conj}), of base points (translations $t_\ee$) or of generators (units $s_\theta$) replaces $\mathcal D(\Aalg)$ by an equivalent datum.
\end{definition}

\begin{theorem}\label{thm:classification}
Let \PO and \QO be finite posets and let $\Aalg=\IPR$ and $\Balg=\IQR$ be \G-graded. Then \Aalg and \Balg are isomorphic as graded \R-algebras if and only if the data $\mathcal D(\Aalg)$ and $\mathcal D(\Balg)$ are equivalent. Consequently, the map $\Aalg\mapsto\mathcal D(\Aalg)$ induces a bijection between the graded isomorphism classes of \G-graded incidence algebras of finite posets over \R and the equivalence classes of \G-graded incidence data satisfying the incidence condition of \Cref{prop:incidence:condition} below.
\end{theorem}
\begin{proof}
If $\varphi\colon\Aalg\to\Balg$ is a graded isomorphism, we may assume, after composing with a graded inner automorphism of \Balg, that $\varphi$ maps the chosen \sset of \Aalg onto that of \Balg (\Cref{thm:graded:iso:structure}). Then $\varphi$ maps $\Dalg_\ee$ onto $\Dalg_{\sigma(\ee)}$, and transporting $\varphi$ through the identifications $\Aalg\cong\Aalg(\mathcal D(\Aalg))$, $\Balg\cong\Aalg(\mathcal D(\Balg))$ gives an equivalence of data. Conversely, an equivalence $\Aalg(\mathcal D(\Aalg))\to\Aalg(\mathcal D(\Balg))$ composed with these identifications is a graded isomorphism $\Aalg\to\Balg$. The last assertion follows from \Cref{prop:incidence:condition}, which characterises the data equivalent to some $\mathcal D(\Aalg)$.
\end{proof}

It remains to decide which data come from gradings on incidence algebras. Let $\mathcal D$ be a datum and put $\PO(\mathcal D)=\coprod_{\ee\in\Eset}\hH_\ee$. The Fourier idempotents $\bee_\chi$, $\chi\in\PO(\mathcal D)$, form a complete set of pairwise orthogonal idempotents of $\Aalg(\mathcal D)$. By \Cref{thm:multiplicity:free}, for $\chi\in\hH_\ee$ and $\tau\in\hH_\hh$ the Peirce component $\bee_\chi\Aalg(\mathcal D)\bee_\tau$ is $\R\bee_\chi$ if $\chi=\tau$, is free of rank one if $\ee\prec\hh$ and $(\chi,\tau)\in B_{\ee\hh}$, and is zero otherwise. Define a relation on $\PO(\mathcal D)$ by
\[
\chi\leq\tau\iff\chi=\tau\quad\text{or}\quad\ee\prec\hh\text{ and }(\chi,\tau)\in B_{\ee\hh}\qquad(\chi\in\hH_\ee,\ \tau\in\hH_\hh),
\]
choose generators $\ww_{\chi\tau}$ of the rank-one components $\bee_\chi\Aalg(\mathcal D)\bee_\tau$ ($\chi<\tau$), put $\ww_{\chi\chi}=\bee_\chi$, and for $\chi<\rho<\tau$ define $c(\chi,\rho,\tau)\in\R$ by
\begin{gather*}
\ww_{\chi\rho}\ww_{\rho\tau}=c(\chi,\rho,\tau)\,\ww_{\chi\tau}\quad\text{if }\chi\leq\tau,\\
\ww_{\chi\rho}\ww_{\rho\tau}=0\ \text{ and }\ c(\chi,\rho,\tau)=0\quad\text{otherwise}.
\end{gather*}

The next proposition establishes an incidence condition on the datum.
\begin{proposition}\label{prop:incidence:condition}
For a \G-graded incidence datum $\mathcal D$ the following are equivalent:
\begin{enumerate}
\item\label{prop:incidence:condition:1} $\Aalg(\mathcal D)\cong\IPR$ as \R-algebras for some finite poset \PO;
\item\label{prop:incidence:condition:2} $\mathcal D$ is equivalent to the datum $\mathcal D(\Aalg)$ of a \G-grading on some $\Aalg=\IPR$, \PO finite (we then say that $\mathcal D$ is \highlight{realizable});
\item\label{prop:incidence:condition:3} \textup{(a)} whenever $\chi<\rho<\tau$ in $\PO(\mathcal D)$, one has $\chi<\tau$ and $c(\chi,\rho,\tau)\in\uR$, and \textup{(b)} the function $c$, which by \textup{(a)} is a $2$-cocycle on the poset $\PO(\mathcal D)$ with values in \uR, is a coboundary: there are units $a_{\chi\tau}$ ($\chi<\tau$) with $c(\chi,\rho,\tau)=a_{\chi\rho}a_{\rho\tau}a_{\chi\tau}^{-1}$ for all $\chi<\rho<\tau$.
\end{enumerate}
When these hold, $\leq$ is a partial order on $\PO(\mathcal D)$ and $\PO\cong\PO(\mathcal D)$. Write $\ee_{\chi\tau}$ for the elementary function of $\Inc{\PO(\mathcal D)}{\R}$. Then $\ee_{\chi\tau}\mapsto a_{\chi\tau}^{-1}\ww_{\chi\tau}$ extends to an \R-algebra isomorphism $\Inc{\PO(\mathcal D)}{\R}\to\Aalg(\mathcal D)$, which transports the grading of $\Aalg(\mathcal D)$ to a \G-grading on $\Inc{\PO(\mathcal D)}{\R}$ with \sset $\{\sum_{\chi\in\hH_\ee}\ww_{\chi\chi}\mid\ee\in\Eset\}$ and datum equivalent to $\mathcal D$.
\end{proposition}
\begin{proof}
\labelcref{prop:incidence:condition:3}$\Rightarrow$\labelcref{prop:incidence:condition:1}: By (a), $\leq$ is transitive, and it is antisymmetric because the levels $\hH_\ee$ are antichains and \Eset is a poset. Hence $\PO(\mathcal D)$ is a finite poset, and $\{\ww_{\chi\tau}\mid\chi\leq\tau\}$ is an \R-basis of $\Aalg(\mathcal D)$ (the Peirce decomposition). Associativity of $\Aalg(\mathcal D)$ gives $c(\rho,\tau,\upsilon)c(\chi,\rho,\upsilon)=c(\chi,\tau,\upsilon)c(\chi,\rho,\tau)$ for $\chi<\rho<\tau<\upsilon$, i.e.\ $c$ is a $2$-cocycle. If $c=\delta a$ as in (b), then the elements $\ww'_{\chi\tau}=a_{\chi\tau}^{-1}\ww_{\chi\tau}$ satisfy $\ww'_{\chi\rho}\ww'_{\rho\tau}=\ww'_{\chi\tau}$, so that $\ee_{\chi\tau}\mapsto\ww'_{\chi\tau}$ is an isomorphism $\Inc{\PO(\mathcal D)}{\R}\to\Aalg(\mathcal D)$. Transporting the grading, $\Inc{\PO(\mathcal D)}{\R}$ becomes \G-graded. The images $1_{\RH_\ee}=\sum_{\chi\in\hH_\ee}\ww_{\chi\chi}$ of the identities of the corners form a \sset. They are homogeneous and orthogonal, with supports partitioning $\PO(\mathcal D)$. They are primitive because a homogeneous idempotent $i\leq1_{\RH_\ee}$ lies in $1_{\RH_\ee}\Aalg(\mathcal D)1_{\RH_\ee}=\RH_\ee$, and the only nonzero homogeneous idempotent of the fine-graded algebra $\RH_\ee$ is $1$. Indeed, if $rh$ is a nonzero homogeneous idempotent, then $r^2h^2=rh$ forces $h^2=h$, for otherwise the two sides lie in distinct homogeneous components, hence $h=1$ and $r^2=r$, so $r=1$ because \R is \dR. The corner algebras, atomic summands and structure maps of this grading are those of $\mathcal D$, so its datum is equivalent to $\mathcal D$. This also proves \labelcref{prop:incidence:condition:3}$\Rightarrow$\labelcref{prop:incidence:condition:2}, and \labelcref{prop:incidence:condition:2}$\Rightarrow$\labelcref{prop:incidence:condition:1} is clear.

\labelcref{prop:incidence:condition:1}$\Rightarrow$\labelcref{prop:incidence:condition:3}: Let $\psi\colon\IPR\to\Aalg(\mathcal D)$ be an \R-algebra isomorphism. The idempotents $\psi^{-1}(\bee_\chi)$ form a complete set of pairwise orthogonal idempotents of \IPR, each primitive because $\bee_\chi\Aalg(\mathcal D)\bee_\chi=\R\bee_\chi$ and \R is indecomposable. By \Cref{rem:diagonalization} we may conjugate $\psi$ by an inner automorphism of \IPR and assume that they are diagonal. A diagonal primitive idempotent of \IPR is some $\ee_{pp}$, so $\psi(\ee_{p(\chi)p(\chi)})=\bee_\chi$ for a bijection $p\colon\PO(\mathcal D)\to\PO$. Then $\psi(\ee_{p(\chi)p(\tau)})\in\bee_\chi\Aalg(\mathcal D)\bee_\tau$, so $p(\chi)\leq p(\tau)$ forces $\bee_\chi\Aalg(\mathcal D)\bee_\tau\neq0$, that is $\chi\leq\tau$. Conversely $\psi^{-1}$ carries $\bee_\chi\Aalg(\mathcal D)\bee_\tau$ into $\ee_{p(\chi)p(\chi)}\IPR\,\ee_{p(\tau)p(\tau)}$, which is $\R\ee_{p(\chi)p(\tau)}$ and vanishes unless $p(\chi)\leq p(\tau)$, so $\chi\leq\tau$ forces $p(\chi)\leq p(\tau)$. It follows that $p$ is an isomorphism of relations and $\leq$ is a partial order. Moreover $\psi(\ee_{p(\chi)p(\tau)})=a_{\chi\tau}^{-1}\ww_{\chi\tau}$ with $a_{\chi\tau}\in\uR$ (as $\psi$ is bijective), and applying $\psi$ to $\ee_{pq}\ee_{qr}=\ee_{pr}$ gives $c(\chi,\rho,\tau)=a_{\chi\rho}a_{\rho\tau}a_{\chi\tau}^{-1}$. In particular $c(\chi,\rho,\tau)$ is a unit and $\chi<\tau$ whenever $\chi<\rho<\tau$.
\end{proof}

\begin{remark}\label{rem:coefficient:systems}
The proposition also yields a concrete description of all \G-gradings on a fixed $\Aalg=\IPR$ with prescribed partial invariants. Fix a \sset structure, that is, a partition of \PO into antichains $\PO_\ee$ ($\ee\in\Eset$) with bijections $\PO_\ee\cong\hH_\ee$, and type sets $\Theta_{\ee\ff}$ with $\dot\bigcup_{\theta\in\Theta_{\ee\ff}}E_\theta=B_{\ee\ff}=\{(\chi_p,\chi_q)\mid p\in\PO_\ee,\ q\in\PO_\ff,\ p<q\}$. Every grading on \Aalg with these invariants has, for each type $\theta$, a homogeneous generator of the form
\[
m_\theta=\sum_{(\chi_p,\chi_q)\in E_\theta}a_{pq}\,\ee_{pq}\qquad(a_{pq}\in\uR),
\]
of degree $g_{C_\theta}$, unique up to a unit factor, and it is determined by the \highlight{coefficient system} $a=(a_{pq})_{p<q}$: the homogeneous components of $\mM_{\ee\ff}$ are $\R\,q\cdot m_\theta=\R\sum_{\xi\in E_\theta}\xi(q)a_\xi\ee_\xi$. Conversely, a coefficient system defines a grading if and only if the products $m_\theta\,(k\cdot m_{\theta'})$ ($\theta\in\Theta_{\ee\ff}$, $\theta'\in\Theta_{\ff\hh}$, $k\in\H_\ff$) are homogeneous of degree $g_{C_\theta}kg_{C_{\theta'}}$ in $\mM_{\ee\hh}$, that is,
\[
\sum_{\rho}\rho(k)\,a_{\chi\rho}a_{\rho\tau}\in\R\cdot\xi(q)a_{\chi\tau}\quad\text{with a factor independent of }\xi=(\chi,\tau)\in E_{\theta''},
\]
where $\rho$ runs over the ``bridges'' with $(\chi,\rho)\in E_\theta$, $(\rho,\tau)\in E_{\theta'}$, and $q\cdot g_{C_{\theta''}}=g_{C_\theta}kg_{C_{\theta'}}$ (the sum must vanish if no such $q$ exists). Two coefficient systems define isomorphic gradings if and only if they differ by a translation of the labels and by a \emph{multiplicative} unit function $u$ ($u_{pq}u_{qr}=u_{pr}$), cf.\ \Cref{thm:graded:iso:structure}\labelcref{thm:graded:iso:structure:3}, up to unit factors on each $E_\theta$. In cohomological terms, the ratios $r(\chi,\rho,\tau)=a_{\chi\rho}a_{\rho\tau}a_{\chi\tau}^{-1}$ form a $2$-coboundary on \PO, the isomorphism class of the grading depends on $r$ only, and the structure constants of $\mathcal D(\Aalg)$ are the Fourier coefficients of $r$ along the bridges. The constants $c$ of \Cref{prop:incidence:condition} are the values of $r$ in the normalisation $\ww_{\chi\tau}=\bee_\chi m_\theta\bee_\tau$.
\end{remark}

The incidence condition can be written directly in terms of the structure constants.

\begin{proposition}\label{prop:incidence:explicit}
Let $\mathcal D$ be a datum and normalise the generators of \Cref{prop:incidence:condition} as $\ww_{\chi\rho}=\bee_\chi m_\theta\bee_\rho$ for $(\chi,\rho)\in E_\theta$, $\theta\in\Theta_{\ee\ff}$. Let $\chi<\rho<\tau$ in $\PO(\mathcal D)$, with $(\chi,\rho)\in E_\theta$, $\theta\in\Theta_{\ee\ff}$, and $(\rho,\tau)\in E_{\theta'}$, $\theta'\in\Theta_{\ff\hh}$. If $\xi=(\chi,\tau)\in E_{\theta''}$ for some $\theta''\in\Theta_{\ee\hh}$, then, in the notation of \Cref{cor:structure:constants},
\[
c(\chi,\rho,\tau)=\frac{1}{|\K|}\sum_{k\in\K}\rho(k)^{-1}\,\xi(q_k)\,c_k(\theta,\theta';\theta''),
\]
Here $c_k(\theta,\theta';\theta'')$ is the scalar of \Cref{cor:structure:constants} attached to the given $k$, which vanishes unless $\H t_k\L=C''$ and $\lambda''|_\U=\nu$. If $(\chi,\tau)$ lies in no $E_{\theta''}$, then $\ww_{\chi\rho}\ww_{\rho\tau}=0$. Consequently, $\mathcal D$ is realizable by a graded incidence algebra if and only if
\begin{enumerate}
\item for all $\chi<\rho<\tau$, the pair $(\chi,\tau)$ lies in $B_{\ee\hh}$ and the displayed sum is a unit of \R, and
\item the $2$-cocycle $c$ so defined is a coboundary on $\PO(\mathcal D)$.
\end{enumerate}
\end{proposition}
\begin{proof}
Since $\bee_\rho=|\K|^{-1}\sum_{k\in\K}\rho(k)^{-1}k$ and $\bee_\rho^2=\bee_\rho$,
\[
\ww_{\chi\rho}\ww_{\rho\tau}=\bee_\chi m_\theta\,\bee_\rho\,m_{\theta'}\bee_\tau=\frac1{|\K|}\sum_{k\in\K}\rho(k)^{-1}\,\bee_\chi\,\mu_{\ee\ff\hh}(m_\theta\otimes k\cdot m_{\theta'})\,\bee_\tau .
\]
Now $m_\theta\otimes k\cdot m_{\theta'}=\sum_{\lambda''}\bee_{\lambda''}(m_\theta\otimes k\cdot m_{\theta'})$ (sum over $\hat{\Sset}_{t_k}(\H,\L)$), and $\mu_{\ee\ff\hh}$ maps the $\lambda''$-term to $c_k(\theta,\theta';\theta'')\,q_k\cdot m_{\theta''}$ if $\theta''=(\H t_k\L,\lambda'')\in\Theta_{\ee\hh}$ and to $0$ otherwise. Now $\bee_\chi(q_k\cdot m_{\theta''})\bee_\tau=\xi(q_k)\bee_\chi m_{\theta''}\bee_\tau$, which is $\xi(q_k)\ww_{\chi\tau}$ if $\xi\in E_{\theta''}$ and $0$ otherwise (\Cref{prop:atomic:bimod:decomposition}). Summing up gives the formula, and the last statement is \Cref{prop:incidence:condition}.
\end{proof}

By \Cref{thm:classification,prop:incidence:explicit}, the classification of \G-gradings with prescribed partial invariants $(\Eset,\preccurlyeq,\H,\Theta)$ reduces to a question about the structure constants $c(\theta,\theta',\bar k;\theta'')$. One has to determine those which satisfy the polynomial equations expressing the associativity of $\Aalg(\mathcal D)$ and the incidence condition, modulo the translations and the torus $\prod_\theta\uR$ of \Cref{lem:datum:equivalence}. Over an algebraically closed field of characteristic zero this problem has only finitely many solutions.

\begin{theorem}\label{thm:rigidity}
Let \F be an algebraically closed field of characteristic zero, \Aalg a finite-dimensional \F-algebra, \G a group and $X\subseteq\G$ a finite subset. Then there are only finitely many isomorphism classes of \G-gradings on \Aalg whose support is contained in $X$.
\end{theorem}
\begin{proof}
Since $\dim\Aalg<\infty$, there are finitely many dimension vectors $(d_g)_{g\in X}$ with $\sum_gd_g=\dim\Aalg$, and we may fix one. Let $\mathrm{Gr}=\prod_{g\in X}\mathrm{Gr}(d_g,\Aalg)$ be the corresponding product of Grassmannians and let $V\subseteq\mathrm{Gr}$ be the set of tuples $(\Aalg^g)_{g\in X}$ such that $\Aalg=\bigoplus_g\Aalg^g$ and $\Aalg^g\Aalg^t\subseteq\Aalg^{gt}$ for all $g,t\in X$, where $\Aalg^{gt}=0$ if $gt\notin X$. The first condition is open and the second is closed, so $V$ is a locally closed subvariety of $\mathrm{Gr}$, and its points are the \G-gradings on \Aalg with support in $X$ and dimension vector $(d_g)$. The algebraic group $\Aut(\Aalg)$ acts on $V$, and two gradings are isomorphic if and only if they lie in the same orbit. Since $V$ is quasi-compact and distinct orbits are disjoint, it suffices to show that every orbit is open in $V$.

Let $x=(\Aalg^g)\in V$. The tangent space of $\mathrm{Gr}$ at $x$ is $\prod_g\Hom_\F(\Aalg^g,\Aalg/\Aalg^g)$, which we identify with the space of \emph{off-diagonal} endomorphisms $N$ of \Aalg, that is, those with $N(\Aalg^g)\subseteq\bigoplus_{t\neq g}\Aalg^t$ for all $g$. The tangent vector $N$ corresponds to the $\F[\varepsilon]$-point $\bigl((1+\varepsilon N)\Aalg^g\bigr)_g$ of $\mathrm{Gr}$ ($\varepsilon^2=0$). This point lies in $V$ if and only if, for all $a\in\Aalg^g$ and $b\in\Aalg^t$, $(a+\varepsilon N(a))(b+\varepsilon N(b))$ lies in $(1+\varepsilon N)(\Aalg^{gt}\otimes\F[\varepsilon])$, i.e.\ if and only if
\[
N(a)b+aN(b)-N(ab)\in\Aalg^{gt}\qquad(a\in\Aalg^g,\ b\in\Aalg^t).
\]
Thus $T_xV$ is the space of off-diagonal $N$ satisfying this condition. But for an off-diagonal $N$ the $\Aalg^{gt}$-component of $N(a)b+aN(b)-N(ab)$ vanishes automatically: $N(a)\in\bigoplus_{s\neq g}\Aalg^s$, so $N(a)b\in\bigoplus_{s\neq g}\Aalg^{st}$ has no component in $\Aalg^{gt}$, and likewise for $aN(b)$. Nor has $N(ab)$, because $ab\in\Aalg^{gt}$. Therefore every $N\in T_xV$ is a derivation of \Aalg. On the other hand $\car\F=0$, so the algebraic group $\Aut(\Aalg)$ is smooth with Lie algebra $\operatorname{Der}(\Aalg)$. The orbit map $\Aut(\Aalg)\to V$, $\varphi\mapsto\varphi\cdot x$, has differential $\operatorname{Der}(\Aalg)\to T_xV$, $D\mapsto$ (off-diagonal part of $D$). By what we have just shown, its image is all of $T_xV$. Since this image is contained in the tangent space at $x$ of the orbit $O=\Aut(\Aalg)\cdot x$, we get $T_xO=T_xV$.

Now $O$ is a locally closed subvariety of $V$, smooth because it is a finite union of translates of the orbit $O^\circ$ of the identity component $\Aut(\Aalg)^\circ$, which is a homogeneous space. It follows that $\dim O=\dim T_xO=\dim T_xV\geq\dim_xV\geq\dim O$, so that $\dim T_xV=\dim_xV$: $V$ is smooth at $x$, hence irreducible in a neighbourhood of $x$, and $O^\circ$ (irreducible, locally closed, of dimension $\dim_xV$) contains an open neighbourhood of $x$ in $V$. As $x\in O$ was arbitrary and the same holds at every point of $O$, the orbit $O$ is open in $V$.
\end{proof}

\begin{remark}\label{rem:rigidity:char}
Characteristic zero is used only once, to know that $\Aut(\Aalg)$ is smooth with Lie algebra $\operatorname{Der}(\Aalg)$. The rest of the argument is characteristic-free. The theorem therefore holds over any algebraically closed field for which $\Aut(\Aalg)$ is smooth. In positive characteristic $\operatorname{Der}(\Aalg)$ is only the Lie algebra of the automorphism group scheme, which may be non-reduced, and the computation of $T_xV$ above does not by itself bound the number of orbits.
\end{remark}

\begin{corollary}\label{cor:finite:classes}
Let \F be an algebraically closed field of characteristic zero and \PO a finite poset. For fixed partial invariants $(\Eset,\preccurlyeq,\H,\Theta)$, there are only finitely many graded isomorphism classes of \G-gradings on $\Inc{\PO}{\F}$ with these invariants. Equivalently, the structure constants of a realizable datum with these invariants take, up to equivalence, only finitely many values. In particular there are no continuous families of pairwise non-isomorphic gradings.
\end{corollary}
\begin{proof}
The support of such a grading is contained in the finite set $X=\bigcup_\ee\H_\ee\cup\bigcup_{\ee\prec\ff}\bigcup_{\theta\in\Theta_{\ee\ff}}C_\theta$, so \Cref{thm:rigidity} applies. The reformulation in terms of data is \Cref{thm:classification}.
\end{proof}

\begin{example}\label{ex:octahedron}
Let $\G=\{1,a,b,ab\}\cong\ZZ_2\times\ZZ_2$, let \R contain $\frac12$ (so that $\R\ZZ_2\cong\R^2$), and let \PO be the poset with three levels $\PO_\ee=\{p_1,p_2\}$, $\PO_\ff=\{q_1,q_2\}$, $\PO_\hh=\{r_1,r_2\}$, each an antichain, and $p_i<q_j<r_l$ for all $i,j,l$ (the order complex of \PO is the octahedron). Put $\H=\H_\ee=\langle a\rangle$, $\K=\H_\ff=\langle ab\rangle$, $\L=\H_\hh=\langle b\rangle$, label the points of each level by the trivial and the sign character of the corresponding group, and let
\[
m=\sum_{i,j}\ee_{p_iq_j},\qquad m'=\sum_{j,l}\ee_{q_jr_l},\qquad m''=\sum_{i,l}\ee_{p_ir_l}.
\]
Since $\H\cap\K=\K\cap\L=\H\cap\L=1$, all stabilizers are trivial, each of $\H\K$, $\K\L$, $\H\L$ is all of \G, and $\Tset_{\H,\K}(\G)=\Tset_{\K,\L}(\G)=\Tset_{\H,\L}(\G)$ consists of the single type $\theta_0=(\G,1)$ with $E_{\theta_0}=\hH\times\hK$ (resp.\ $\hK\times\hL$, $\hH\times\hL$). For $g,g',g''\in\G$ define a grading $\Aalg(g,g',g'')$ on \IPR by
\[
\Aalg^t=\bigoplus_{\ee\in\Eset}\Dalg_\ee^t\ \oplus\ \bigoplus_{\substack{h\in\H,k\in\K\\ hgk=t}}\R\,hmk\ \oplus\ \bigoplus_{\substack{k\in\K,l\in\L\\ kg'l=t}}\R\,km'l\ \oplus\ \bigoplus_{\substack{h\in\H,l\in\L\\ hg''l=t}}\R\,hm''l,
\]
where $\Dalg_\ee^h=\R\sum_p\chi_p(h)\ee_{pp}$ is the fine grading of $\Dalg_\ee\cong\RH$ and $h$, $k$, $l$ act through these identifications (so that $hmk=\sum_{i,j}\chi_{p_i}(h)\chi_{q_j}(k)\ee_{p_iq_j}$, and similarly for $m',m''$). Each block is an atomic bimodule of type $\theta_0$ with homogeneous generator $m$ (resp.\ $m'$, $m''$) of degree $g$ (resp.\ $g'$, $g''$), and $\Aalg(g,g',g'')$ is a grading if and only if the products are compatible: $(hmk)(k'm'l)=h\,m(kk')m'\,l$ and
\[
m\,(kk')\,m'=\sum_{i,l}\Bigl(\sum_j\chi_{q_j}(kk')\Bigr)\ee_{p_ir_l}=
\begin{cases}2m'',&kk'=1,\\ 0,&kk'=ab,\end{cases}
\]
so that $\Aalg^{hgk}\Aalg^{k^{-1}g'l}=\R\,hm''l\subseteq\Aalg^{hg''l}$ is required, i.e.\ $gg'=g''$. It follows that $\Aalg_1=\Aalg(1,1,1)$ and $\Aalg_2=\Aalg(b,b,1)$ are \G-gradings on \IPR, with the same \sset $\{\ee,\ff,\hh\}$, $\ee=\ee_{p_1p_1}+\ee_{p_2p_2}$ etc., the same corner algebras, and the same type sets $\Theta_{\ee\ff}=\Theta_{\ff\hh}=\Theta_{\ee\hh}=\{\theta_0\}$ (which are moreover invariant under all translations). They are not graded isomorphic: in $\Aalg_1$, $\ee\Aalg_1^1\ff\cdot\ff\Aalg_1^1\hh=\R m\cdot\R m'=\R m''\neq0$, whereas in $\Aalg_2$, $\ee\Aalg_2^1\ff=\R\,(a\,m\,ab)$ and $\ff\Aalg_2^1\hh=\R\,(m'b)$ (as $a\cdot b\cdot ab=1=b\cdot b$), and $(a\,m\,ab)(m'b)=a\,(m\,(ab)\,m')\,b=0$. A graded isomorphism $\Aalg_1\to\Aalg_2$ would map the \sset of $\Aalg_1$ onto a \sset of $\Aalg_2$, conjugate to $\{\ee,\ff,\hh\}$ by a degree-preserving inner automorphism, and would preserve the chain $\ee\prec\ff\prec\hh$. The vanishing or not of $\ee\Aalg^1\ff\cdot\ff\Aalg^1\hh$ is therefore an invariant, and $\Aalg_1\not\cong\Aalg_2$. In the language of \Cref{def:datum}, the two data differ only in their structure maps: $\mu(m\otimes m')=2m''$ in the first case, while in the second $\mu(m_1\otimes m_1')=0$ and $\mu(m_1\otimes(ab)m_1')=2\,(a,b)\cdot m''$ for the degree-$1$ generators $m_1=a\,m\,ab$, $m'_1=m'b$.

If \R contains a square root $i$ of $-1$ (and $\frac12$), there is a third class. Take the coefficient system with $a_{p_1q_1}=a_{p_1q_2}=1$, $a_{p_2q_1}=-i$, $a_{p_2q_2}=i$, $a_{q_jr_l}=\frac12(1\pm i)$ according to $\chi_{q_j}(ab)\chi_{r_l}(b)=\pm1$, and $a_{p_ir_l}=1$ except $a_{p_2r_2}=-1$. It defines a grading $\Aalg_3$ with all generators of degree $1$ and the same invariants, in which both $\ee\Aalg_3^1\ff\cdot\ff\Aalg_3^1\hh$ and $\ee\Aalg_3^1\ff\cdot\ff\Aalg_3^{ab}\hh$ are nonzero. The three gradings are pairwise non-isomorphic, since the pair of products $\bigl(\ee\Aalg^1\ff\cdot\ff\Aalg^1\hh,\ \ee\Aalg^1\ff\cdot\ff\Aalg^{ab}\hh\bigr)$ is $(\neq0,0)$ for $\Aalg_1$, $(0,\neq0)$ for $\Aalg_2$ and $(\neq0,\neq0)$ for $\Aalg_3$. Both products are invariants of the graded isomorphism class, because a graded isomorphism carries the \sset to a conjugate one and fixes each $\ee,\ff,\hh$, the groups $\H_\ee,\H_\ff,\H_\hh$ being distinct. These are all the gradings with these invariants up to isomorphism. In terms of \Cref{rem:coefficient:systems}, the ratios $r(\chi,\rho,\tau)$ are here forced to have the form $\frac12(c_1+c_{ab}\,\chi(a)\rho(ab)\tau(b))$, the requirement that $r$ be a coboundary on the $2$-sphere is $(c_1-c_{ab})^4=(c_1+c_{ab})^4$, and $c_{ab}/c_1\in\{0,\infty,\pm i\}$, the last two values being identified by translations.
\end{example}

\begin{remark}\label{rem:classification:comments}
(1) A \G-graded incidence algebra of a finite poset over \R is determined up to graded isomorphism by its datum: the poset \Eset of a \sset, the split abelian groups $\H_\ee$, the type sets $\Theta_{\ee\ff}$ and the structure constants $c(\theta,\theta',\bar k;\theta'')$, the last two up to translations and rescaling (\Cref{thm:classification,lem:datum:equivalence,cor:structure:constants}). A datum arises from an incidence algebra if and only if $\Aalg(\mathcal D)$ is associative and satisfies the incidence condition, which is a finite system of polynomial equations in the structure constants (\Cref{prop:incidence:condition,prop:incidence:explicit}). For fixed $(\Eset,\preccurlyeq,\H,\Theta)$ this system has only finitely many solutions up to equivalence when \R is an algebraically closed field of characteristic zero (\Cref{cor:finite:classes}). For an arbitrary indecomposable \R the same equations, with coefficients in the subring generated by the roots of unity involved, describe the realizable data. Which of the solutions are defined over \R, and which of them are inequivalent over \R, is an arithmetic question. The third class of \Cref{ex:octahedron}, which exists precisely when $i\in\R$, illustrates this.

(2) The incidence condition contains two ingredients of different nature: a combinatorial one, the transitivity in \labelcref{prop:incidence:condition:3}\textup{(a)}, which in terms of types says that the bridges between two types must land in types of $\Theta_{\ee\hh}$ with nonvanishing structure constants, and a cohomological one, \labelcref{prop:incidence:condition:3}\textup{(b)}. The latter is automatic when $H^2(\PO(\mathcal D);\uR)=0$, e.g.\ when the order complex of $\PO(\mathcal D)$ is contractible, but not in general, as the octahedron shows.

(3) The example shows that the type data are insufficient even for abelian \G and $\R=\CC$. The description of \cite[Section 4]{Santulo-Souza-Yasumura}, in which types reduce to cosets of $\H_\ee\H_\ff$ and characters of $\H_\ee\cap\H_\ff$, must therefore be supplemented by the structure constants. When \G is abelian, \R is an algebraically closed field of characteristic zero and the subgroup $\G_0\leq\G$ generated by the support is finite, the grading corresponds to an action of the finite dual group $\widehat{\G_0}$, and \Cref{thm:rigidity} reduces to the rigidity of actions of finite groups. The support itself is only a subset of \G, so it is the subgroup it generates that acts here. The argument given above does not use this correspondence and applies to arbitrary \G.
\end{remark}

%
%
%
\section*{Acknowledgments}
This project began some ten years ago, in unpublished work of the second author on gradings of incidence algebras over commutative rings. It was taken up again, and brought to the present form, after the first author joined it. The present work draws ideas from and was greatly inspired by the excellent works \cite{Santulo-Souza-Yasumura} and \cite{Miller-Spiegel}. We thank Prof. F. Y. Yasumura for the stimulating conversations we had about this topic.

\addcontentsline{toc}{chapter}{References}
%
%

\clearpage \thispagestyle{empty}
\end{document}